\documentclass{amsart}

\usepackage[usenames,dvipsnames]{xcolor}
\usepackage[utf8]{inputenc}
\usepackage[a4paper,inner=3cm,outer=3cm,top=3cm,bottom=3cm]{geometry}
\usepackage[colorlinks,citecolor=blue,urlcolor=blue,linkcolor=blue,linktocpage=true]{hyperref}
\usepackage[foot]{amsaddr}

\usepackage[]{amsmath}
\usepackage{amssymb,amsthm,amsbsy,latexsym,amsxtra}
\usepackage{stmaryrd}
\usepackage[cal=pxtx,scr=boondoxo]{mathalfa}
\makeatletter
\def\set@curr@file#1{%
  \begingroup
    \escapechar\m@ne
    \xdef\@curr@file{\expandafter\string\csname #1\endcsname}%
  \endgroup
}
\def\quote@name#1{"\quote@@name#1\@gobble""}
\def\quote@@name#1"{#1\quote@@name}
\def\unquote@name#1{\quote@@name#1\@gobble"}
\makeatother
\usepackage{graphicx}

\usepackage{todonotes}
\usepackage{appendix}
\usepackage{url}

\usepackage[capitalise]{cleveref}
\crefname{equation}{}{}
\Crefname{equation}{Eqn.~}{Eqns.~}
\crefname{enumi}{}{}
\Crefname{enumi}{}{}
\creflabelformat{enumi}{(#2#1#3)}

\usepackage{bbm,bm}
\usepackage{comment}
\usepackage{mathtools}
\usepackage{placeins}
\usepackage[labelfont={sl}]{caption}
\usepackage[labelfont={normalfont}]{subcaption}

\usepackage[normalem]{ulem}

\newtheorem{theorem}{Theorem}[section]
\newtheorem{lemma}[theorem]{Lemma}
\newtheorem{proposition}[theorem]{Proposition}
\newtheorem{corollary}[theorem]{Corollary}

\newtheorem{assumption}[theorem]{Assumption}
\newtheorem{remark}[theorem]{Remark}
\newtheorem{claim}[theorem]{Claim}
\newtheorem{algo}[theorem]{Algorithm}
\newcommand{\proofends}{\qed}

\usepackage{enumerate}
\newenvironment{enumeratei}{\begin{enumerate}[\upshape (i)]}{\end{enumerate}}

\newenvironment{enumeraten}{\begin{enumerate}[\upshape (1)]}{\end{enumerate}}
\newenvironment{enumerateA}{\begin{enumerate}[\upshape (A)]}{\end{enumerate}}

\usepackage{algpseudocode}
\algblockdefx[stepenvironment]{startstep}{endstep}
    [2][Unknown,Unknown]{{\tt #1: #2}}
    [1][Unknown]{{\tt end #1}}

\numberwithin{equation}{section}

\renewcommand{\P}{\mathbb{P}}
\newcommand{\E}{\mathbb{E}}
\newcommand{\Var}{\mathrm{Var}}

\newcommand{\cond}{\mid}
\newcommand{\bcond}{\;\big\vert\;}
\newcommand{\countin}{\,\vert\,}

\newcommand{\Exp}{\mathrm{Exp}}

\newcommand{\Unif}{\mathrm{Unif}}
\newcommand{\Bernoulli}{\mathrm{Bernoulli}}

\newcommand{\bal}{\begin{aligned}}
\newcommand{\eal}{\end{aligned}}
\newcommand*{\beq}{\begin{equation}}
\newcommand*{\eeq}{\end{equation}}
\newcommand*{\sss}{\scriptscriptstyle}

\newcommand{\wih}{\widehat}
\newcommand{\wit}{\widetilde}
\def \toinp    {\buildrel {\P}\over{\longrightarrow}}
\def \toindis  {\buildrel \textit{d}\over{\longrightarrow}}

\def \eqindis  {\buildrel \textit{d}\over{=}}

\DeclarePairedDelimiter\abs{\lvert}{\rvert}

\DeclareMathOperator{\supp}{supp}

\DeclareMathOperator{\symmdiff}{\triangle}

\newcommand{\bmatch}{\omega}
\newcommand{\bmatchset}{\Omega}
\newcommand{\unimatch}{\wih{\omega}}
\newcommand{\comrole}{\mapsfrom}

\newcommand{\actives}{A}
\newcommand{\sleeping}{S}
\newcommand{\sleepingdeg}{\CV}
\newcommand{\waiting}{W}
\newcommand{\living}{L}

\newcommand{\vertices}{\SV}
\newcommand{\edges}{\SE}

\newcommand{\comp}{\SC}
\newcommand{\leftpar}{\SV^\msl}
\newcommand{\rightpar}{\SV^\msr}

\newcommand{\graphs}{\SG}
\newcommand{\comgraphs}{\SH}

\newcommand{\rg}{\CH}
\newcommand{\distr}{\sim}
\newcommand{\isom}{\simeq}

\newcommand{\groot}{o}

\newcommand{\nth}{\nn}
\newcommand{\mth}{\mm}
\newcommand{\ith}{^\text{th}}

\newcommand{\e}{{\mathrm e}}

\newcommand{\RIG}{\mathrm{RIG}}
\newcommand{\RIGC}{\mathrm{RIGC}}
\newcommand{\BCM}{\mathrm{BCM}}
\newcommand{\rCM}{\mathrm{CM}}
\newcommand{\BP}{\mathrm{BP}}

\newcommand{\rCP}{\mathrm{CP}}

\newcommand{\com}{\mathrm{Com}}
\newcommand{\comvect}{\mathbf{Com}}
\newcommand{\type}{\mathit{Type}}

\newcommand{\ldeg}{\msl\text{-}\mathrm{deg}}
\newcommand{\rdeg}{\msr\text{-}\mathrm{deg}}
\newcommand{\bdeg}{\msb\text{-}\mathrm{deg}}
\newcommand{\cdeg}{\msc\text{-}\mathrm{deg}}
\newcommand{\pdeg}{\msp\text{-}\mathrm{deg}}

\newcommand{\hitting}{\tau}
\newcommand{\hittingstd}{\hitting_{\mathrm{stn}}}
\newcommand{\hittingskip}{\hitting_{\mathrm{skip}}}
\newcommand{\livingstd}{\living^\msl_{\mathrm{stn}}}
\newcommand{\indexes}{\CI}
\newcommand{\indexesskip}{\CJ}
\newcommand{\maxindex}{j_{\mathrm{max}}}
\newcommand{\usedup}{\SU}
\newcommand{\alarmproc}{Z}
\newcommand{\partialsum}{\Sigma}

\newcommand{\goodevent}{\CE}

\newcommand{\dmax}[1]{d_{\mathrm{max}}^{#1}}
\newcommand{\dmin}[1]{d_{\mathrm{min}}^{#1}}

\newcommand{\nn}[1]{{#1}^{\sss (n)}}
\newcommand{\mm}[1]{{#1}^{\sss (m)}}

\newcommand{\step}[1]{{\tt step#1}}

\newcommand{\steps}{\MSS}

\newcommand{\inv}[1]{{#1}^{\sss (-1)}}

\newcommand{\CE}{\mathcal{E}}

\newcommand{\CH}{\mathcal{H}}
\newcommand{\CI}{\mathcal{I}}
\newcommand{\CJ}{\mathcal{J}}

\newcommand{\CT}{\mathcal{T}}

\newcommand{\CV}{\mathcal{V}}

\newcommand{\he}{\mathscr{h}}

\newcommand{\msl}{\mathscr{l}}
\newcommand{\msr}{\mathscr{r}}
\newcommand{\msb}{\mathscr{b}}
\newcommand{\msp}{\mathscr{p}}

\newcommand{\msc}{\mathscr{c}}

\newcommand{\SV}{\mathscr{V}}
\newcommand{\SA}{\mathscr{A}}

\newcommand{\SC}{\mathscr{C}}

\newcommand{\SE}{\mathscr{E}}

\newcommand{\SG}{\mathscr{G}}
\newcommand{\SH}{\mathscr{H}}

\newcommand{\SP}{\mathscr{P}}

\newcommand{\MSS}{\mathscr{S}}

\newcommand{\SU}{\mathscr{U}}
\newcommand{\SZ}{\mathscr{Z}}

\newcommand{\rd}{\mathrm{d}}

\newcommand{\bitd}{\boldsymbol{d}}

\newcommand{\bitp}{\boldsymbol{p}}
\newcommand{\bitq}{\boldsymbol{q}}

\newcommand{\R}{\mathbb{R}}
\newcommand{\N}{\mathbb{N}}
\newcommand{\Z}{\mathbb{Z}}

\newcommand{\ind}{\mathbbm{1}}

\newcommand*{\ro}{\varrho}

\newcommand{\eps}{\varepsilon}

\usepackage[backend=bibtex,style=numeric,useprefix=false,sorting=nyt, doi=false,url=false,maxnames=10,abbreviate=false,giveninits=true]{biblatex}

\begin{document}

\title[Phase transition in the $\RIGC$]{Phase transition in random intersection graphs\\ with communities} 
\author[R.v.d.Hofstad]{Remco van der Hofstad}
\author[J.Komj\'athy]{J\'ulia Komj\'athy}
\author[V.Vadon]{Vikt\'oria Vadon}
\date{\today}
\subjclass[2010]{Primary: 60C05, 05C80, 90B15, 82B43.}
\keywords{Random networks, community structure, overlapping communities, random intersection graphs, bipartite configuration model, phase transition, percolation}
\address{Department of Mathematics and
	    Computer Science, Eindhoven University of Technology, P.O.\ Box 513,
	    5600 MB Eindhoven, The Netherlands.}
\email{r.w.v.d.hofstad@tue.nl, j.komjathy@tue.nl, v.vadon@tue.nl}

\begin{abstract}
The `random intersection graph with communities' models networks with communities, assuming an underlying bipartite structure of groups and individuals. Each group has its own internal structure described by a (small) graph, while groups may overlap. The group memberships are generated by a bipartite configuration model. The model generalizes the classical random intersection graph model that is included as the special case where each community is a complete graph (or clique).

The `random intersection graph with communities' is analytically tractable. We prove a phase transition in the size of the largest connected component based on the choice of model parameters. Further, we prove that percolation on our model produces a graph within the same family, and that percolation also undergoes a phase transition. Our proofs rely on the connection to the bipartite configuration model, however, with the arbitrary structure of the groups, it is not completely straightforward to translate results on the group structure into results on the graph. Our related results on the bipartite configuration model are not only instrumental to the study of the random intersection graph with communities, but are also of independent interest, and shed light on interesting differences from the unipartite case.
\end{abstract}

\maketitle

\section{Introduction}
\label{s:introduction}
Real-world networks often exhibit a higher amount of clustering (also called transitivity, \cite[Chapter 7.9, 11]{Newm10}), i.e., a larger amount of triangles, than expected by pure chance \cite{WatStro98}. 
A possible explanation for this phenomenon are \emph{communities}: (small) subgraphs that are (significantly) denser than the network average. Such a community structure has been observed in several real-life networks \cite{GirNew02}, such as the Internet, collaboration networks and social networks. In particular, many networks are well-explained \cite{GuiLat04,GuiLat06} by an underlying (possibly hidden) structure of individuals and groups that the individuals are part of. 
This is the kind of model we focus on in this paper. 
The prime example is collaboration networks, such as the Internet movie database IMDb or the ArXiv, where the `individuals' are the actors and actresses or the authors, and the `groups' are the movies or articles they collaborate in. We can also model a social network in a similar fashion, giving `groups' the interpretation of families, common interests, workplaces or cities. 
We take most of our terminology from social networks, however the model we present below is more widely applicable, for any network that builds on a group structure.

Due to the complexity of real-world networks, they are often modeled using \emph{random graphs} \cite{Bol01,Dur07,JanLucRuc00}. Models are built to mimic some empirically observed properties of the network that we consider as defining features, such as degree structure, clustering, small-world property, etc. Then one may study further properties and processes of interest, such as network evolution and information or epidemic spreading processes, on the model to make predictions for real-life networks.

The traditional random graph model for networks with a group structure is the random intersection graph ($\RIG$) \cite{BalGer2009unifRIG,Bloz10,Bloz13,Bloz17,BloGodJawKurRyb15, DeiKet09,GodeJaw2003,New2003configRIG,Ryb2011unifRIG,Sing96, Yag16,YagMak12}. In this model, the underlying group structure mentioned above is represented by a bipartite graph, where the two partitions correspond to the individuals (people) and the groups (or attributes), and an edge represents a group membership. 
The group memberships, that is, connections in this bipartite graph, are random. Individuals are then connected in the intersection graph when they are together in a group, i.e., share at least one group as neighbor in the bipartite graph. As a result, the members of a group form a complete subgraph, though in some variations of the model this complete graph is thinned \cite{KarLeuwLes18,New2003configRIG}. 
An other direction of research focuses on networks where the building blocks are communities that have an arbitrary internal structure, however no overlap \cite{BallSirlTrap09,BallSirlTrap10, SteHofLeuw2016epidemic,SteHofLeuw2015}. 
In \cite{vdHKomVad18locArxiv}, we have introduced a generalization of the RIG model that also incorporates an \emph{arbitrary} internal structure for each group, which we call the random intersection graph with communities ($\RIGC$). The $\RIGC$ thus combines the efforts to model networks with overlapping communities, as well as using arbitrary communities as building blocks. The $\RIGC$ is applicable to real-life network data, while we also derive rigorous analytic results. 

In \cite{vdHKomVad18locArxiv}, we have studied ``local'' properties of the $\RIGC$ model, such as local weak convergence (convergence of subgraph counts), degrees, local clustering coefficient, and the overlapping structure of communities. In this paper, we instead focus on ``global'' properties, in particular the \emph{giant component problem} and \emph{percolation} (defined shortly), as well as the relation between ``local'' and ``global'' properties. 

The \emph{giant component problem} studies whether there exists a component containing a \emph{linear} proportion of vertices. It has garnered quite some interest in the random graph literature since the seminal work of Erd\H{o}s and R\'{e}nyi \cite{ErdosRenyi60evolution}, and has been studied on several other models (e.g.\ Chung-Lu model \cite{ChungLu02giant,ChungLu06giant} and configuration model \cite{BolRio15,JanLuc2009,MolRee95,MolRee98}, see also the survey by Spencer \cite{Spencer10giant} and the references therein). We prove that as we vary the parameters of the $\RIGC$, the size of the largest component undergoes a \emph{phase transition}: either all components are sublinear as the network size grows, or there exists a unique \emph{giant component} containing a constant fraction of the vertices, while the rest of the components are sublinear. When a giant component exists, we are able to further characterize it. 
To solve the giant component problem for the $\RIGC$, we also solve it for a bipartite version of the configuration model ($\BCM$), which is of independent interest.

\emph{Percolation} means keeping each edge of a graph independently with a fixed probability, and it has applications related to epidemiology and attack vulnerability of networks. We give a detailed introduction to percolation in \cref{sss:perc_intro}. We prove that percolation on the $\RIGC$ model can be represented as an $\RIGC$ with different parameters; we thus find that percolation on the $\RIGC$ undergoes a similar phase transition, and identify the critical point implicitly.

\smallskip
\paragraph*{\textbf{Main contribution and novelty of methodology}}
The main innovation of this paper is the methodology required to solve the giant component problem for the bipartite configuration model. We make a non-trivial adaptation of the continuous-time exploration algorithm of the configuration model, originating from Janson and Luczak \cite{JanLuc2009}, to the bipartite case. The analysis of the bipartite case involves studying a death process where jumps occasionally happen with \emph{infinite} rate.

Showing that percolation on the $\RIGC$ stays within the family of models suggests that the $\RIGC$ is also a natural way to study  percolation on the classical random intersection graph.

Finally, we combine the results on the giant component of the $\RIGC$ from the continuous-time exploration with the results of local weak convergence from \cite{vdHKomVad18locArxiv}, to obtain further properties of the giant component.

\smallskip
\paragraph*{\textbf{Organization of the paper}} 
In \cref{ss:RIGC_def,ss:asmp}, we introduce the model and our assumptions. In \cref{ss:results_giant_RIGC,ss:results_BCM,ss:results_perc}, we state our results on the largest component of the $\RIGC$, the largest component of the $\BCM$ and percolation on the $\RIGC$, respectively. In \cref{s:proof_giant}, we prove the results on the phase transition in the $\BCM$ as well as the $\RIGC$ using the adapted continuous-time exploration, and in particular, in \cref{ss:livingproof} we analyze the death process with occasional instantaneous jumps. In \cref{s:perc_proof}, we prove the phase transition of (bond) percolation in the $\RIGC$ as a consequence of the largest component phase transition. Finally, in \cref{s:local_global} we study further properties of the largest component of the $\RIGC$. 

\smallskip
\paragraph*{\textbf{Notational conventions}}
To study the asymptotic behavior of various quantities, we will consider a sequence of graphs and consequently, a sequence of input parameters, both indexed by $n\in\N$. We note that $n$ does not necessarily mean the size or any other parameter of the graph, and to keep the notation light, we often omit indicating the dependence on $n$, as long as it does not cause confusion. Throughout this paper, we denote the set of positive integers as $\Z^+$ and the set of non-negative integers as $\N$. The notions $\toinp$ and $\toindis$ stand for convergence in probability and convergence in distribution (weak convergence), respectively. We write $X\eqindis Y$ to mean that the random variables $X$ and $Y$ have the same distribution. For an $\N$-valued random variable $X$ such that $\E[X]<\infty$, we define its \emph{size-biased} distribution $X^\star$ and its shifted version $\wit X := X^\star-1$ with the following probability mass functions (pmf): for all $k\in\N$,
\beq\label{eq:def_sizebiasing} \P(X^\star = k) = k \,\P(X=k) / \E[X], 
\qquad \P(\wit X = k) = \P(X^\star - 1 = k). \eeq
For a random variable $X$ taking values in $\N$, we denote its  probability generating function by $G_X:[0,1]\to[0,1]$, given by
\beq\label{eq:def_genfunc} G_X(z):= \E\bigl[z^X\bigr] = \sum_{k=0}^\infty \P(X=k) z^k. \eeq
Note that $G_{X^\star}(z) = z G'_{X}(z)/\E[X]$ and $G_{\wit{X}}(z) = G_{X^\star}(z) / z = G'_X(z)/\E[X]$. We say that a sequence of events $(A_n)_{n\in\N}$ occurs with high probability (whp), when $\lim_{n\to\infty}\P(A_n)=1$. For two (possibly) random sequences $(X_n)_{n\in\N}$ and $(Y_n)_{n\in\N}$, we say that $X_n = o_{\sss\P}(Y_n)$ if $X_n/Y_n \toinp 0$ as $n\to\infty$. We say that the collection $(X_a)_{a\in\SA}$, indexed by a set $\SA$, of $\R$-valued random variables is uniformly integrable (UI) if $\lim_{K\to\infty} \sup_{a\in\SA} \E\bigl[ X_a \ind_{\abs{X_a}\geq K} \bigr] = 0$. 
We denote $[n]:=\{1,2,\ldots,n\}$ and the indicator of an event $A$ by $\ind_A$. For a graph $G$, we denote its vertex set by $\vertices(G)$ and its edge set by $\edges(G)$.

\smallskip
\paragraph*{\textbf{List of abbreviations}}
In this paper, we make use of the following abbreviations: 
without loss of generality (wlog), 
with respect to (wrt), 
independent and identically distributed (iid), 
uniformly at random (uar), 
probability mass function (pmf), 
with high probability (whp), 
left-hand side (lhs), 
right-hand side (rhs), 
uniformly integrable (UI), 
random intersection graph with communities (RIGC), 
bipartite configuration model (BCM), 
branching process (BP), 
random intersection graph (RIG), 
and configuration model (CM).

\FloatBarrier
\section{Model and results}
\label{s:RIGC}
In this section, we introduce our model and study some of its global properties.

\subsection{Definition of the random intersection graph with communities}
\label{ss:RIGC_def}
This section is a more concise transcription of the model definition from the companion paper \cite{vdHKomVad18locArxiv} on the local properties of the $\RIGC$ model. For more details on the construction, we direct the reader to that paper. Given the parameters $\bitd^\msl$, $\comvect$, whose meaning is explained below, we construct the random graph $\RIGC(\bitd^\msl,\comvect)$ in two steps. First, we construct the bipartite matching that determines the group memberships, then we show how to obtain the $\RIGC$ based on the group memberships.

\smallskip
\paragraph*{\textbf{The parameters}}
We start with a bipartite graph comprising of individuals and groups. We call the set of individuals the left-hand side (lhs) partition $\leftpar=[N_n]$, where $N_n\to\infty$ is the number of individuals. We may refer to an individual $v\in\leftpar$ as an $\msl$-vertex. We call its number of group memberships its $\msl$-degree (left-degree), and denote it by $d_v^\msl=\ldeg(v)$. The parameter $\bitd^\msl=(d_v^\msl)_{v\in\leftpar}$ is the vector of $\msl$-degrees. Without loss of generality (wlog), we assume that $\bitd^\msl\geq1$ (element-wise).

Analogously, we call the set of groups the right-hand side (rhs) partition $\rightpar=[M_n]$, where also the number of groups $M_n$ satisfies $M_n\to\infty$, and we may refer to groups as $\msr$-vertices. Each $\msr$-vertex $a$ is associated with a community graph $\com_a$, with properties explained below, and $\comvect=(\com_a)_{a\in\rightpar}$ is the vector of community graphs. Let $\comgraphs$ be the set of possible community graphs: simple, finite, connected graphs $H$, and we label each $H$ arbitrarily by $[\abs{H}]$, so that any two community graphs that are isomorphic are also labeled identically. We assume that each assigned community graph $\com_a\in\comgraphs$ satisfies $\abs{\com_a}\geq 1$. We call $\abs{\com_a}$ the $\msr$-degree (right-degree) of group $a$ and denote it by $d_a^\msr=\rdeg(a)$. We collect all $\msr$-degrees in the vector $\bitd^\msr := (d_a^\msr)_{a\in\rightpar}$.

\smallskip
\paragraph*{\textbf{Community memberships}}
In the bipartite graph of group memberships, the $\msl$- and $\msr$-degrees act as degrees. We refer to them together as $\msb$-degrees (bipartite degrees). To ensure the existence of a bipartite graph with these given degrees, we assume and denote
\beq\label{eq:def_halfedges} \he_n := \sum_{v\in\leftpar} d_v^\msl = \sum_{a\in\rightpar} d_a^\msr. \eeq
With the given $\msb$-degrees, we construct the group memberships according to a bipartite matching, described as follows. 
Denote the disjoint union of all vertices in community graphs by $\vertices(\comvect)$, which we call the set of community vertices (community roles). To each $\msr$-vertex $a$, we assign $\rdeg(a)$ $\msr$-half-edges, labeled by $(a,l)_{l\in[\rdeg(a)]}$. We think of $(a,l)$ as the membership token corresponding to the community vertex $j\in\vertices(\comvect)$ in $\com_a$ with label $l$, which gives a natural correspondence between the set of $\msr$-half-edges and $\vertices(\comvect)$. 
To each $\msl$-vertex $v$, we assign $\ldeg(v)$ $\msl$-half-edges, labeled by $(v,i)_{i\in[\ldeg(v)]}$. In contrast to $\msr$-half-edges, the $\msl$-half-edges incident to the same $\msl$-vertex are interchangeable as membership tokens.

Denote by $\bmatchset_n$ the set of all bijections between the set of $\msl$-half-edges $(v,i)_{i\in[\ldeg(v)],v\in\leftpar}$ to the set of $\msr$-half-edges $(a,l)_{l\in[\rdeg(a)],a\in\rightpar}$. (Equivalently, bijections between the set of $\msl$-half-edges and $\vertices(\comvect)$). Let $\bmatch_n\sim\Unif[\bmatchset_n]$ denote a bipartite matching (or bipartite configuration) chosen uniformly at random (uar).\footnote{Note that, by re-indexing the half-edges, we can think of $\bmatch_n$ as a permutation of $[\he_n]$, thus $\abs{\bmatchset_n} = \he_n!$.} 
If the $\msl$-half-edge $(v,i)$ and the $\msr$-half-edge $(a,l)$ are paired by $\bmatch_n$, this intuitively means one of the community roles taken by $v$ is the community vertex $j\in\vertices(\comvect)$ in $\com_a$ with label $l$, and we denote the indicator of this event by $\ind_{\{v \comrole j\}} = \ind_{\{v \comrole j\}}(\bmatch_n)$. Note that (almost surely)
\beq \sum_{v\in\leftpar} \ind_{\{v \comrole j\}} = 1,
\quad \sum_{j\in\vertices(\comvect)} \ind_{\{v \comrole j\}} = \ldeg(v). \eeq
The community memberships of the $\RIGC$ model (see \cref{fig:matching}) are determined by the uniform(ly random) bipartite matching $\bmatch_n$ in the above fashion. Before we define the $\RIGC$ graph based on the community memberships, we make some observations about $\bmatch_n$.

\begin{figure}[hbt]
\centering
\begin{subfigure}[b]{\textwidth}
	\centering
	\includegraphics[width=0.7\textwidth]{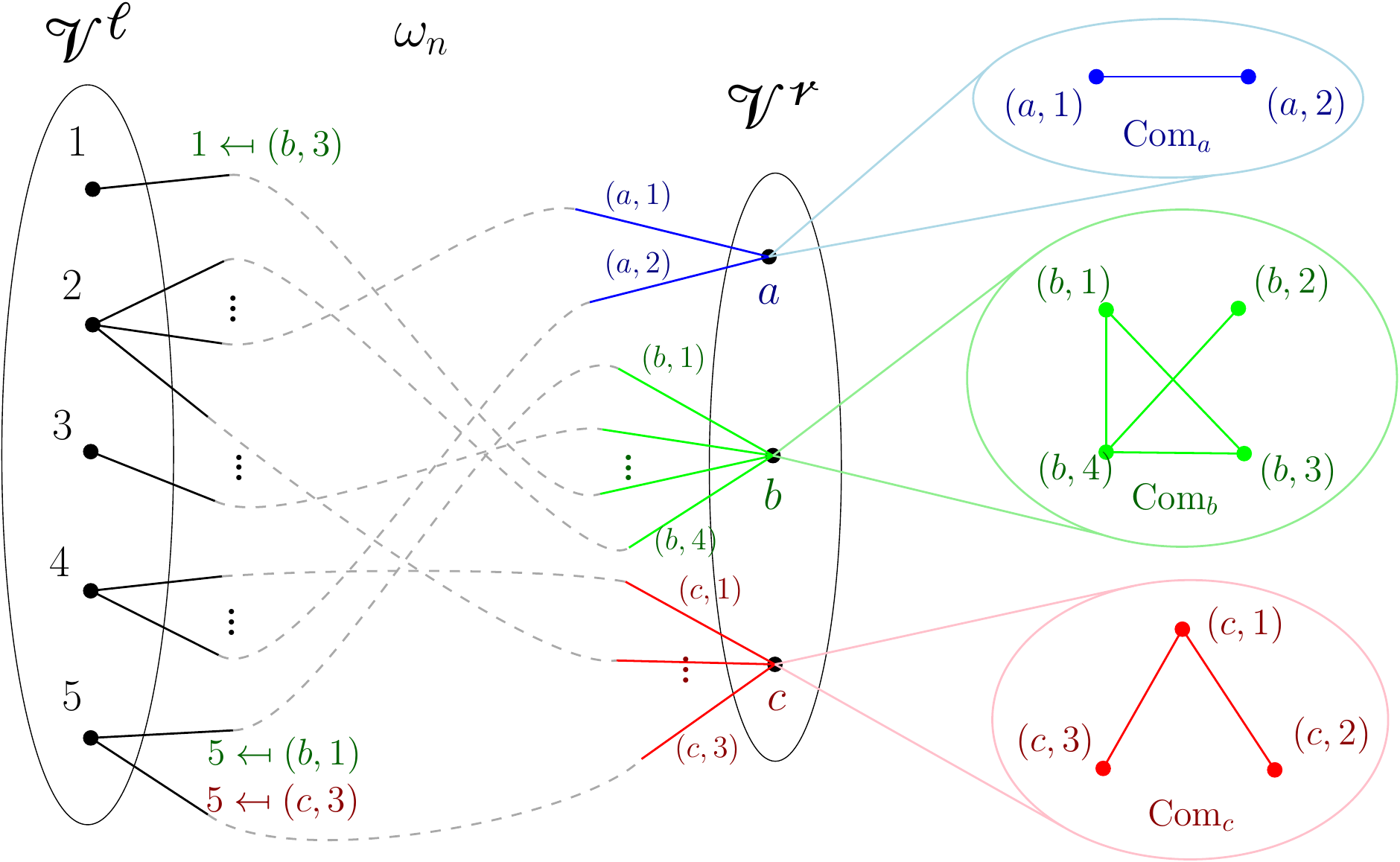}
	\caption{\emph{Community roles assigned by the bipartite matching}}
	\label{fig:matching}
\end{subfigure}\\
\begin{subfigure}[b]{0.75\textwidth}
	\centering
	\begin{subfigure}[b]{0.32\textwidth}
  	\includegraphics[width=0.95\textwidth]{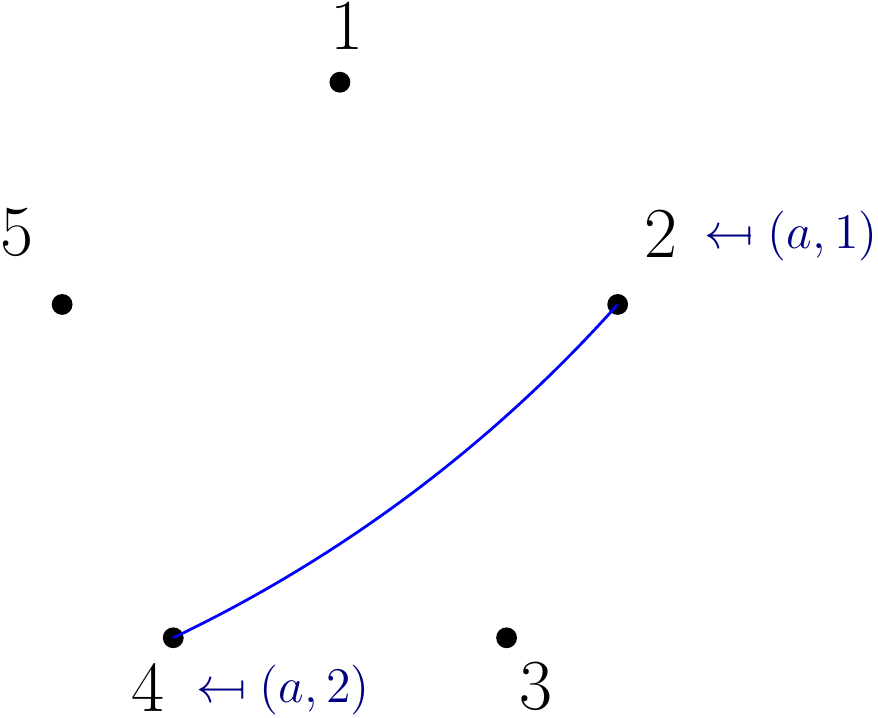}
  \end{subfigure}
 	\begin{subfigure}[b]{0.32\textwidth}
  	\includegraphics[width=0.95\textwidth]{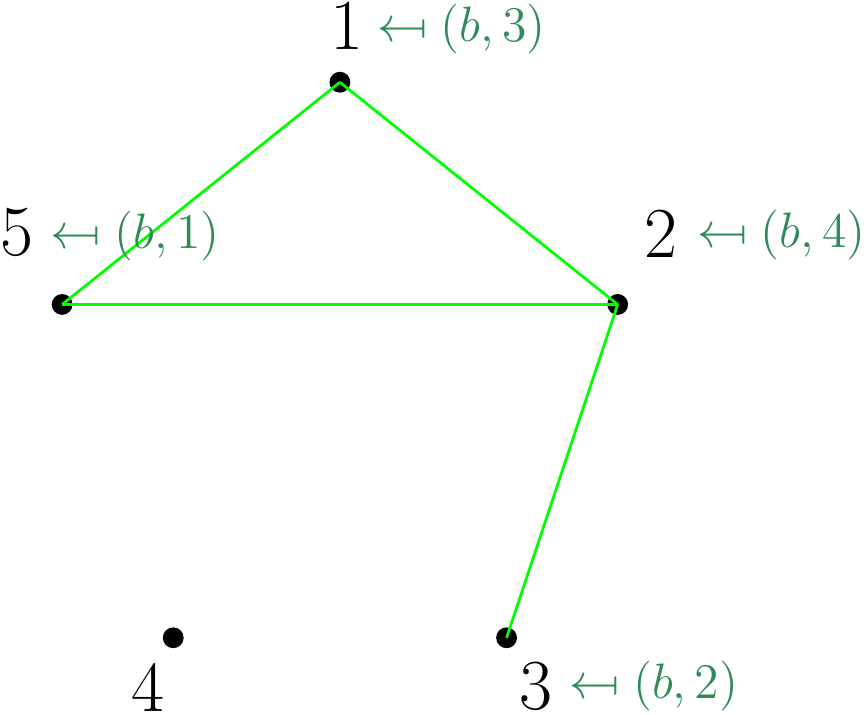}
  \end{subfigure}
 	\begin{subfigure}[b]{0.32\textwidth}
  	\includegraphics[width=0.95\textwidth]{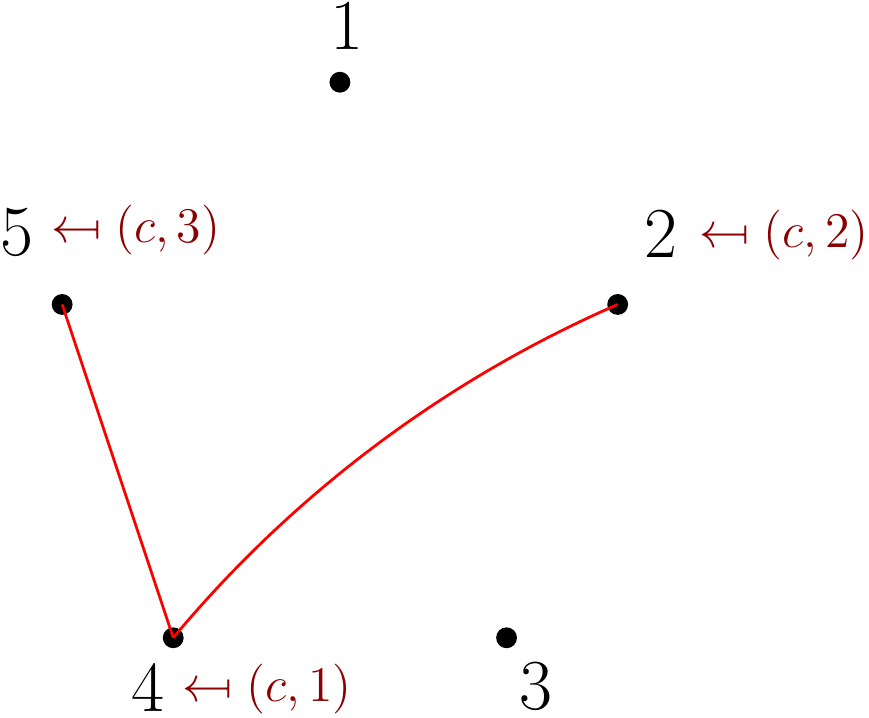}
  \end{subfigure}
  \caption{\emph{The projection of each community}\\ Each edge between community roles is copied to the corresponding \\individuals (which are assigned the community roles forming the edge)}
\end{subfigure}
\begin{subfigure}[b]{0.24\textwidth}
  \centering
 	\includegraphics[width=0.95\textwidth]{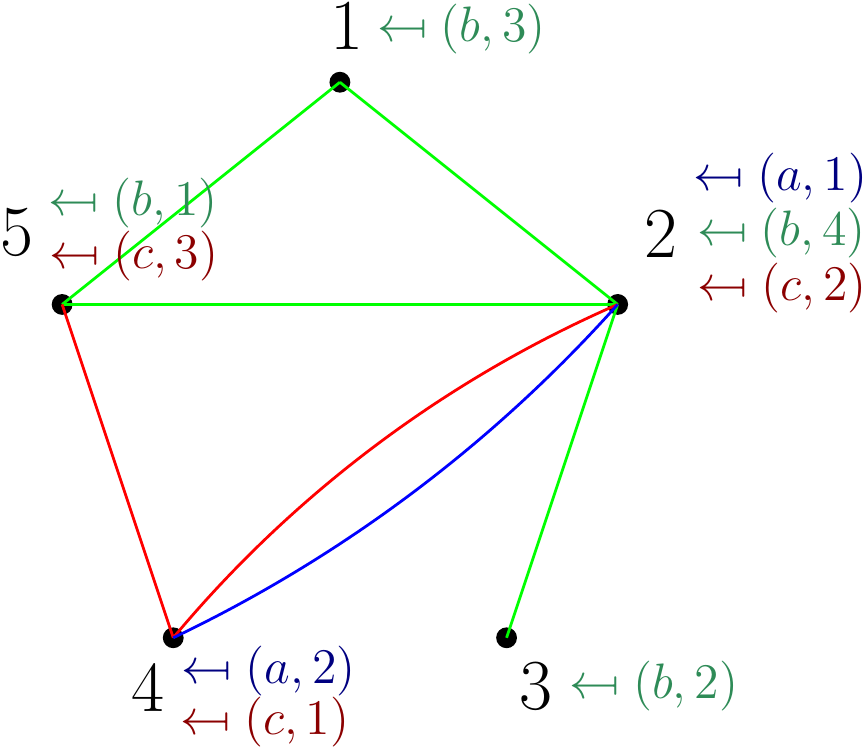}
  \caption{\emph{The resulting $\RIGC$}\\ The multigraph combining the communities}
\end{subfigure}
\caption{Construction of the RIGC graph.}
\label{fig:embedding}
\end{figure}

\begin{remark}[Algorithmic pairing, {\cite[Remark 2.1]{vdHKomVad18locArxiv}}]\label{rem:pairingalgo}
The uniform bipartite matching $\bmatch_n$ can be produced sequentially, as follows. In each step, we pick an arbitrary unpaired half-edge, and match it to a uniform unpaired half-edge in the opposite partition. As the choices are arbitrary, they may even depend on the past of the pairing process.
\end{remark}

\begin{remark}[The underlying BCM, {\cite[Definition 2.2]{vdHKomVad18locArxiv}}]\label{rem:bcm}
We may view the half-edges as tokens to form edges (rather than group membership tokens as above), as usual in the configuration model. Then the bipartite matching $\bmatch_n$ also determines a bipartite (multi)graph, defined as follows. If the $\msl$-half-edge $(v,i)$ and the $\msr$-half-edge $(a,l)$ are matched, we replace them by an edge labeled by $(i,l)$ between $v\in\leftpar$ and $a\in\rightpar$. 
Note that with the edge labels, we can recover the matched half-edges, thus this multigraph provides an equivalent representation of $\bmatch_n$. Thus we refer to this labeled bipartite (multi)graph as the \emph{underlying bipartite configuration model}.

Deleting the edge labels introduced above, we obtain the (classical) bipartite configuration model $\BCM(\bitd^\msl,\bitd^\msr)$ with degree sequences $\bitd^\msl$ and $\bitd^\msr$.
\end{remark}

\smallskip
\paragraph*{\textbf{Community-projection}}
In the following, we complete the construction of the $\RIGC$, based on the community structure defined by $\bmatch_n$. This step is entirely deterministic and we can think of it as an operator $\SP$ from $\bmatchset_n$ into the set of \emph{multigraphs}.\footnote{For a discussion on why multigraphs arise and why we have chosen to work with them, see the companion paper \cite[Section 2.4, p.\ 13]{vdHKomVad18locArxiv}.} In the following, we define the $\RIGC$ by its (random, $\bmatch_n$-dependent) edge multiplicities $(X(v,w))_{v,w\in\leftpar}$.

Recall that when an individual $v\in\leftpar$ takes on community role $j\in\vertices(\comvect)$, we write $v \comrole j$. Let us denote the disjoint union of edges in all community graphs by $\edges(\comvect)$ that we refer to as the set of community edges. Intuitively, we construct the $\RIGC$ by copying each community edge $(j_1,j_2)$ to the individuals $(v,w)$ such that $v \comrole j_1$ and $w \comrole j_2$. Formally, we define the multiplicity of the edge $(v,w)$ as
\beq\label{eq:def_edge_multiplicities} X(v,w) = X(v,w;\bmatch_n) := \sum_{(j_1,j_2)\in\edges(\comvect)} \ind_{\{v \comrole j_1, w \comrole j_2\} \cup \{v \comrole j_2, w \comrole j_1\}}. \eeq
Then, the (random) degree of $v\in\leftpar$ in the resulting $\RIGC$, sometimes referred to as projected degree ($\msp$-degree) for clarification, is given by
\beq\label{eq:def_deg} d_v^\msp = \pdeg(v) := 
X(v,v) + \sum_{w\in\leftpar} X(v,w) 
= 2 X(v,v) + \sum_{w\in\leftpar,w\neq v} X(v,w). \eeq

\FloatBarrier
\subsection{Assumptions on the parameters}
\label{ss:asmp}
In this section, we introduce some notation and state the assumptions necessary for our results in \cref{ss:results_giant_RIGC,ss:results_perc}. We remark that the notation and assumptions are identical to those introduced in \cite[Section 2.2]{vdHKomVad18locArxiv}.

\smallskip
\paragraph*{\textbf{The bipartite degrees}}
We define uniformly chosen $\msl$- and $\msr$-vertices
\beq\label{eq:def_unifvertex} V_n^\msl \sim \Unif[\leftpar],
\qquad V_n^\msr \sim \Unif[\rightpar], \eeq
and denote their degrees by
\beq\label{eq:def_ldeg_rdeg} D_n^\msl := \ldeg\bigl(V_n^\msl\bigr),
\qquad D_n^\msr := \rdeg(V_n^\msr). \eeq
Further, denote the sets of $\msl$-vertices and $\msr$-vertices with degree $k$, respectively, by
\beq\label{eq:def_Vk}
\vertices_k^\msl := \{ v\in\leftpar:\, \ldeg(v)=k \},
\qquad
\vertices_k^\msr := \{ a\in\rightpar:\, \rdeg(a)=k \}.
\eeq
Then the following probability mass functions (pmf), for $k\in\Z^+$,
\beq\label{eq:def_pq_pmf} \nth p_k := \abs{\leftpar_k}/N_n, 
\qquad \nth q_k := \abs{\rightpar_k}/M_n, \eeq
describe the distribution of the variables $D_n^\msl$ and $D_n^\msr$, as well as the empirical distribution of $\bitd^\msl$ and $\bitd^\msr$, respectively. We collect the pmfs in the (infinite-dimensional) probability vectors $\nth\bitp = (\nth p_k)_{k\in\Z^+}$, $\nth\bitq = (\nth q_k)_{k\in\Z^+}$.

\smallskip
\paragraph*{\textbf{The empirical community distribution}}
Recall that the possible community graphs are $\comgraphs$, the set of simple, finite, connected graphs $H$ with an arbitrary but fixed labeling by $[\abs{H}]$. For a fixed $H\in\comgraphs$, define
\beq\label{eq:def_VH} \vertices_H^\msr := \{a\in\rightpar:\, \com_a=H\}. \eeq
We introduce the pmf
\beq\label{eq:def_mu} \nth \mu_H := \abs{\vertices_H^\msr}/M_n, \qquad \nth {\bm\mu} = (\nth \mu_H)_{H\in\comgraphs}. \eeq
Thus $\nth{\bm\mu}$ describes the empirical pmf of $\comvect$, as well as the pmf of $\com_{V_n^\msr}$, with $V_n^\msr\sim\Unif[\rightpar]$. Define
\beq\label{eq:def_Hk} \comgraphs_k := \bigl\{ H\in\comgraphs:\, \abs{H}=k \bigr\}. \eeq
Note that $\nth q_k = \sum_{H\in\comgraphs_k} \nth\mu_H$ (with $\nn q_k$ from \eqref{eq:def_pq_pmf}).

\smallskip
\paragraph*{\textbf{The community degrees}}
For $a\in\rightpar$ and $j\in\vertices(\com_a)$, define the community degree ($\msc$-degree) of $j$, denoted by $d_j^\msc$, as the number of connections $j$ has within $\com_a$. Recall that $\vertices(\comvect)$ denotes the disjoint union of vertices in all communities, and let $J_n\distr\Unif[\vertices(\comvect)]$ be a uniformly chosen community vertex, among all $\he_n$ possibilities.\footnote{This is equivalent to choosing $\com_a$ in a \emph{size-biased fashion}, i.e., picking $\com_a$ with probability $\abs{\com_a}/\he_n$, then picking $J_n \cond \com_a \distr \Unif[\vertices(\com_a)]$.} 

Introduce the random variable $D_n^\msc := d_{J_n}^\msc$. We define the pmf that describes $D_n^\msc$ as well as the empirical distribution of the collection $(d_{j}^\msc)_{j\in\vertices(\comvect)}$, by
\beq\label{eq:def_rho} \nth \ro_{k} := \frac1{\he_n} \sum_{j\in\vertices(\comvect)} \ind_{ \{ d_{j}^\msc = k \} },
\qquad \nth{\bm\ro} := \bigl( \nth \ro_{k} \bigr)_{k\in\Z^+}. \eeq 

\paragraph*{\textbf{Assumptions}}
\label{sss:assumptions}
Recall \cref{eq:def_ldeg_rdeg,eq:def_pq_pmf,eq:def_mu}. We can now summarize our assumptions on the model parameters:

\begin{assumption}[{\cite[Assumption 2.3]{vdHKomVad18locArxiv}}]
\label{asmp:convergence} The conditions for the empirical distributions are summarized as follows:
\begin{enumerateA}
\item\label{cond:limit_ldeg} There exists a random variable $D^\msl$ with pmf $\bitp$ s.t.\ $\nth\bitp \to \bitp$ pointwise as $n\to\infty$, i.e.,
\beq D_n^\msl \toindis D^\msl. \eeq
\item\label{cond:mean_ldeg} $\E[D^\msl]$ is finite, and as $n\to\infty$,
\beq \E[D_n^\msl] \to \E[D^\msl]. \eeq
\item\label{cond:limit_com} There exists a probability mass function $\bm\mu$ on $\comgraphs$ such that $\nth{\bm\mu}\to\bm\mu$ pointwise as $n\to\infty$.
\begin{enumeraten}
\item\label{cond:limit_rdeg} Consequently, by $\nth q_k=\sum_{H\in\comgraphs_k} \nth\mu_H$, with the finite set $\comgraphs_k$ from \eqref{eq:def_Hk}, there exists a random variable $D^\msr$ with pmf $\bitq$ such that $\nth\bitq\to\bitq$ pointwise as $n\to\infty$, or equivalently,
\beq D_n^\msr\toindis D^\msr. \eeq
\end{enumeraten}
\item\label{cond:mean_rdeg} $\E[D^\msr]$ is finite, and as $n\to\infty$,
\beq \E[D_n^\msr] \to \E[D^\msr]. \eeq
\end{enumerateA}
\end{assumption}

\begin{remark}[Consequences of \cref{asmp:convergence}, {\cite[Remark 2.4]{vdHKomVad18locArxiv}}]
\label{rem:asmp_consequences}
\normalfont
We note the following:
\begin{enumeratei}
\item\label{cond:partition_ratio} Recall $N_n = \abs{\leftpar}$ and $M_n = \abs{\rightpar}$. By its definition in \eqref{eq:def_halfedges}, $\he_n = N_n \E[D_n^\msl] = M_n \E[D_n^\msr]$. By \cref{asmp:convergence} (\ref{cond:mean_ldeg},\ref{cond:mean_rdeg}), 
\beq\label{eq:mnratio} M_n / N_n = \E[D_n^\msl] / \E[D_n^\msr] \to \E[D^\msl] / \E[D^\msr] =: \gamma \in \R^+. \eeq
\item\label{cond:limit_cdeg} Since $\nth{\bm\ro}$ (see \eqref{eq:def_rho}) can be obtained from $\nth{\bm\mu}$, \cref{asmp:convergence} \cref{cond:limit_com} also implies that there exists a random variable $D^\msc$ with pmf $\bm\ro$ such that $\nth{\bm\ro}\to\bm\ro$ pointwise as $n\to\infty$, or equivalently, $D_n^\msc \toindis D^\msc$.
\item\label{cond:dmax} \Cref{asmp:convergence} (\ref{cond:limit_ldeg},\ref{cond:mean_ldeg}) imply that $\dmax{\msl} := \max_{v\in\leftpar} d_v^\msl = o(\he_n)$, and similarly, conditions (\ref{cond:limit_rdeg}-\ref{cond:mean_rdeg}) imply that $\dmax{\msr} := \max_{a\in\rightpar} d_a^\msr = o(\he_n)$.
\end{enumeratei}
\end{remark}

\begin{remark}[Random parameters, {\cite[Remark 2.5]{vdHKomVad18locArxiv}}]\label{rem:randomparam} 
The results in \cref{ss:results_giant_RIGC} below remain valid when the sequence of parameters $(\bitd^\msl,\comvect)$ (resp., $(\bitd^\msl,\bitd^\msr)$) is random itself. In this case, we require that $N_n\toinp\infty$ and $M_n\toinp\infty$,\footnote{By $N_n \toinp \infty$, we mean that for all $K\in\R^+$, $\P(N_n > K) \to 1$ as $n\to\infty$.} and we replace \cref{asmp:convergence} {\rm(\ref{cond:limit_ldeg}-\ref{cond:mean_rdeg})} (resp., \cref{asmp:convergence} {\rm(\ref{cond:limit_ldeg},\ref{cond:mean_ldeg},\ref{cond:limit_rdeg},\ref{cond:mean_rdeg})}) by the conditions 
$\nth{\bitp}\toinp\bitp$ pointwise, $\E[D_n^\msl \cond \bitd^\msl] \toinp \E[D^\msl]$, $\nth{\bm\mu}\toinp\bm\mu$ pointwise (resp., $\nth{\bitq}\toinp\bitq$) and $\E[D_n^\msr \cond \bitd^\msr] \toinp \E[D^\msr]$. For a similar setting in the configuration model, see \cite[Remark 7.9]{rvHofRGCNvol1}, where this is spelled out in more detail.
\end{remark}

Note that analogously to \cref{rem:asmp_consequences} \cref{cond:partition_ratio}, 
the conditions of \cref{rem:randomparam} imply that $M_n/N_n \toinp \gamma$.

\subsection{Results on the largest component of the random intersection graph with communities}
\label{ss:results_giant_RIGC}
In this paper, we study global properties of the $\RIGC$ model; its local properties have been studied in the companion paper \cite{vdHKomVad18locArxiv}. In particular, in this section we study the largest connected component of the $\RIGC$. We prove a phase transition in the size of the largest component in terms of the model parameters, and explicitly identify the conditions under which a unique linear-sized component exists. We study further properties of this component, i.e., its degree distribution and its number of edges. Denote the largest connected component (the component containing the most $\msl$-vertices, breaking ties arbitrarily) by $\comp_1 = \nn\comp_1$, and the second largest by $\comp_2 = \nn\comp_2$. Recall \eqref{eq:def_sizebiasing}, \eqref{eq:def_genfunc}, \eqref{eq:def_ldeg_rdeg} and \eqref{eq:def_pq_pmf}.

\begin{theorem}[Size of the largest component]\label{thm:giantcomp}
Consider $\RIGC(\bitd^\msl,\comvect)$ under \cref{asmp:convergence}, and further assume that $p_2+q_2<2$. 
Then, there exists $\eta_\msl \in [0,1]$, the smallest solution of the fixed point equation
\beq\label{eq:fixedpoint} \eta_\msl = G_{\wit D^\msr}\bigl( G_{\wit D^\msl} (\eta_\msl) \bigr), \eeq
and $\xi_\msl := 1-G_{D^\msl}(\eta_\msl) \in [0,1]$ such that
\beq\label{eq:rigsize} \abs{\comp_1}/N_n \toinp \xi_\msl. \eeq
Furthermore, $\xi_\msl > 0$ exactly when
\beq\label{cond:supercrit} \E[\wit D^\msl] \E[\wit D^\msr] > 1, \eeq
which we call the \emph{supercritical case}. In this case, $\comp_1$ is unique in the sense that $\abs{\comp_2} = o_{\sss\P}(N_n)$, and we call $\comp_1$ the \emph{giant component}.
\end{theorem}

We prove \cref{thm:giantcomp} in \cref{s:proof_giant} subject to the upcoming \cref{thm:BCM_giantcomp}, and we discuss the relevance of the condition $p_2+q_2<2$ in \cref{ss:discussion}. Note that the size of the largest connected component only depends on $D^\msl$ and $D^\msr=\abs{\rg}$, where $\rg\in\comgraphs$ follows distribution $\bm\mu$; this is because the communities are connected. Consequently \cref{thm:giantcomp} applies to the classical $\RIG$, which is the special case of $\RIGC$ with complete graph communities. We continue by studying the degree distribution and the number of edges in the giant component. Naturally, these quantities depend more sensitively on $\bm\mu$ and are non-trivial: 
our results show that the degree distribution in the giant component is considerably different from the degree distribution of the whole graph (unless $\xi_\msl=1$ and the giant component contains almost all vertices). The reason for this is a size-biasing effect of the giant (see \cref{thm:BCM_giantcomp,cor:BCM_giantcomp} below). Recall \eqref{eq:def_deg} and \eqref{eq:def_Vk} and define 
\beq\label{eq:def_Dk} \vertices_d^\msp := \bigl\{ v\in\leftpar:\,\pdeg(v)=d \bigr\}. \eeq
For $H\in\comgraphs$ and $c\in\Z^+$, define $\nu(c \countin H) := \abs{\{ j\in\vertices(H):\, \cdeg(j) = c \}}$, the number of vertices in $H$ with $\msc$-degree $c$.

\begin{theorem}[Degrees in the giant]\label{thm:degrees_giant} Consider $\RIGC(\bitd^\msl,\comvect)$ under \cref{asmp:convergence}, \allowbreak additionally assuming the supercriticality condition \cref{cond:supercrit}. Define $\eta_\msr := G_{\wit D^\msl}(\eta_\msl)$, with $\eta_\msl$ from \cref{thm:giantcomp}. For $k\in\Z^+,d\in\N$, define
\beq\label{eq:def_deggiant_distr} A(k,d) := 
p_k \sum_{H_1,\ldots,H_k\in\comgraphs} \sum_{\substack{c_1,\ldots,c_k\in\N\\ c_1+\ldots+c_k=d}} 
\Bigl( 1 - \eta_\msr^{ \sum_{i=1}^k (\abs{H_i}-1)} \Bigr)
\prod_{i=1}^k \frac{\nu(c_i \countin H_i) \mu_{H_i}}{\E[D^\msr]}. \eeq
Then, the proportion of individuals that have $k$ group memberships, total degree $d$ and are in the giant component converges as $n\to\infty$: 
\beq\label{eq:deggiant_distr} \frac{ \abs{\vertices_k^\msl\cap\vertices_d^\msp\cap\comp_1} }{N_n} 
\toinp A(k,d). \eeq
\end{theorem}

The proof of \cref{thm:degrees_giant} is deferred to \cref{ss:local_global_rel}, where we also give a heuristic interpretation of $A(k,d)$. It follows (via a truncation argument) from \cref{thm:degrees_giant} and \eqref{eq:rigsize} that the empirical degree distribution in the giant converges:
\beq \frac{\abs{\vertices_d^\msp\cap\comp_1}}{\abs{\comp_1}} 
= \sum_{k\in\Z^+} \frac{ \abs{\vertices_k^\msl\cap\vertices_d^\msp\cap\comp_1} }{N_n} \frac{N_n}{\abs{\comp_1}} 
\toinp \sum_{k\in\Z^+} A(k,d) / \xi_\msl. \eeq
While the expression in \eqref{eq:def_deggiant_distr} seems quite involved, the following remark shows that in fact, it is closely related to the limiting degree distribution of the whole graph. Recall $D^\msl$ from \cref{asmp:convergence} \cref{cond:limit_ldeg}, and $D^\msc$ from \cref{rem:asmp_consequences} \cref{cond:limit_cdeg}. We recall from \cite[(2.20)]{vdHKomVad18locArxiv} that the limiting degree distribution is $D^\msp \eqindis \sum_{i=1}^{D^\msl} D_{(i)}^\msc$, 
where $(D_{(i)}^\msc)_{i\in\Z^+}$ are independent, identically distributed (iid) copies of $D^\msc$, and also independent of $D^\msl$.

\begin{remark}[{Relation of \cref{thm:degrees_giant} and \cite[Corollary 2.9]{vdHKomVad18locArxiv}}] 
We now explore the relation between the degree distribution in the giant and the degree distribution in the entire graph. First, note that for any fixed $c$, with $\varrho_{c}$ from \cref{rem:asmp_consequences} \cref{cond:limit_cdeg},
\beq\label{eq:rem_deg_0} \sum_{H\in\comgraphs} \frac{\nu(c \countin H) \mu_{H}}{\E[D^\msr]} = \varrho_{c} = \P(D^\msc = c). \eeq
Here, the denominator $\E[D^\msr]$ only serves for renormalization, since $\nn{\bm\mu}$ is a distribution on $\rightpar$ with size $M_n$, while $\nn{\bm\varrho}$ is a distribution on $\vertices(\comvect)$ with size $\he_n$. 
The factor $\bigl( 1 - \eta_\msr^{ \sum_{i=1}^k (\abs{H_i}-1)} \bigr)$ in \eqref{eq:def_deggiant_distr} heuristically corresponds to belonging to the giant, which is later justified by \eqref{eq:deg_k_right}. By \eqref{eq:rem_deg_0}, omitting the factor $\bigl( 1 - \eta_\msr^{ \sum_{i=1}^k (\abs{H_i}-1)} \bigr)$ from \eqref{eq:def_deggiant_distr}, its rhs becomes a convolution:
\beq\label{eq:deggiant_relation} \bal 
& p_k \sum_{H_1,\ldots,H_k\in\comgraphs} \sum_{\substack{c_1,\ldots,c_k\in\N\\ c_1+\ldots+c_k=d}} 
\prod_{i=1}^k \frac{\nu(c_i \countin H_i) \mu_{H_i}}{\E[D^\msr]} 
= p_k \sum_{\substack{c_1,\ldots,c_k\in\N\\ c_1+\ldots+c_k=d}} \prod_{i=1}^k \varrho_{c_i} \\
&= p_k \, \P\Bigl( \sum_{i=1}^k D_{(i)}^\msc = d \Bigr) 
= \P\bigl(D^\msl=k\bigr) \,\P\bigl( D^\msp=d \bcond D^\msl=k \bigr)
= \P\bigl( D^\msp=d, D^\msl=k \bigr),
\eal \eeq
which is the asymptotic joint distribution of $\msl$- and $\msp$-degrees in the whole graph. Indeed, combining \cite[Corollary 2.9 and (2.20)]{vdHKomVad18locArxiv} implies that 
\beq \frac{ \abs{\vertices_k^\msl\cap\vertices_d^\msp} }{N_n} \toinp \P\bigl(D^\msp=d,D^\msl=k\bigr). \eeq
\end{remark}

Next, we state our result regarding the number of edges in the giant component. Recall $D_n^\msc$ with pmf $\nn{\bm\varrho}$ from \eqref{eq:def_rho}, 
$\eta_\msr$ from \cref{thm:degrees_giant} and $\gamma$ from \eqref{eq:mnratio}.

\begin{theorem}[Edges in the giant]\label{thm:edgesRIGC}
Consider $\RIGC(\bitd^\msl,\comvect)$ under \cref{asmp:convergence} and the supercriticality condition \cref{cond:supercrit}, and additionally assume that
\beq\label{cond:UI_Dc} \text{$(D_n^\msc)_{n\in\N}$ is uniformly integrable (UI),} \eeq
which in particular implies that $\E[D_n^\msc] \to \E[D^\msc] < \infty$ as $n\to\infty$. 
Let $\rg$ denote a random graph with pmf $\bm\mu$. Then the number of edges in the giant component $\comp_1$ of the $\RIGC$ satisfies, as $n\to\infty$,
\beq\label{eq:rigedges} \frac{\abs{\edges( \comp_1)}}{N_n} 
\toinp \gamma\, \E\bigl[\abs{\edges(\rg)}\bigl(1-\eta_\msr^{\abs{\rg}}\bigr)\bigr]. \eeq
\end{theorem}

We sketch the proof shortly below, which requires the following lemma: 
\begin{lemma}[Uniform integrability]
\label{lem:UI_equiv}
The following statements are equivalent:
\begin{enumeratei}
\item\label{UI:comdeg} $(D_n^\msc)_{n\in\N}$ is UI;
\item\label{UI:deg} $(D_n^\msp)_{n\in\N}$ is UI;
\item\label{UI:edges} $(\abs{\edges(\rg_n)})_{n\in\N}$ is UI.
\end{enumeratei}
\end{lemma}
The proof of \cref{lem:UI_equiv} is rather technical and tedious and we postpone it to \cref{apxs:UI}, but we discuss the relevance of \cref{lem:UI_equiv} and the condition \eqref{cond:UI_Dc} now. 
The statement in \cref{lem:UI_equiv} \eqref{UI:edges}, or equivalently, the condition \cref{cond:UI_Dc}, is the necessary and sufficient condition for the rhs of \eqref{eq:rigedges} to be finite. 
Since this condition takes the community structure into account, it is more refined than moment conditions on the community size $D_n^\msr$. 
By our assumption that community graphs are simple and connected, $\abs{H} - 1 \leq \abs{\edges(H)} \leq \abs{H}(\abs{H} - 1) / 2$, which implies $\E[D_n^\msr] - 1 \leq \E\bigl[ \abs{\edges(\rg_n)} \bigr] \leq \E[D_n^\msr(D_n^\msr-1)]/2$. 
Thus the condition \cref{lem:UI_equiv} \eqref{UI:edges} is weaker than $\E\bigl[ (D_n^\msr)^2 \bigr] \to \E\bigl[ (D^\msr)^2 \bigr] < \infty$ (which is sufficient, but not necessary), but stronger than $\E[D_n^\msr] \to \E[D^\msr] < \infty$, that is \cref{asmp:convergence} \cref{cond:mean_rdeg}. 
In the general case under \cref{asmp:convergence}, it is still possible that $\E\bigl[ \abs{\edges(\rg_n)} \bigr]$ diverges, which implies that $\abs{\edges(\comp_1)}/N_n$ diverges and \cref{thm:edgesRIGC} does not hold. 

\begin{proof}[Sketch of proof of \cref{thm:edgesRIGC} subject to \cref{lem:UI_equiv}]
\Cref{thm:edgesRIGC} follows from Theorem \ref{thm:degrees_giant} under the extra condition of uniform integrability in \cref{cond:UI_Dc}. 
Denote a uniform $\msl$-vertex $V_n^\msl \distr \Unif[\leftpar]$ and its (projected) degree (see \eqref{eq:def_deg}) $D_n^\msp = \pdeg(V_n^\msl)$. Let $\rg_n$ denote a random graph with pmf $\nn{\bm\mu}$ from \eqref{eq:def_mu}. 
In \cite[Corollary 2.9]{vdHKomVad18locArxiv}, the distributional limit of $D_n^\msp$ is established as $D^\msp \eqindis \sum_{i=1}^{D^\msl} D_{(i)}^\msc$, where $D_{(i)}^\msc$ are iid copies of $D^\msc$ that are independent of $D^\msl$. 
Under condition \eqref{cond:UI_Dc}, \cref{lem:UI_equiv} \eqref{UI:deg} ensures that $\E[D^\msp] < \infty$, which implies that the average degree in the giant is also finite. Thus we can show that under \cref{cond:UI_Dc}, as $n\to\infty$,
\beq\label{eq:edgesconv_Akd} \frac{\abs{\edges(\comp_1)}}{N_n} 
= \frac12 \sum_{d\in\N} d \sum_{k\in\Z^+} \frac{\abs{\leftpar_k \cap \vertices^\msp_d \cap \comp_1}}{N_n} 
\toinp \frac12 \sum_{d\in\N} d \cdot \sum_{k\in\Z^+} A(k,d). \eeq
Under condition \eqref{cond:UI_Dc}, \cref{lem:UI_equiv} \eqref{UI:edges} implies that $\E\bigl[ \abs{\edges(\rg)} \bigr] < \infty$, so that the rhs of \eqref{eq:rigedges} is finite. Then, we can show that the average degree is in fact related to the number of edges in the community graphs:
\beq\label{eq:equiv_edgesformulas} \frac12 \sum_{d\in\N} d \cdot \sum_{k\in\Z^+} A(k,d) 
= \gamma \cdot \E\bigl[ \abs{\edges(\rg)} \bigl( 1 - \eta_\msr^{\abs{\rg}} \bigr) \bigr]. \eeq
We provide the details of the proof in \cref{apxs:edges_giant}. 
\end{proof}

\subsection{The largest component of the bipartite configuration model}
\label{ss:results_BCM}
In this section, we introduce our results on the largest component of the $\BCM$ (see \cref{rem:bcm}), which are of independent interest, and we further apply them to prove our results on the $\RIGC$. Denote the largest component (the component containing the largest total number of vertices, with ties broken arbitrarily) of the $\BCM_n(\bitd^\msl,\bitd^\msr)$ by $\comp_{1,\msb} = \nth\comp_{1,\msb}$, and the second largest by $\comp_{2,\msb}=\nth\comp_{2,\msb}$. Recall $\eta_\msl$, $\xi_\msl$ and $\eta_\msr$ from \cref{thm:giantcomp,thm:degrees_giant} respectively, and $\vertices_k^\msl$ from \eqref{eq:def_Vk}. Our main result on the $\BCM$ is as follows:

\begin{theorem}[The largest component of the $\BCM$]\label{thm:BCM_giantcomp}
Consider $\BCM(\bitd^\msl,\bitd^\msr)$ under \cref{asmp:convergence} {\rm(\ref{cond:limit_ldeg},\ref{cond:mean_ldeg},\ref{cond:limit_rdeg},\ref{cond:mean_rdeg})}, and further assume that $p_2+q_2<2$. Under the supercriticality condition \cref{cond:supercrit}, that we call the \emph{supercritical case} of the $\BCM$, we have that $\xi_\msl > 0$, $\eta_\msl < 1$ and $\eta_\msr < 1$. Then, as $n\to\infty$,
\begin{gather}
\label{eq:leftsize} \frac{\lvert \comp_{1,\msb}\cap\leftpar \rvert}{N_n} \toinp \xi_\msl, \\
\label{eq:deg_k} \frac{\abs{\comp_{1,\msb} \cap \vertices_k^\msl}}{N_n} \toinp p_k \bigl(1-\eta_\msl^k\bigr), \\
\label{eq:edges} \frac{\lvert \edges(\comp_{1,\msb}) \rvert}{N_n} \toinp \E[D^\msl] \bigl(1-\eta_\msl \eta_\msr \bigr).
\end{gather}
In this case, $\comp_{1,\msb}$ is unique in the sense that $\abs{\comp_{2,\msb}}/(N_n+M_n) \toinp 0$, and we refer to $\comp_{1,\msb}$ as the \emph{giant component} of the $\BCM$. When \cref{cond:supercrit} does not hold, $\abs{\comp_{1,\msb}}/(N_n+M_n) \toinp 0$.
\end{theorem}

We prove \cref{thm:BCM_giantcomp} in \cref{s:proof_giant}, and highlight the main ideas behind the proof shortly below. First, we provide some comments on and corollaries of \cref{thm:BCM_giantcomp}. 
We note that while \eqref{eq:edges} looks ``asymmetric'', since $\he_n = N_n\E[D_n^\msl]=M_n\E[D_n^\msr]$ (see \cref{rem:asmp_consequences} \cref{cond:partition_ratio}), we can rephrase it as $\lvert \edges(\comp_{1,\msb}) \rvert / M_n \toinp \E[D^\msr] (1-\eta_\msl \eta_\msr )$, as well as $\lvert \edges(\comp_{1,\msb}) \rvert / \he_n \toinp 1-\eta_\msl \eta_\msr$. In \cref{ss:discussion}, we discuss why the condition $p_2+q_2<2$ is needed. 
Recall $\vertices_k^\msr$ from \eqref{eq:def_Vk} and $\gamma$ from \cref{rem:asmp_consequences} \cref{cond:partition_ratio}.

\begin{corollary}[The rhs partition]
\label{cor:BCM_giantcomp}
Under the conditions of \cref{thm:BCM_giantcomp} and the supercriticality condition \cref{cond:supercrit}, with $\xi_\msr := 1-G_{D^\msr}(\eta_\msr) \in [0,1]$, as $n\to\infty$,
\begin{gather}
\label{eq:rightsize} \frac{\lvert \comp_{1,\msb}\cap\rightpar \rvert}{M_n} \toinp \xi_\msr, \\
\label{eq:deg_k_right} \frac{\abs{\comp_{1,\msb} \cap \vertices_k^\msr}}{M_n} \toinp q_k \bigl(1-\eta_\msr^k\bigr), \\
\label{eq:allsize} \frac{\lvert \comp_{1,\msb} \rvert }{N_n+M_n} \toinp \frac{\xi_\msl + \gamma\xi_\msr}{1+\gamma}.
\end{gather}
\end{corollary}

\begin{proof}[Proof of \cref{cor:BCM_giantcomp}]
Observe that the role of the lhs and rhs partitions, and in particular,  the role of the quantities $\eta_\msl$ and $\eta_\msr$, as well as that of $\xi_\msl$ and $\xi_\msr$ are symmetric; we formally establish this symmetry below in \cref{ss:BPapprox}. Thus, by switching left and right, (\ref{eq:rightsize}-\ref{eq:deg_k_right}) follow from (\ref{eq:leftsize}-\ref{eq:deg_k}), and combining \eqref{eq:leftsize} and \eqref{eq:rightsize} yields \eqref{eq:allsize}. Thus, 
\cref{cor:BCM_giantcomp} follows from \cref{thm:BCM_giantcomp}.
\end{proof}

\paragraph*{\textbf{Overview of the proof of \cref{thm:BCM_giantcomp}}}
The proof relies on a continuous-time exploration algorithm of the $\BCM$. This algorithm is based on the continuous-time exploration algorithm for the (traditional, unipartite) configuration model proposed in \cite{JanLuc2009}. However, the algorithm in \cite{JanLuc2009} must be modified significantly, since without modification
it only yields the average density in \eqref{eq:allsize}. We give a brief explanation of our modified algorithm and highlight the challenges in its analysis here; the details are provided in \cref{s:proof_giant}.

The new algorithm builds and explores the graph simultaneously and unveils the connected components one by one. We start the exploration of each component by picking an unexplored $\msl$-vertex. During the exploration of a component, each \emph{round} of the algorithm is a \emph{double-step}, described as follows. We first match a free half-edge incident to an $\msl$-vertex $v$ in the component that we are currently exploring, and reach an $\msr$-vertex neighbor $a$. We then match all remaining half-edges of $a$, using it as a bridge to reach second neighbors of $v$ that are again $\msl$-vertices. This constitutes one round (or double-step). By the end of each round, all unmatched half-edges belonging to the component being explored are incident to $\msl$-vertices. Thus, the component is fully explored when a round is completed and there are no more unmatched $\msl$-half-edges. 

The asymmetric roles of the left and right partitions are necessary for obtaining the more refined results on the size of the giant within each partition. 
However, it leads to a more complex analysis, as the $\msr$-vertex to explore in each round is chosen \emph{randomly}. (In particular, we find this vertex by matching the chosen $\msl$-half-edge to a uniform unmatched $\msr$-half-edge, thus we choose the $\msr$-vertex in a size-biased fashion.) Consequently the degree of the $\msr$-vertex is also random, and by matching all its remaining half-edges, we create a random number of edges in the $\BCM$ in one round. In contrast, in the original algorithm exactly one edge is created in each round.

Studying our modified algorithm, only a small portion of the analysis remains the same as for the original algorithm in \cite{JanLuc2009}. The novelty and mathematical challenge lies in studying the evolution of the number of unmatched $\msl$-half-edges. It evolves as a pure death process, where most jumps happen with rate $i$ from state $i$, however some randomly chosen jumps happen \emph{instantaneously} (with rate infinity). The reason for this is exactly the random number of edges created in one round of the algorithm: one instantaneous jump happens in each round due to adding the new $\msr$-vertex, and the ``regular'' jumps happen when we match the rest of the half-edges of this $\msr$-vertex. See \cref{ss:exploration} for more explanation. 
To analyze this death process with two types of jumps, we compare it to a well-understood ``standard'' death process, where there are no instantaneous jumps. The comparison is carried out through hitting times, allowing us to think of the effect of the instantaneous jumps as the \emph{time saved}. We are able to analyze the \emph{time saved} in an elegant way by giving it a new probabilistic interpretation in terms of a size-biased reordering of $\msr$-degrees. 

\subsubsection{Discussion and open problems}
\label{ss:discussion}
For a discussion on the RIGC model (about its applicability, overlapping structure and simplicity), see the companion paper \cite[Section 2.4]{vdHKomVad18locArxiv}. In this section, we provide a discussion on the extra condition $p_2+q_2<2$ and the use of the $\BCM$ to generate simple bipartite graphs with a given degree sequence.

\smallskip
\paragraph*{\textbf{The condition $p_2+q_2<2$}}
We briefly explain why the almost-2-regular graph $p_2+q_2=2$ is excluded. First, we show that the $\rCM$ can be obtained from the $\BCM$ as a special case, then we recall from the literature why the general results are not applicable for the almost-2-regular case of the $\rCM$.

Assume that $\vertices_2^{\msr} = \rightpar$ for all $n$, i.e., all $\msr$-vertices have degree $2$. Then each $\msr$-vertex $a$ only serves as connecting two $\msl$-vertices, say, $v$ and $w$ through the $2$-length path $(v,a,w)$. We can construct a unipartite graph on $\leftpar$ by contracting each of these $(v,a,w)$ paths into an edge $(v,w)$. We show that this unipartite graph has the distribution of the configuration model $\rCM_n(\bitd^\msl)$. For each (unipartite) matching $\wih\bmatch_n = \{(v,i_1),(w,i_2)\}$ of the $\msl$-half-edges corresponding to a unipartite graph, 
there are exactly $\abs{\rightpar}!2^{\abs{\rightpar}}$ bipartite matchings $\bmatch_n$ that are mapped into $\unimatch_n$ by the above contraction. The reason is that we can permute all $\msr$-vertices, as well as each pair of $\msr$-half-edges attached to the same $\msr$-vertex. Since $\bmatch_n$ is a uniform bipartite matching, necessarily $\unimatch_n$ is a uniform (unipartite) matching.

In \cite{JanLuc2009}, the $p_2=1$ case of the $\rCM$ is excluded for the reason that the size of the giant component is \emph{not concentrated}: it shows diverse behavior depending on the more refined asymptotics of the degree structure. In particular, if there are only degree-2 vertices, then the density of the largest component converges to a \emph{non-degenerate} distribution, rather than a constant. However, adding a sublinear proportion of degree-1 vertices makes the size of the giant component drop to sublinear. In contrast, when almost all vertices have degere $2$ and a sublinear proportion has degree $4$, the giant component constitues almost all vertices. For a more detailed discussion see \cite{rvHofRGCNvol2,JanLuc2009}.

By the contraction described above, the $p_2+q_2=2$ case of the $\BCM$ includes the ambiguous $p_2=1$ case of the $\rCM$. In particular, when $\rightpar_2=\rightpar$ for all $n$ and $p_2=1$, the $\BCM$ is equivalent to the $\rCM$ with $p_2=1$. This shows that not only the proof fails for this case, but \cref{thm:BCM_giantcomp} itself does not hold. 

\smallskip
\paragraph*{\textbf{Uniform simple bipartite graphs with given degrees}}
It is well known that the (traditional, unipartite) configuration model ($\rCM$) conditioned on being simple is a uniform simple graph with the given degree sequence. Not surprisingly, the corresponding statement is also true for the $\BCM$. We provide a brief justification below. 
Let $G$ be an arbitrary bipartite multigraph with $\msl$-degree and $\msr$-degree sequences $\bitd^\msl$ and $\bitd^\msr$, and for $v\in\leftpar, a\in\rightpar$, let $x_\msb(v,a)$ denote the multiplicity of the edge $(v,a)$ in $G$. Then the number of (bipartite) matchings $\bmatch_n$ that realize $G$ is
\beq \frac{\prod_{v\in\leftpar}d_v^\msl! \prod_{a\in\rightpar}d_a^\msr!}{\prod_{v\in\leftpar,a\in\rightpar} x_\msb(v,a)!}. \eeq
We justify the formula, as follows. The numerator arises since all half-edges attached to the same vertex are equivalent, hence permuting them leads to the same graph, but a different matching. The denominator in turn arises since all instances of a multi-edge are equivalent, and by permuting both $\msl$- and $\msr$-half-edges, the same set of pairs appears in all possible orderings. Then all simple bipartite graphs, i.e., where all the multiplicities $x_\msb(v,a)$ are $0$ or $1$, arise from $\prod_{v\in\leftpar}d_v^\msl! \prod_{a\in\rightpar}d_a^\msr!$ matchings and thus have the same probability. Thus conditioning the $\BCM$ on being simple indeed leads to a uniform simple bipartite graph with the given $\msl$- and $\msr$-degree sequences.

Note that the probability of obtaining a simple graph might tend to $0$ as $n\to\infty$. Whether the asymptotic probability of obtaining a simple graph is positive is a non-trivial question and falls out of the scope of this paper. Partial results are known, e.g.\ the condition $\E[(D^\msl)^2]<\infty, \E[(D^\msr)^2]<\infty$ guarantees a positive simplicity probability, as shown in \cite{AngHofHol2016}.

We remark that using the above observed relation of the $\BCM$ to uniform random graphs with given degree sequences, our results can be extended beyond the scope of the $\BCM$. It is known that the generalized random graph (GRG) conditioned on its degree sequence yields a uniform random graph with those degrees \cite{BriDeiLof06}. One can define a bipartite version of the model, with lhs partition $[N_n]$ and rhs partition $[M_n]$ with weight sequences $(w_i^\msl)_{i\in [N_n]}$ and $(w_j^\msr)_{j\in [M_n]}$ such that $\he_n := \sum_{i\in [N_n]} w_i^\msl = \sum_{j\in [M_n]} w_j^\msr$. Then the edge probability can be defined as $p_{ij} = \frac{w_i^\msl w_j^\msr }{\he_n +  w_i^\msl w_j^\msr}$ for $i\in[N_n], j\in[M_n]$. A similar argument as in \cite[Section 3]{BriDeiLof06} shows that conditionally on the lhs and rhs degree sequences, this model also yields a uniform bipartite graph with the given degree sequences. Consequently, the bipartite version of the GRG also undergoes a phase transition as in \cref{thm:BCM_giantcomp}. We omit further details.

\smallskip
\paragraph*{\textbf{Open problems and future research directions}}
For such a young model as the $\RIGC$, there are obviously plenty of open questions. It would be really interesting to fit the model to real-world network data to gain more insight into what type of network it is a good fit for, as well as study its performance for finite network sizes in comparison with the asymptotic theoretical results. Another exciting but challenging problem is studying graph distances: due to the community structures added, distances in the $\RIGC$ can be significantly different from the underlying $\BCM$. The \emph{homogeneous bond percolation} (retaining each edge independently with the same probability), that we study in the next section, also leaves open problems and plenty of room for generalizations. The question of \emph{robustness}, formally defined in \cref{ss:results_perc}, informally speaking the ability of the network to withstand random attacks, is explored further in a manuscript in preparation \cite{KomVad19}. One can consider \emph{inhomogeneous} percolation (with different retention probabilities), for example make the retention probability dependent on the degrees of the endpoints or the community graph the edge is part of. (The methods we present in \cref{s:perc_proof} would work for the latter case, however not the former.) 
Another common generalization is site percolation, i.e., percolating \emph{vertices} rather than edges, or even combining the two approaches.

\subsection{Results on percolation on the random intersection graph with communities}
\label{ss:results_perc}
In this section, we introduce the percolation model and state our results on percolation on the random intersection graph with communities.

\subsubsection{Introduction to percolation}
\label{sss:perc_intro}
In this section, we motivate and introduce the percolation model, and prove that percolation on the $\RIGC$ exhibits a phase transition (to be defined later) as we vary the percolation parameter.

\emph{Percolation} \cite{BolRio06perc,Grim99perc} is a probabilistic model introduced in \cite{BroHam57} to study a group of physical phenomena related to a ``fluid'' spreading through a ``porous medium'' in a unified, abstract way. 
Examples and motivations given in \cite{BroHam57} include adsorption of gas or liquid into a porous rock and spreading of a disease through a social network. Percolation processes differ from diffusion processes in that a diffusion process is largely determined by properties of the fluid, while in the case of percolation, the spreading behavior is largely determined by properties of the \emph{medium}. 
In the mathematical model of percolation, we represent the ``porous medium'' by a graph and define a random environment where edges (bond percolation) or vertices (site percolation) of this graph are randomly removed. The `fluid' can then spread through all retained edges (resp., vertices). Many variations of the model exist, but here we focus on bond percolation and the Bernoulli case: each edge is removed with the same probability, independently of each other. 

The notion of \emph{phase transition} also has its roots in physics and refers to the phenomenon when a model shows significantly different behavior depending on a specific parameter. The most common example is the different states of matter, sometimes referred to as phases, that the same material assumes at different temperatures. The parameter value (or interval, in the case of finite systems) where the behavior change occurs is referred to as the \emph{critical point} (or critical window for large finite systems).

Percolation was extensively studied first on infinite (deterministic) lattices, where the phase transition is characterized by the presence or absence of an infinite connected component in the percolated graph, i.e., after the removal of edges. It is straightforward to apply the percolation model for finite as well as random graphs, but less straightforward to define a phase transition. Phase transition on finite graphs is commonly re-interpreted in the large graph limit, as whether or not a linear proportion of the graph is connected after percolation.

We are motivated to study percolation on the $\RIGC$ model by possible applications in epidemiology and large-scale randomized attacks on the network. The correspondence between random removal of edges and a randomized attack on the network is quite intuitive. For a virus spread, whether a computer or biological virus, the percolation model is able to capture the final infected cluster of an information cascade \cite{Cascades01,CascadesSurvey13} or an SI-epidemic \cite{MarL86,MarL98} as defined below.

In the SI-epidemic, individuals have two possible states: susceptible and infected, and infected individuals never recover. Initially, all individuals are susceptible, and at time $0$, we infect a single individual, the source. In each (discrete) time step, all the individuals that became infected in the previous step attempt to transmit the infection through all incident edges. (Each individual only attempts to spread the infection once.) Each transmission succeeds with probability $\pi$, independently of each other. If a successful transmission is made to a susceptible neighbor, then it becomes infected. This continues until there is a time when no new individual becomes infected, and then the process stops. In a system of size $N$, the process is terminated at the latest by time $N-1$. It is easy to see that the infected individuals are exactly the individuals in the percolated component of the source.

Formally, we define (bond) percolation on the $\RIGC$ as follows. Let $\pi\in[0,1]$ be a parameter called the edge \emph{retention} probability. 
Given a realization of the $\RIGC$, we retain each edge, independently of each other, with probability $\pi$, and otherwise delete it. We call the remaining subgraph (with two layers of randomness) the percolated $\RIGC$ and denote it by $\RIGC(\pi)$. Note that $\RIGC(1) = \RIGC$, and $\RIGC(0)$ is the empty graph.

\subsubsection{Phase transition of bond percolation}
\label{sss:bondperc}
Recall 
$D^\msl$ and $D^\msr$ from \cref{asmp:convergence} \cref{cond:limit_ldeg} and \cref{cond:limit_rdeg} respectively. 
Recall $\xi_\msl$ from \cref{thm:giantcomp} and $N_n=\abs{\leftpar}$. 
Denote the largest connected component\footnote{The component containing the most vertices, with ties broken arbitrarily.} of $\RIGC(\pi)$ by $\comp_1(\pi) = \nn\comp_1(\pi)$, and the second largest by $\comp_2(\pi) = \nn\comp_2(\pi)$.

\begin{theorem}[Percolation phase transition on the $\RIGC$]
\label{thm:perc_bond}
Consider (bond) percolation \allowbreak with edge retention probability $\pi\in[0,1]$ on $\RIGC(\bitd^\msl,\comvect)$ under \cref{asmp:convergence} and the supercriticality condition \cref{cond:supercrit}. Then there exists a function $\pi \mapsto \xi_\msl(\pi)$ and a threshold $\pi_c\in[0,1]$ such that $\abs{\comp_1(\pi)}/N_n \toinp \xi_\msl(\pi)$ and
\begin{enumeratei}
\item\label{case:perc_sub} if $\pi<\pi_c$, then $\xi_\msl(\pi) = 0$;
\item\label{case:perc_super} if $\pi>\pi_c$, then $\xi_\msl(\pi)\in(0,\xi_\msl]$ and $\comp_1(\pi)$ is whp unique: $\abs{\comp_2(\pi)}/N_n \toinp 0$.
\end{enumeratei}
\end{theorem}

We prove \cref{thm:perc_bond} as a consequence of \cref{thm:giantcomp} in \cref{s:perc_proof}. 
We refer to the behavior in case \eqref{case:perc_sub} as subcritical percolation and in case \eqref{case:perc_super} as supercritical percolation. We assume the supercriticality condition \cref{cond:supercrit} since when this condition fails, case \eqref{case:perc_super} becomes impossible, thus there is no phase transition. 
In the following, we characterize the threshold $\pi_c$. Recall \eqref{eq:def_sizebiasing} and \cref{asmp:convergence} \cref{cond:limit_ldeg}.  

\begin{proposition}[Characterization of the threshold $\pi_c$]
\label{prop:pi_c_properties}
Let $\rg$ denote a random graph with pmf $\bm\mu$ and let $U_\rg \cond \rg \distr \Unif[\vertices(\rg)]$. Let $\comp^\rg(U_\rg,\pi)$ denote the percolated component of $U_\rg$ within $\rg$ with edge retention probability $\pi$. The threshold of the edge retention probability in \cref{thm:perc_bond} above is given by 
\beq\label{eq:pi_c_representation} \pi_c 
= \inf \bigl\{ \pi:\, \E[\wit D^\msl] \cdot \E\bigl[ \abs{\rg} \,(\abs{\comp^\rg(U_\rg,\pi)}-1) \bigr] \big/ \E\bigl[ \abs{\rg} \bigr] > 1 \bigr\}, \eeq
where $\E[\cdot]$ denotes total expectation (with respect to all sources of randomness). Furthermore, $\pi_c<1$.
\end{proposition}

We prove \cref{prop:pi_c_properties} in \cref{ss:pi_c_properties}. We remark that $\pi_c < 1$ ensures that the set of supercritical percolation parameters is always non-empty. However, the set of subcritical parameters $\pi<\pi_c$ \emph{may} be empty. The phenomenon when $\pi_c = 0$ is called \emph{robustness}, and we explore it further in a manuscript in preparation \cite{KomVad19}.

\FloatBarrier
\section{The giant component of the {\rm RIGC} and the {\rm BCM}}
\label{s:proof_giant}
In this section, we prove \cref{thm:giantcomp} on the phase transition of the $\RIGC$ as a corollary of \cref{thm:BCM_giantcomp} on the phase transition of the $\BCM$, and prove \cref{thm:BCM_giantcomp} itself. The latter proof makes use of the continuous-time exploration algorithm sketched in \cref{ss:results_BCM}, that we describe in more detail later in this section, then analyze it.

\begin{proof}[Proof of \cref{thm:giantcomp} subject to \cref{thm:BCM_giantcomp}]
For some $v\in\leftpar$, let us denote its connected component in the $\RIGC$ by $\comp^\msp(v)$, and its connected component in the \emph{underlying} $\BCM$ (see \cref{rem:bcm}) by $\comp^\msb(v)$. Since every community graph is connected, two $\msl$-vertices are connected within the $\RIGC$ exactly when they are connected within the underlying $\BCM$. Consequently, $\comp^\msp(v) = \leftpar \cap \comp^\msb(v)$, and each connected component of the $\RIGC$ is exactly the set of $\msl$-vertices in the corresponding connected component of the underlying $\BCM$. 
Note that ordering the connected components of the underlying $\BCM$ by size generally \emph{does not} ensure that the corresponding connected components of the $\RIGC$ are also ordered by size. 
In the subcritical and critical case, $\abs{\nth\comp_{1,\msb}} = o_{\sss\P}(N_n)$ by \cref{thm:BCM_giantcomp} (recall that $M_n = \gamma N_n + o(N_n)$ by \cref{rem:asmp_consequences} \cref{cond:partition_ratio}). Since $\lvert\comp^\msp(v)\rvert \leq \lvert\comp^\msb(v)\rvert$ for any $v\in\leftpar$, we conclude that
\beq\label{eq:subcritlargest} \lvert\comp_{1}\rvert
= \max_{v\in\leftpar}\, \lvert\comp^\msp(v)\rvert
\leq \max_{v\in\leftpar}\, \lvert\comp^\msb(v)\rvert
= \lvert\comp_{1,\msb}\rvert = o_{\sss\P}(N_n). \eeq
Under the supercriticality condition \cref{cond:supercrit}, $\bigl\lvert \comp_{1,\msb}\cap\leftpar \bigr\rvert \big/N_n \toinp \xi_\msl$ by \cref{thm:BCM_giantcomp}, and for \emph{any other} component $\comp'$ of the $\BCM$, $\abs{\comp'\cap\leftpar} \leq \abs{\comp'} \leq \abs{\comp_{2,\msb}} = o_{\sss\P}(N_n)$. 
Thus necessarily,
\beq\label{eq:giant_correspondance} \comp_1 = \comp_{1,\msb}\cap\leftpar \text{ whp}, \eeq
which implies that $\abs{\comp_1}/N_n \toinp \xi_\msl$, and analogously with \eqref{eq:subcritlargest}, $\abs{\comp_2} = o_{\sss\P}(N_n)$. This concludes the proof of \cref{thm:giantcomp} subject to \cref{thm:BCM_giantcomp}.
\end{proof}

\subsection{Global exploration}
\label{ss:exploration}
We prove our results regarding the giant component of the $\BCM$ with the aid of the exploration algorithm sketched in \cref{ss:results_BCM}. The algorithm is an adaptation of the exploration algorithm of the $\rCM$ proposed by Janson and Luczak in \cite{JanLuc2009}, however the analysis poses new challenges. Below, we introduce the required terminology and notation, and formalize the algorithm in the form of pseudo-code. 

We call two (or more) half-edges \emph{siblings} (a family of half-edges) if they are incident to the same vertex. To keep notation simple, we do not always explicitly indicate the dependence on $n$, however it is always meant. Instead, we add the superscripts $\msl$ or $\msr$ to emphasize which partition each quantity is related to. We define the algorithm focusing on the lhs partition to obtain the statements in \cref{thm:BCM_giantcomp}. We could analogously define and analyze the algorithm focusing on the rhs partition and obtain the statements in \cref{cor:BCM_giantcomp} instead. 
Note that the number of paired half-edges in the two partitions must always be equal. 
All the quantities below are defined to be right-continuous, i.e., if the algorithm updates a quantity at time $t$, the value \emph{at} time $t$ is the updated value.

At any given time, $\leftpar$ is partitioned into the time-dependent set of \emph{sleeping} and \emph{awake} vertices. Initially, all $\msl$-vertices are sleeping, then they are later moved one by one to the awake set and never return to sleeping. Intuitively, an awake vertex is at least partially explored. We denote the number of sleeping $\msl$-vertices of degree $k$ at time $t$ by $\sleepingdeg_k^\msl(t)$. Similarly, $\rightpar$ is partitioned into the sleeping set and awake set, and each $\msr$-vertex starts in the sleeping set and later progresses into the awake set. 

The set of $\msl$-half-edges, at any given time, is partitioned as follows: the sleeping set of size $\sleeping^\msl(t)$, the active set of size $\actives^\msl(t)$ and the paired (dead) set. Intuitively, active half-edges are those half-edges that we already know belong to the component we are currently exploring and are still unpaired. We can thus use them to progress the exploration. Note that
\beq\label{eq:def_sleepingtotal} \sleeping^\msl(t) = \sum_{k=1}^\infty k \sleepingdeg_k^\msl(t). \eeq
Each $\msl$-half-edge progresses from sleeping to active to paired, or directly from sleeping to paired. Sometimes we say a half-edge ``dies'' to mean that we pair it or that we must pair it immediately. We thus refer to the union of the sleeping and active sets as the living (unmatched) set, which has size $\living^\msl(t) = \actives^\msl(t) + \sleeping^\msl(t)$. Further, we assign iid $\Exp(1)$ random variables to each $\msl$-half-edge, that we call the \emph{alarm clock} of the half-edge. Once the exploration time reaches the value of this variable, the alarm goes off, and if the half-edge is still unpaired, it dies and must be paired immediately. When an $\msl$-half-edge dies, if the incident $\msl$-vertex is sleeping, we set it awake, and set all sibling half-edges active. (If the incident $\msl$-vertex is already awake, we do not change the status of the vertex or the sibling half-edges.) When we set an $\msl$-vertex awake for a different reason, we set each incident half-edge active.

The $\msr$-half-edges are partitioned into the sleeping set, the \emph{waiting-to-be-paired} set of size $\waiting^\msr(t)$, and the paired (dead) set. Half-edges may progress from sleeping to paired directly, or through the waiting-to-be-paired status, but never move backwards. While the waiting-to-be-paired set on the rhs plays a role analogous to those of active half-edges on the lhs, we use a different notion to emphasize their different roles in the algorithm: while the set of active $\msl$-half-edges is allowed to grow large, the waiting-to-be-paired set always must be exhausted immediately. When an $\msr$-half-edge is paired, the incident $\msr$-vertex is set to awake, and all sibling half-edges are set to be waiting-to-be-paired. (By the design of the algorithm, this is the only way to set an $\msr$-vertex awake.)

\begin{algo}[Continuous-time exploration of the $\BCM$]
\label{algo:exploration} \normalfont
\begin{algorithmic}
\State Initially, $t=0$, all vertices and half-edges are sleeping, and all half-edges are unpaired. Let $\steps_1$ and $\steps_2$ be initially empty lists.
\While{there are unpaired $\msl$-half-edges ($\living^\msl(t)>0$)}
\If{there are no active $\msl$-half-edges ($\actives^\msl(t) = 0$)}
\startstep[step1]{Starting the exploration of a new component.}
\State pick a \emph{sleeping} $\msl$-half-edge $x$ uar.
\State set the incident $\msl$-vertex $v$ as \emph{awake}. ($\sleepingdeg^\msl_{\ldeg(v)}(t) := \sleepingdeg^\msl_{\ldeg(v)}(t) - 1$)
\State set $x$ and all sibling half-edges as \emph{active}. ($\actives^\msl(t) := \ldeg(v)$)
\State append $t$ to the list $\steps_1$. ($t$ remains unchanged.)
\endstep[\step1]
\EndIf
\startstep[step2]{Discovering a new $\msr$-vertex.}
\State pick an \emph{active} $\msl$-half-edge $x$ arbitrarily and a \emph{sleeping} $\msr$-half-edge $y$ uar.
\State match $x$ and $y$ to form an edge and set both as \emph{paired}. ($\actives^\msl(t) := \actives^\msl(t) - 1$)
\State set the $\msr$-vertex $a$ incident to $y$ as \emph{awake}.
\State set sibling half-edges of $y$ as \emph{waiting-to-be-paired}. ($\waiting^\msr(t) := \rdeg(a) - 1$)
\State append $t$ to the list $\steps_2$. ($t$ remains unchanged.)
\endstep[step2]
\While{there are waiting-to-be-paired $\msr$-half-edges ($\waiting^\msr(t) > 0$)}
\startstep[step3]{Exploring further connections of the $\msr$-vertex.}
\State pick a \emph{waiting-to-be-paired} $\msr$-half-edge $y$ arbitrarily.
\State wait $\rd t$ time until the first alarm clock of an \emph{unpaired} $\msl$-half-edge $x$ rings.
\If{$x$ is sleeping}
\State set the $\msl$-vertex $v$ incident to $x$ as \emph{awake}. ($\sleepingdeg^\msl_{\ldeg(v)}(t + \rd t) := \sleepingdeg^\msl_{\ldeg(v)}(t) - 1$)
\State set sibling half-edges of $x$ as \emph{active}. ($\actives^\msl(t + \rd t) := \actives^\msl(t) + \ldeg(v) - 1$)
\EndIf
\State match $x$ and $y$ to form an edge and set both as \emph{paired}. ($\living^\msl(t + \rd t) := \living^\msl(t) - 1$, $\waiting^\msr(t + \rd t) := \waiting^\msr(t) - 1$)
\State increase time $t := t + \rd t$.]
\endstep[step3]
\EndWhile
\EndWhile
\end{algorithmic}
\end{algo}

The unit of the algorithm we often focus on is one \emph{iteration} of the outer \emph{while} loop, i.e., the conditional execution of \step1, the execution of \step2 and the internal \emph{while} loop of \step{3}s, which corresponds to discovering an $\msr$-vertex and matching all its remaining half-edges. By construction, the lists $\steps_1$ and $\steps_2$ contain the time stamps of all executions of \step1 and \step2, respectively. Noting that in each iteration, \step2 is executed once while \step1 is executed once only if the condition is satisfied and is otherwise not executed, $\steps_1$ must be a sublist of $\steps_2$. We also remark that both $\steps_1$ and $\steps_2$ may contain duplicates of the same time stamp, as the time variable is only increased in \step3, which is not executed in those iterations when the condition of the internal \emph{while} loop fails, that is, when the $\msr$-vertex found has degree one, so that the chosen $\msr$-half-edge does not have any sibling half-edges. 

\begin{remark}[Original algorithm as special case]
In \cref{ss:discussion}, we have shown that when each $\msr$-vertex has degree $2$, the bipartite configuration model $\BCM_{n}(\bitd^\msl,\bitd^\msr)$ is equivalent to $\rCM_n(\bitd^\msl)$. In this case, \step{3} is executed exactly once in each iteration, and our algorithm gives back the exploration for the $\rCM$ in \cite{JanLuc2009}.  
\end{remark}

\subsection{Analysis of the exploration algorithm}\label{ss:ingredients_giantcomp_proof}
In this section, we study \cref{algo:exploration}. The results obtained serve as ingredients to the proof of \cref{thm:BCM_giantcomp} in \cref{ss:BCM_giantcomp_proof}. Recall that we begin the exploration of a new component exactly when \step{1} is executed, for which $\actives^\msl(t)=0$ is a \emph{necessary} condition. Thus our aim is to understand the behavior of $t\mapsto \actives^\msl(t)$ during the course of the exploration, in particular, to determine the zeros of this function. Our analysis, as in \cite{JanLuc2009}, is based on the simple observation that $\actives^\msl(t) = \living^\msl(t) - \sleeping^\msl(t)$. We move on to studying the quantities $\sleeping^\msl(t)$ and $\living^\msl(t)$ separately.

The dynamics of $\sleeping^\msl(t)$, similarly to the corresponding quantity in the algorithm in \cite{JanLuc2009}, are the following. Note that \step{2} does not affect $\sleeping^\msl(t)$. Regularly, $\msl$-half-edges are removed from the sleeping set when the alarm clock of the half-edge itself or one of its siblings rings, due to \step{3}. However, some families of $\msl$-half-edges are removed from the sleeping set due to \step{1}, when we start the exploration of a new component by picking a uniform $\msl$-half-edge and set it active together with its siblings, and set the incident $\msl$-vertex awake. Let $\wih \sleepingdeg_k^\msl(t)$ denote the number of $\msl$-vertices of degree $k$ such that the alarm clocks of all $\msl$-half-edges show a time greater than $t$, and define
\beq\label{eq:sleepingapprox} \wih \sleeping^\msl(t) := 
\sum_{k=1}^\infty k \wih \sleepingdeg_k^\msl(t). \eeq
Comparing with \eqref{eq:def_sleepingtotal}, we intuitively think of $\wih \sleeping^\msl(t)$ as the number of sleeping $\msl$-half-edges ignoring the contribution of \step{1}, and it serves as an approximation for $\sleeping^\msl(t)$. We recall the following result that holds unchanged for the bipartite case:

\begin{lemma}[Sleeping vertices and half-edges, {\cite[Lemma~5.2.]{JanLuc2009}}]\label{lem:sleeping}
Define
\beq\label{eq:defh1} h_1(z):=\E[D^\msl] z G_{\wit D^\msl}(z) \eeq
for $z\in[0,1]$. For any $t_0$ fixed, as $n\to\infty$,
\begin{gather}
\label{eq:sleeping_deg_k} \forall \,k\geq1,\quad \sup_{t\leq t_0} \Big\lvert \frac{1}{N_n} \wih \sleepingdeg_k^\msl(t) - p_k \e^{-kt} \Big\rvert \toinp 0; \\
\label{eq:sleeping_vert} \sup_{t\leq t_0} \Big\lvert \frac{1}{N_n} \sum_{k=1}^\infty \wih \sleepingdeg_k^\msl(t) - G_{D^\msl}(\e^{-t}) \Big\rvert \toinp 0; \\
\label{eq:sleeping_halfedge} \sup_{t\leq t_0} \Big\lvert \frac{1}{N_n} \wih \sleeping^\msl(t) - h_1(\e^{-t} ) \Big\rvert \toinp 0.
\end{gather}
\end{lemma}

We introduce
\beq\label{eq:atilde} \wih \actives^\msl(t) := \living^\msl(t) -\wih \sleeping^\msl(t), \eeq
that serves as our approximation for $\actives^\msl(t)$. 
The next lemma, recalled from \cite{JanLuc2009}, bounds the error that we make with this approximation: 

\begin{lemma}[The effect of \step{1}, {\cite[Lemma~5.3.]{JanLuc2009}}] \label{lem:step1effect}
With $\dmax{\msl}$ from \cref{rem:asmp_consequences} \cref{cond:dmax},
\beq 0 \leq \wih \sleeping^\msl(t) - \sleeping^\msl(t) < \sup_{s\leq t} \bigl(\wih \sleeping^\msl(s) - \living^\msl(s)\bigr) + \dmax{\msl}. \eeq
The above bound can be rewritten in the more convenient form
\beq 0 \leq \actives^\msl(t) - \wih \actives^\msl(t) = \wih \sleeping^\msl(t) - \sleeping^\msl(t) < -\inf_{s\leq t} \wih \actives^\msl(t) + \dmax{\msl}. \eeq
\end{lemma}

Recall the sequence $\steps_1$ from \cref{algo:exploration} that contains the time stamps of all executions of \step1, and that it may contain the same time stamp several times. 
Since this does not occur in the original algorithm in \cite{JanLuc2009}, we reprove the lemma to show that this does not cause an issue.

\begin{proof} First, we study what happens at a time $t \in \steps_1\subseteq\steps_2$. 
As explained after \cref{algo:exploration}, $t$ may appear several times in the sequences $\steps_1 \subseteq \steps_2$, as finding a degree $1$ $\msr$-vertex uses up its single half-edge in \step{2}, which results in not executing \step{3} in that iteration and not increasing the time variable. However, the number of active $\msl$-half-edges changes with each execution of \step{1} and \step{2}, thus we overwrite (redefine) $\actives^\msl(t)$ each time until an $\msr$-vertex with degree at least $2$ is found in \step{2}. This \step{2} then necessarily corresponds to the last instance of $t$ in $\steps_2$ and sets the final value of $\actives^\msl(t)$, as \step{3} must be executed next and the time variable will increase. Consider the last execution of \step{1} at $t$, which is either in the same or an earlier iteration than the last \step{2}, and denote the $\msl$-vertex woken up by this \step{1} by $v$. As \step{2} is executed at least once afterwards, 
we have $\actives^\msl(t) \leq \ldeg(v)-1 < \dmax{\msl}$. Recall that $\living^\msl(t) = \actives^\msl(t)+\sleeping^\msl(t)$, hence
\beq\label{eq:glob_proof1} \wih \sleeping^\msl(t)-\sleeping^\msl(t) = \wih \sleeping^\msl(t) - \living^\msl(t) + \actives^\msl(t) < \wih \sleeping^\msl(t) - \living^\msl(t) + \dmax{\msl}. \eeq
By the definition of $\wih \sleeping^\msl(t)$, $\wih \sleeping^\msl(t) - \sleeping^\msl(t) \geq 0$ and the difference grows only due to \step{1}, while it might decrease due to \step{2} or \step{3}.\footnote{E.g.\ when a clock of an $\msl$-half-edge rings that was waken up in \step{1} previously.} Hence for a time $t' \notin \steps_1$, $\wih \sleeping^\msl(t') - \sleeping^\msl(t') \leq \wih \sleeping^\msl(s) - \sleeping^\msl(s)$, where $s := \max\{t\in\steps_1:\, t<t'\}$. Then, using \eqref{eq:glob_proof1} and that the supremum is actually a finite maximum over a subsequence of $\steps_1$,
\beq \wih \sleeping^\msl(t') - \sleeping^\msl(t')
\leq \sup_{s\leq t'}\, \bigl(\wih \sleeping^\msl(s) - \sleeping^\msl(s)\bigr)
< \sup_{s\leq t'}\, \bigl(\wih \sleeping^\msl(s) - \living^\msl(s)\bigr) + \dmax{\msl}, \eeq
which concludes the proof.
\end{proof}

Next, we state our novel result on the process of living (unmatched) half-edges $\living^\msl(t)$. As remarked in the sketch of the proof of \cref{thm:BCM_giantcomp} in \cref{ss:results_BCM}, the dynamics of this process are significantly different from the corresponding process in \cite{JanLuc2009}. The analysis of the new process, carried out in \cref{ss:livingproof} below, is our major novel contribution to generalizing the algorithm to the bipartite case.

We introduce some notation necessary to state our result. For an arbitrary invertible function $f$, let $\inv f$ denote the inverse function of $f$, i.e., $\inv f ( f(z) ) = z$ for any $z$ in the domain of $f$ and $f(\inv f (z)) = z$ for any $z$ in the domain of $\inv f$ (i.e., the range of $f$). Recall \eqref{eq:def_sizebiasing}, \eqref{eq:def_genfunc}, $D^\msl$ and $D^\msr$ from \cref{asmp:convergence} \cref{cond:limit_ldeg}  and \cref{cond:limit_rdeg}. Since the generating function $G_X$ of a random variable $X$ taking values from $\N$ (such that $\P(X=0)<1$) is continuous and strictly increasing, $\inv G_X$ exists on the interval $\bigl[\P(X=0),1\bigr]$.

\begin{proposition}[Living half-edges]\label{prop:living}
Define the function
\beq\label{eq:defh2} h_2(z) := \E[D^\msl]z \inv G_{\wit D^\msr}(z) \eeq
on $[\wit q_0,1]$, where $\wit q_0 := \P\bigl(\wit D^\msr = 0\bigr) = q_1/\E[D^\msr]$. The process of living half-edges $\living^\msl(t)$ satisfies, for any $0<t_0<-\log\, \wit q_0$,
\beq\label{eq:living} \sup_{t\leq t_0} \Big\lvert \frac{1}{N_n} \living^\msl(t) - h_2(\e^{-t} ) \Big\rvert \toinp 0. \eeq
\end{proposition}

We prove \cref{prop:living} in \cref{ss:livingproof}. We remark that for the $\RIGC$ and its underlying $\BCM$, postulating $q_1=0$ implies $\wit q_0 = 0$ and $-\log\wit q_0 = \infty$. It is hard to intuitively interpret the appearance of an \emph{inverse} generating function in (\ref{eq:defh2}-\ref{eq:living}). The deeper analysis of the process $\living^\msl(t)$ in \cref{ss:livingproof} reveals that it is due to \step{2} happening instantaneously. The proof of \cref{thm:BCM_giantcomp} in \cref{ss:BCM_giantcomp_proof} provides additional justification that the inverse must appear here in order to obtain \eqref{eq:fixedpoint}, the fixed point equation for the \emph{composition} of the generating functions, which is given an intuitive interpretation later in \cref{ss:BPapprox}. 

\subsection{Living half-edges}
\label{ss:livingproof}
In this section, we carry out the analysis of the process of living half-edges in \cref{algo:exploration} and in particular, prove \cref{prop:living}. 
We first introduce our approach and the ingredients of the proof and complete the proof before proving our lemmas.

\subsubsection{Asymptotics for the living half-edges: proof of Proposition \ref{prop:living}}
\label{sss:hittingtimes_decomp}

As remarked in the sketch of the proof in \cref{ss:results_BCM}, the process of living half-edges is significantly more complex in the bipartite case, as it is a death process with \emph{occasional} instantaneous jumps. In light of \cref{algo:exploration}, we can now explain precisely how this death process arises. Note that in each execution of \step2, as well as in each execution of \step3, one $\msl$-half-edge is paired, thus both steps correspond to a jump of size $-1$. As \step2 does not increase the time variable, it corresponds to an instantaneous jump (a jump with infinite rate). In \step3, we wait for the first $\Exp(1)$ alarm clock (see \cref{ss:exploration}) of an \emph{unmatched} $\msl$-half-edge to ring, which corresponds to each living half-edge dying at rate $1$. Thus in the death process, the regular jumps corresponding to iterations of \step3 happen with rate $i$ from position $i$. 

To determine how often the instanteneous jumps happen, consider that in \step2 of each iteration (of the outer \emph{while} loop), we pick a sleeping $\msr$-vertex in a \emph{size-biased} fashion, since a uniform unmatched $\msr$-half-edge is drawn. It is the remaining degree of the chosen $\msr$-vertex that determines the number of iterations of \step3 in the inner \emph{while} loop. In contrast, in the original algorithm (see \cref{rem:pairingalgo}) \step3 is executed exactly once in each iteration, thus the two jumps can be ``merged'' into a single jump of size $-2$ with rate $2i$ from position $2i$. Consequently, in the original algorithm it suffices to study a death process with only regular jumps (of size $2$). However in our case, such a merging is not possible at all, 
due to the fact that $\msr$-vertices are used up in an order determined by a size-biased reordering (formally defined below), hence the degree distribution is continuously changing throughout the course of the algorithm. Thus, we take an alternative approach, instead using \emph{hitting times}. Note that the initial value of the death process $\living^\msl(0) = \living^{\msl,\sss (n)}(0) = \he_n$ is deterministic. For $c\in[0,1]$, we define the hitting time process
\beq\label{eq:taudef} \hitting(c) := \min\{ t:\; \living^\msl(t)\leq c \he_n \}. \eeq
The following claim ensures that studying the hitting times is essentially equivalent to studying the death process:

\begin{claim}[Concentration of a death process and its hitting times]\label{claim:concentration}
For each $n\in\N$, let $\bigl(\nth X(t)\bigr)_{t\geq0}$ be a pure death process with deterministic initial condition $a_n := \nth X(0)\to\infty$ as $n\to\infty$. For $c\in(0,1]$, let $\nth\CT(c) := \min\{t:a_n^{-1}\nth X(t)\leq c\}$ and let $f:[0,\infty)\to(0,1]$ be a strictly decreasing function such that $f(0) = 1$ and both $f$ and its inverse $\inv f$ are continuous. Then the following two statements are equivalent:
\begin{enumeratei}
\item\label{stm:t} for any $t_0<\infty$, $\displaystyle \sup_{t\leq t_0} \bigl\lvert a_n^{-1}\nth X(t) - f(t)\bigr\rvert \toinp 0$,
\item\label{stm:c} for any $c_0\in(0,1)$, $\displaystyle \sup_{c\geq c_0} \bigl\lvert \nth\CT(c) - \inv f(c) \bigr\rvert \toinp 0$.
\end{enumeratei}
\end{claim}

We prove \cref{claim:concentration} in \cref{apx:concentration_deathprocesses}. 
\Cref{claim:concentration} is straightforwardly tailored to be applicable for $\living^\msl(t)$. It is stated in slightly more generality to allow application for similar processes that we define shortly and are necessary for the analysis.

To understand $\living^\msl(t)$, we compare it to a ``standard'' process $\livingstd(t)$, defined as the pure death process where each individual dies independently with rate $1$. 
That is, in the process $\livingstd(t)$ each jump happens with rate $i$ from state $\livingstd(t)=i$, and we set the same initial condition $\livingstd(0) := \he_n$. 
The processes $\living^\msl(t)$ and $\livingstd(t)$ can be coupled in an intuitive way by using the same realization of jumps, however $\livingstd(t)$ ``forgets'' about the occasional infinite rates; in other words, 
all jumps happen with rate $i$ from position $i$. 
Due to its simpler dynamics, the behavior of $\livingstd(t)$ is well understood, and hence so is the behavior of its hitting times
\beq\label{eq:tfulldef} \hittingstd(c) := \min\{t:\, \livingstd(t)\leq c\he_n\}. \eeq
However, in the process $\living^\msl(t)$, the instantaneous jumps due to \step{2} \emph{save us time}, which gives rise to a crucial correction term. We define the saved time as
\beq\label{eq:tskipdef} \hittingskip(c) := \hittingstd(c)-\hitting(c) > 0, \eeq
with $\hitting(c)$ and $\hittingstd(c)$ defined in \eqref{eq:taudef} and \eqref{eq:tfulldef}. Recall \eqref{eq:def_sizebiasing}, \eqref{eq:def_genfunc}, $D^\msl$ and $D^\msr$ from \cref{asmp:convergence} \cref{cond:limit_ldeg} and \cref{cond:limit_rdeg}, and that $\inv f$ denotes the inverse of a function $f$. We can summarize the asymptotics of $\hitting(c)$, $\hittingstd(c)$ and $\hittingskip(c)$ in the following lemma:

\begin{lemma}[Concentration of $\hitting(c)$]\label{lem:tau} For any $c_0 > 0$, as $n\to\infty$,
\begin{gather}
\label{eq:tau_full} \sup_{c\geq c_0} \bigl\lvert \hittingstd(c) + \log(c) \bigr\rvert \toinp 0,\\
\label{eq:tau_skip} \sup_{c\geq c_0} \bigl\lvert \hittingskip(c) + \log\bigl(\inv G_{D^{\msr,\star}}(c)\bigr) \bigr\rvert \toinp 0.
\end{gather}
Consequently, by \eqref{eq:tskipdef}, 
\beq\label{eq:tau} \sup_{c\geq c_0} \bigl\lvert \hitting(c) + \log(c) - \log\bigl(\inv G_{D^{\msr,\star}}(c)\bigr) \bigr\rvert \toinp 0.\eeq
\end{lemma}

We prove \cref{lem:tau} in \cref{sss:hittingtimes_proof}. 
We point out the appearance of the inverse generating function in the asymptotics of the time saved $\hittingskip(c)$, which is related to the size-biased reordering. However, also note that the inverse generating function is of the distribution $D^{\msr,\star} = \wit D^\msr + 1$. Next, we 
prove \cref{prop:living} subject to \cref{claim:concentration,lem:tau}.

\begin{proof}[Proof of \cref{prop:living} subject to \cref{claim:concentration,lem:tau}]
By \eqref{eq:tau}, $\hitting(c)$ concentrates around
\beq\label{eq:finv} \inv f(c) = - \log \,c + \log\bigl(\inv G_{D^{\msr,\star}}(c)\bigr).\eeq
Thus, by \cref{claim:concentration}, $\living^\msl(t)/\he_n$ concentrates around $f$. We claim that $f$ can be expressed as
\beq\label{eq:ft} c = f(t) = \e^{-t} \inv G_{\wit D^\msr} (\e^{-t} ).\eeq
We show that the inverse of the above function $f$ is indeed $\inv f$ by rearranging for $t$ in a clever way. Let $s := \inv G_{\wit D^\msr} (\e^{-t} )$, then $\e^{-t} = G_{\wit D^\msr}(s)$, and $c = G_{\wit D^\msr}(s) \cdot s = G_{D^{\msr,\star}}(s)$, by \eqref{eq:def_sizebiasing} and \eqref{eq:def_genfunc}. Hence
\beq\label{eq:propliving_1} \e^{-t} = G_{\wit D^\msr}(s) = \frac{G_{D^{\msr,\star}}(s)}{s} = \frac{c}{\inv G_{D^{\msr,\star}}(c)}. \eeq
Applying the function $-\log(\cdot)$ on both sides of \eqref{eq:propliving_1}, and noting that $t = \inv f(c)$, yields \eqref{eq:finv} as required, and we conclude that $\sup_{t\leq t_0} \bigl\lvert \living^\msl(t) / \he_n - f(t) \bigr\rvert \toinp 0$. By \eqref{eq:defh2}, $h_2(\e^{-t}) = \E[D^\msl] f(t)$, and $\he_n / N_n = \E[D_n^\msl] \to \E[D^\msl]$ by \cref{rem:asmp_consequences} \cref{cond:partition_ratio} and \cref{asmp:convergence} \cref{cond:mean_ldeg}, thus \eqref{eq:living} follows. This concludes the proof of 
\cref{prop:living} subject to \cref{claim:concentration,lem:tau}.
\end{proof}

\subsubsection{Concentration of the hitting times}
\label{sss:hittingtimes_proof}
This section is dedicated to proving \cref{lem:tau}. We prove \eqref{eq:tau}, \eqref{eq:tau_full} and \eqref{eq:tau_skip} in this order.

\begin{proof}[Proof of \eqref{eq:tau}, subject to \eqref{eq:tau_full} and \eqref{eq:tau_skip}]
Combining \eqref{eq:tau_full} and \eqref{eq:tau_skip} through the triangle inequality yields that, for any $\eps > 0$ fixed,
\beq \begin{split} &\P\biggl( \sup_{c\geq c_0} \bigl\lvert \hitting(c) + \log(c) - \log\bigl(\inv G_{D^{\msr,\star}}(c)\bigr) \bigr\rvert > \eps \biggr) \\
&\leq \P\biggl( \sup_{c\geq c_0} \bigl\lvert \hittingstd(c) + \log(c) \bigr\rvert + \bigl\lvert - \hittingskip(c) - \log\bigl(\inv G_{D^{\msr,\star}}(c)\bigr) \bigr\rvert > \eps \biggr) \\
&\leq \P\biggl( \sup_{c\geq c_0} \bigl\lvert \hittingstd(c) + \log(c) \bigr\rvert > \eps/2 \biggr)
+ \P\biggl( \sup_{c\geq c_0} \bigl\lvert \hittingskip(c) + \log\bigl(\inv G_{D^{\msr,\star}}(c)\bigr) \bigr\rvert > \eps/2 \biggr) \to 0 \end{split} \eeq
as $n\to\infty$. That is, by the definition of convergence in probability, \eqref{eq:tau} holds.
\end{proof}

\begin{proof}[Proof of \eqref{eq:tau_full}]
Recall the ``standard'' pure death process $\livingstd(t)$ and its hitting times \eqref{eq:tau_full} from \cref{sss:hittingtimes_decomp}. Also recall that the process jumps from state $i$ to state $i-1$ at rate $i$. Using that $\Exp(i) \eqindis \Exp(1)\big/i$,
\beq\label{eq:taudecomp} \hittingstd(c) = \hittingstd \Bigl(\frac{\lfloor c\he_n \rfloor}{\he_n}\Bigr)
\eqindis \sum_{i=\lfloor c\he_n\rfloor +1}^{\he_n} \frac{\nth E_i}{i}, \eeq
where for any fixed $n$, $(\nth E_i)_{i\in\Z^+}$ are independent $\Exp(1)$ random variables. For convenience, 
we define the index set
\beq\label{eq:indexset} \indexes_c = \nth\indexes_c := \{ \lfloor c\he_n \rfloor + 1 \leq i \leq \he_n\}. \eeq
Then, for any $c$ fixed, using \eqref{eq:taudecomp} and recognizing the Riemann-approximation sums, 
\beq\label{eq:taufullE} \E[\hittingstd(c)] = \sum_{i\in\indexes_c} \frac{1}{i}
= \log(\he_n) - \log(c\he_n) + O\bigl(N_n^{-1}\bigr)
= - \log(c) + O\bigl(N_n^{-1}\bigr), \eeq
and
\beq\label{eq:taufullVar} \Var\bigl(\hittingstd(c)\bigr)= \sum_{i\in\indexes_c} \frac{1}{i^2} < \: \sum_{\mathclap{i=\lfloor c\he_n\rfloor+1}}^{\infty} \: \frac1{i^2} \to 0 \eeq
as $n\to\infty$, since $\he_n\to\infty$ (and $c>0$). For $s\geq0$, define the process 
\beq Y(s) = \nth Y(s) := \hittingstd(\e^{-s}) - \E\bigl[\hittingstd(\e^{-s})\bigr] = \sum_{\mathclap{i\in\nth\indexes_{\exp\{-s\}}}} \frac{\nth E_i-1}{i}. \eeq
Note that $Y$ is a zero-mean martingale and thus $Y^2(s)$ is a non-negative submartingale. We apply Doob's martingale inequality and \eqref{eq:taufullVar} to obtain the following, for any fixed $\eps>0$ and $c_0>0$, with $s_0:=-\log(c_0)<\infty$,
\beq\label{eq:proof_taufull_1} \bal \P\Bigl( \textstyle \sup_{c\geq c_0} &\bigl\{\hittingstd(c) - \E[\hittingstd(c)]\bigr\}^2 \geq \eps \Bigr) 
= \P\Bigl( \textstyle \sup_{s\leq s_0} Y^2(s) \geq \eps \Bigr) \\
&\leq \frac{\E\bigl[Y^2(s_0)\bigr]}{\eps}
= \frac{\Var\bigl(Y(s_0)\bigr)}{\eps} 
= \frac{\Var\bigl(\hittingstd(c_0)\bigr)}{\eps} \to 0 \eal \eeq
as $n\to\infty$. It follows that
\beq \sup_{c\geq c_0} \bigl\lvert \hittingstd(c) - \E[\hittingstd(c)] \bigr\rvert \toinp 0. \eeq
Consequently, by \eqref{eq:taufullE}, we can bound
\beq \sup_{c\geq c_0} \bigl\lvert \hittingstd(c) + \log(c) \bigr\rvert 
\leq \sup_{c\geq c_0} \bigl\lvert \hittingstd(c) - \E[\hittingstd(c)] \bigr\rvert + o_{\sss\P}(1) + O\bigl(N_n^{-1}\bigr) \toinp 0. \eeq
This concludes the proof of \eqref{eq:tau_full}.
\end{proof}

We remark that \cite[Lemma 6.1]{JanLuc2009} is applicable to $\livingstd(t)$, which provides a shorter alternative proof for \eqref{eq:tau_full}. However, we adopted the proof above to shed light on the decomposition \eqref{eq:taudecomp}, preparing for the proof of \eqref{eq:tau_skip}, which is much more interesting and insightful.

\begin{proof}[Proof of \eqref{eq:tau_skip}]
Recall the definition of the process $\livingstd(t)$ and its hitting times $\hittingstd(c)$ from \cref{sss:hittingtimes_decomp}. The decomposition in \eqref{eq:taudecomp} is equivalent to:
\beq\label{eq:taudecomp_easy} \hittingstd(c) \eqindis \sum_{i\in\indexes_c} \frac{\nth E_i}{i}. \eeq
Next, we derive a similar decomposition for $\hitting(c)$. Let $\indexesskip_c\subseteq\indexes_c$ denote the set of such indices $i\in\indexes_c$ that the jump from position $i$ to position $i-1$ in the process $\living^\msl(t)$ happened instantaneously, i.e., due to \step{2}. (We provide a formal definition of the set $\indexesskip_c$ later.) Clearly, since both processes are defined using the same realization of jumps, the difference in $\hitting(c)$ and $\hittingstd(c)$ only arises due to the different jump rates from positions $i\in\indexesskip_c$. While rate $i$ in $\livingstd(t)$ results in the term $\nth E_i/i$, the instantaneous jump in $\living^\msl(t)$ results in a $0$ term. That is, we can write
\beq \hitting(c) = \sum_{i\in\indexes_c\setminus\indexesskip_c} \frac{\nth E_i}{i} + \sum_{i\in\indexesskip_c} 0 
= \sum_{i\in\indexes_c\setminus\indexesskip_c} \frac{\nth E_i}{i}, \eeq
and necessarily the saved time is
\beq\label{eq:tskipsum} \hittingskip(c) = \hittingstd(c) - \hitting(c) 
= \sum_{i\in\indexesskip_c} \frac{\nth E_i}{i}. \eeq
We analyze $\hittingskip(c)$ through the index set $\indexesskip_c$. Recall that we discover a new $\msr$-vertex exactly when \step{2} is executed. 
This happens exactly when all half-edges of the previous $\msr$-vertex have been paired.\footnote{Here, we ignore the potential \step{1} in between, as \step1 does not pair any half-edges and consequently does not correspond to any jump in the process $\living^\msl(t)$.} 
Cumulatively, we execute \step{2} for the $(j+1)^\mathrm{st}$ time when all half-edges of the first $j$ $\msr$-vertices are paired. Let us denote the $\msr$-degree of the $j\ith$ explored $\msr$-vertex by $d_{\pi(j)}^\msr$. Clearly, $(d_{\pi(j)}^\msr)_{j\in[M_n]}$ is a random reordering of $\bitd^\msr$, or equivalently, $(\pi(j))$ is a random permutation. We pick the next $\msr$-vertex to explore by choosing a uniform unpaired $\msr$-half-edge, thus $\msr$-vertices are always chosen in a \emph{size-biased} fashion wrt their degrees. That is, $\msr$-vertices are explored in the order defined by a size-biased reordering. Define $\usedup_{j} := \{ \pi(1),\ldots,\pi(j) \} \subset [M_n]$, the random set of indices chosen (used) in the first $j$ steps, then the distribution of $\pi$ is given by 
\beq\label{eq:sizebiased_reorder} \P\bigl( \pi(j) = k \bcond \usedup_{j-1} \bigr) = 
\begin{cases}
0 &\text{for $k\in\usedup_{j-1}$}, \\
\displaystyle \frac{d_k^\msr}{\sum_{i\in[M_n]\setminus\usedup_{j-1}} d_i^\msr} &\text{for $k\in[M_n]\setminus\usedup_{j-1}$},
\end{cases}
\eeq
Denote the partial sums of the first $j$ $\msr$-degrees in this reordering by
\beq \partialsum_j := \sum_{i=1}^{j} d_{\pi(i)}^\msr, \eeq
where the empty sum $\partialsum_0=0$ by convention. Then $\he_n - \partialsum_j$ gives the state of $\living^\msl(t)$ after we finish exploring the $j\ith$ $\msr$-vertex, thus \step{2} must be executed again and from this position, an instantaneous jump happens. We can now give an alternative, formal definition of the index set 
\beq \indexesskip_c = \nth\indexesskip_c = \bigl\{ \he_n-\partialsum_j,\,j\in[M_n] \bigr\} \cap\nth\indexes_c. \eeq
Define
\beq\label{eq:jmaxn} \maxindex(c) := \max\bigl\{j:\, \he_n - \partialsum_j > c\he_n \bigr\}, \eeq
then we can rewrite \eqref{eq:tskipsum} as
\beq\label{eq:tskipdecomp} \hittingskip(c) = \sum_{j=0}^{\maxindex(c)} \frac{\mth E_j}{\he_n-\partialsum_j}, \eeq
where the set $(\mth E_j)_{j\in\N}$ is a (possibly reordered) subset of $(\nth E_i)_{i\in\Z^+}$, hence it is composed of iid $\Exp(1)$ random variables. We give a convenient alternative probabilistic interpretation to the decomposition in \eqref{eq:tskipdecomp}, allowing us to relate it to a process that we already understand.

We define a process $\alarmproc^\msr(s) = \alarmproc^{\msr,\sss (n)}(s)$ in continuous time $s\geq0$ on the $\msr$-half-edges, completely independent of the exploration algorithm. The process $\alarmproc^\msr(s)$ follows dynamics analogous to $\wih \sleeping^\msl(t)$,\footnote{We avoid the intuitive notion $\wih\sleeping^\msr(s)$ to emphasize that this process is not related to the exploration algorithm. We use a separate time variable $s$ rather than $t$ for the same reason.} 
formally defined as follows.
Initially, all $\msr$-vertices and $\msr$-half-edges are sleeping, and we assign independent $\Exp(1)$ alarm clocks to each $\msr$-half-edge. An $\msr$-vertex and \emph{all} its half-edges are woken up (and never return to sleeping) when the alarm clock on \emph{any} of the half-edges goes off. The process $\alarmproc^\msr(s)$ keeps track of the number of sleeping $\msr$-half-edges. The hitting times of this process correspond to $\hittingskip(c)$, formally,
\beq\label{eq:ztskiprelation} \bigl(\min\bigl\{ s:\, \alarmproc^\msr(s)\leq c\he_n \bigr\} \bigr)_{1\geq c>0}
\eqindis \bigl( \hittingskip(c) \bigr)_{1\geq c>0}, \eeq
where the distributional equality is meant as stochastic processes. 
We prove \eqref{eq:ztskiprelation} by induction on the number of awake $\msr$-vertices. Clearly, $\alarmproc^\msr(0)=\he_n=\he_n-\partialsum_0$. Assume the number of sleeping $\msr$-half-edges to be $\alarmproc^\msr(s)=\he_n-\partialsum_j$. Since the alarm clocks of awake $\msr$-half-edges can be ignored, the time we have to wait for the next $\msr$-half-edge $y$ to wake up has distribution $\mth E_{j}/ (\he_n - \partialsum_j)$. The $\msr$-half-edge $y$ is chosen uar among the sleeping ones, hence the incident $\msr$-vertex $a$ is chosen in a size-biased fashion. That is, $\rdeg(a) = d_{\pi(j+1)}^\msr$, where $\pi$ is a random permutation with distribution \eqref{eq:sizebiased_reorder}. Also note that all $\msr$-half-edges incident to $a$ are woken up at once, thus the change in $\alarmproc^\msr(s)$ is $-\rdeg(a)$. Hence
\beq \min \bigl\{s: \alarmproc^\msr(s) = \he_n-\partialsum_{j+1} \bigr\} - \min \bigl\{s: \alarmproc^\msr(s) = \he_n-\partialsum_j \bigr\}
= \frac{E'_{j}}{\he_n - \partialsum_j}, \eeq
where $E'_j$ is an $\Exp(1)$ random variable, independent of everything else. 
Then by induction,
\beq \min \bigl\{s: \alarmproc^\msr(s) = \he_n - \partialsum_{k+1} \bigr\}
= \sum_{j=0}^k E'_j \big/ (\he_n - \partialsum_j). 
\eeq
To determine the hitting time $\min\{s:\, \alarmproc^\msr(s)\leq c\he_n\}$, we want the smallest $k$ such that $\he_n - \partialsum_{k+1} \leq c\he_n$. Since $\alarmproc^\msr(s)$ is non-increasing, this is equivalent to finding the largest $k$ such that $\he_n - \partialsum_k > c\he_n$, which is straightforwardly $\maxindex(c)$ by \eqref{eq:jmaxn}. Thus
\beq \min\{s:\, \alarmproc^\msr(s)\leq c\he_n\} 
= \min\bigl\{s:\, \alarmproc^\msr(s)=\he_n- \partialsum_{\maxindex(c)+1}\bigr\} \\
= \sum_{j=0}^{\maxindex(c)} \frac{E'_j}{\he_n - \partialsum_j}, 
\eeq
where we recognize a decomposition analogous to that of $\hittingskip(c)$ from \eqref{eq:tskipdecomp}, with $(E'_j)_{j\in\N}$ rather than $(\mth E_j)_{j\in\N}$. However, as both sets contain iid $\Exp(1)$ random variables, the two processes evolve in the exact same way and the distributional identity \eqref{eq:ztskiprelation} follows.

Now all that is left is to determine the asymptotics of $\alarmproc^\msr(s)$ and apply \cref{claim:concentration} to translate it into the asymptotics of $\hittingskip(c)$. As $\alarmproc^\msr(s)$ is defined analogously to $\wih\sleeping^\msl(t)$, following the same dynamics on the opposite partition, we can use the results in \cref{lem:sleeping} for $\alarmproc^\msr(s)$, with the exchange of lhs and rhs quantities. Replacing the $\msl$-degree distribution by the $\msr$-degree distribution in \eqref{eq:defh1} and \eqref{eq:sleeping_halfedge} yields 
that for any $s_0$ fixed, as $M_n\to\infty$,
\beq\label{eq:zconcentration} \sup_{s\leq s_0} \big\lvert M_n^{-1} \alarmproc^\msr(s) - \E[D^\msr]\e^{-s}G_{\wit D^\msr}(\e^{-s} ) \big\rvert \toinp 0. \eeq
Since $z \cdot G_{\wit D^\msr}(z) = G_{D^{\msr,\star}}(z)$ by \eqref{eq:def_sizebiasing} and \eqref{eq:def_genfunc}, $\E[D_n^\msr]\to\E[D^\msr]$ by \cref{asmp:convergence} \cref{cond:mean_rdeg} and $\he_n = M_n \E[D_n^\msr]$ by \eqref{eq:def_halfedges}, we can rewrite \eqref{eq:zconcentration} as 
\beq \sup_{s\leq s_0} \big\lvert \he_n^{-1} \alarmproc^\msr(s) - G_{D^{\msr,\star}}(\e^{-s} ) \big\rvert \toinp 0, \eeq
for any $s_0$ fixed. If $c=f(s) = G_{D^{\msr,\star}}(\e^{-s})$, then $s=\inv f(c) = -\log\bigl(\inv G_{D^{\msr,\star}}(c)\bigr)$. Then by \cref{claim:concentration} and \eqref{eq:ztskiprelation}, for any $c_0$ fixed,
\beq \sup_{c\geq c_0} \bigl\lvert \hittingskip(c) + \log\bigl(\inv G_{D^{\msr,\star}}(c)\bigr) \bigr\rvert \toinp 0. \eeq
This concludes the proof of \eqref{eq:tau_skip}.
\end{proof}

\subsection{The giant of the BCM: proof of Theorem \ref{thm:BCM_giantcomp}}\label{ss:BCM_giantcomp_proof} 
In this section, we combine the insight gained above and prove \cref{thm:BCM_giantcomp}. 
For technical reasons, our proof requires that $\eta_\msl > 0$, with $\eta_\msl$ defined in \cref{thm:giantcomp}. 
Note that the composition $G_{\wit D^\msr} \circ G_{\wit D^\msl}$ from \eqref{eq:fixedpoint} is the generating function of the random sum 
\beq\label{eq:def_Nr} N^\msr := \sum_{i=1}^{\wit D^\msr} \wit D_{(i)}^\msl, \eeq
where $\wit D_{(i)}^\msl$ are iid copies of $\wit D^\msl$ and independent of $\wit D^\msr$. (This random variable will re-appear in \cref{ss:BPapprox} where it receives an intuitive explanation.) Note that, by properties of generating functions, $\eta_\msl = 0$ exactly when $G_{N^\msr}(0) = \P(N^\msr = 0) = 0$, which is equivalent to $\P(\wit D^\msl = 0) = \P(\wit D^\msr = 0) = 0$ by \eqref{eq:def_Nr}, which in turn is equivalent to $p_1 = q_1 = 0$ by \eqref{eq:def_sizebiasing}. 
In the proof, we shall impose the condition $q_1>0$, with $q_1$ defined in \cref{asmp:convergence} \cref{cond:limit_rdeg}, to ensure that $\eta_\msl>0$. Hence we first show that proving \cref{thm:BCM_giantcomp} for $q_1>0$ is sufficient:

\begin{claim}[Reduction to the case $q_1>0$]\label{claim:reduction}
\Cref{thm:BCM_giantcomp} with $q_1>0$ implies \cref{thm:BCM_giantcomp} for $q_1=0$.
\end{claim}

\begin{proof} 
Assume that \cref{thm:BCM_giantcomp} holds for $q_1>0$, and we are given a graph sequence with $q_1=0$. 
In the following, we introduce a modification of the graph sequence, parametrized by $\eps$, such that $q_1(\eps)>0$ for all $\eps>0$, while we get better approximations of the original graph sequence as $\eps\to 0$. 
Let $\dmin{\msr} := \min \{ k\in\Z^+:\, q_k>0 \}\geq2$, i.e., the minimal degree of the \emph{asymptotic} $\msr$-degree distribution, and fix $\eps$ such that $0<\eps<q_{\dmin{\msr}}$. Then (for $n$ large enough) we cut $\eps M_n$ $\msr$-vertices of degree $\dmin{\msr}$ into vertices of degree $1$, i.e., we replace each of them by $\dmin{\msr}$ $\msr$-vertices of degree $1$. The empirical $\msl$-degrees $D_\eps^{\msr,\sss (n)}$ then converge as $n\to\infty$ to a modified limit $D_\eps^\msr$ with pmf:
\beq q_k(\eps) := \begin{cases}
\displaystyle \frac{\eps \dmin{\msr}}{1+\eps(\dmin{\msr}-1)} &\text{for $k=1$,}\\
\displaystyle \frac{q_{\dmin{\msr}}-\eps}{1+\eps(\dmin{\msr}-1)} &\text{for $k=\dmin{\msr}$,}\\
\displaystyle \frac{q_k}{1+\eps(\dmin{\msr}-1)} &\text{otherwise.}
\end{cases} \eeq
Denote the smallest fixed point of $G_{\wit{D_\eps^\msr}}\circ G_{\wit{D^\msl}}$ by $\eta_\msl(\eps)>0$ and define $\xi_\msl(\eps):= 1 - G_{D_\eps^\msr}\bigl(\eta_\msl(\eps)\bigr)<1$. (Recall $\xi_\msl$ from \cref{thm:giantcomp}.) By our assumptions, \cref{thm:BCM_giantcomp} holds for the modified graph sequence, and consequently formulas (\ref{eq:leftsize}-\ref{eq:edges}) hold with $\eta_\msl(\eps)>0$ and $\xi_\msl(\eps)<1$. We now let $\eps \to 0$, then $D_\eps^\msr\toindis D^\msr$, 
thus $G_{\wit{D_\eps^\msr}} \to G_{\wit D^\msr}$ pointwise on $[0,1]$, which implies $\eta_\msl(\eps) \to \eta_\msl=0$ and $\xi_\msl(\eps) \to \xi_\msl=1$. The above cutting operation can only decrease the number of $\msl$-vertices in each connected component: we are duplicating $\msr$-vertices only and components may become disconnected. Then considering that $\xi_\msl(\eps) \to 1$, \eqref{eq:leftsize} must extend to $\xi_\msl=1$ as well, and (\ref{eq:deg_k}-\ref{eq:edges}) follow for $\eta_\msl=0$.
\end{proof}

\paragraph*{\textbf{Identifying components in the exploration}} 
In the following, wlog we assume that $q_1>0$ or equivalently, $\wit q_0 >0$, with $\wit q_0 = q_1/\E[D^\msr]$ defined in \cref{prop:living}. 
Recall \eqref{eq:defh1} and \eqref{eq:defh2} and define, for $z\in[\wit q_0,1]$, 
\beq\label{eq:defH} H(z) := h_2(z)-h_1(z) = \E[D^\msl] z \bigl( \inv G_{\wit D^\msr}(z) - G_{\wit D^\msl}(z) \bigr). \eeq
For convenience, denote
\beq\label{eq:def_func_h} h(z):= \inv G_{\wit D^\msr}(z) - G_{\wit D^\msl}(z). \eeq
Recall the process $\wih\actives^\msl(t)$ from \eqref{eq:atilde} that approximates $\actives^\msl(t)$. By \cref{lem:sleeping,prop:living}, for any $t_0<-\log\wit q_0$, $\wih\actives^\msl(t)$ satisfies
\beq\label{eq:H_Atilde_relation} \sup_{t\leq t_0} \Bigl\lvert \frac{1}{N_n} \wih \actives^\msl(t) - H \bigl(\e^{-t}\bigr) \Bigr\rvert \toinp 0. \eeq
Recall that we start exploring a new component when \step{1} is executed, for which $\actives^\msl(t)=0$ is a necessary condition. By the intuition that $\actives^\msl(t)/N_n \approx \wih\actives^\msl(t)/N_n$ from \cref{lem:step1effect}, and \eqref{eq:H_Atilde_relation}, we want to find the zero(s) of $t\mapsto H\bigl(\e^{-t}\bigr)$ on $\R^+$. By (\ref{eq:defH}-\ref{eq:def_func_h}), the zeros of this function are described by $h(\e^{-t}) = 0$. Rearranging leads to the fixed point equation $G_{\wit D^\msr} (G_{\wit D^\msl}(\e^{-t}) ) = \e^{-t}$ for some $t\in\R^+$, or equivalently, $G_{\wit D^\msr} \circ G_{\wit D^\msl}(z) = z$ for some $z\in(0,1)$, 
which is the generating function of $N^\msr$ defined in \eqref{eq:def_Nr}. 
We always have the trivial fixed point $1$, however whether a second fixed point exists or not 
depends on whether the derivative $G_{N^\msr}'(1) = \E[N^\msr] = \E[\wit D^\msr] \E[\wit D^\msl] > 1$, which is exactly the supercriticality condition \cref{cond:supercrit}. In the following, we study the two cases separately, and show that a giant component exists if and only if \cref{cond:supercrit} holds.

\subsubsection{The supercritical case}
First, we study the case when \cref{cond:supercrit} holds. Recall $N^\msr$ from \eqref{eq:def_Nr} and \eqref{eq:def_genfunc}. Since $G_{N^\msr}'(1) > 1$, there exists a second fixed point $\eta_\msl<1$ in the interval $[0,1]$. 
In fact, $\eta_\msl > \wit q_0$ with $\wit q_0$ from \cref{prop:living}, by the following reasoning. By the definition of $N^\msr$, $\P(N^\msr=0) \geq \P(\wit D^\msr=0) = \wit q_0 = q_1/\E[D^\msr]$, which is positive by assumption. Thus the fixed point $\eta_\msl$ cannot be $0$, and consequently by the strict monotonicity of $G_{N^\msr}$, $\eta_\msl = G_{N^\msr}(\eta_\msl) > G_{N^\msr}(0) > \P(N^\msr=0) \geq \wit q_0$. 
Define 
\beq\label{eq:deftstar} t^\star := -\log\eta_\msl, \eeq
which lies in $(0,-\log \wit q_0)$, and consequently $t^\star$ is the \emph{unique} value of $t\in\R^+$ such that $H\bigl(\e^{-t^\star}\bigr) = 0$. 
In the following, we work towards showing that 
the exploration of the giant component lasts from time $0+o_{\sss\P}(1)$ to time $t^\star \pm o_{\sss\P}(1)$. Define
\beq\label{eq:def_t0} t_0 := -\log\bigl( (\eta_\msl+\wit q_0)/2\bigr), \eeq
so that $t^\star < t_0 < -\log\wit q_0$, 
and denote the ``good event''
\beq\label{eq:def_goodevent_G} \goodevent_1 (\delta) = \nth\goodevent_1(\delta) := \bigl\{ \sup\nolimits_{t\leq t_0} \bigl\lvert N_n^{-1} \wih \actives^\msl(t) - H(\e^{-t}) \bigr\rvert < \delta \bigr\}. \eeq
Note that since $t_0<-\log \wit q_0$, both \cref{lem:sleeping,prop:living} are applicable for this choice of $t_0$. Consequently, for any fixed $\delta$, by \eqref{eq:H_Atilde_relation} the good event happens whp, i.e.,
\beq\label{eq:goodevent_G_whp} \P\bigl(\nth\goodevent_1(\delta)\bigr) \to 1 \eeq
as $n\to\infty$. By properties of the generating function $G_{N^\msr}$ and \eqref{eq:defH}, rearranging yields that $z\mapsto H(z)$ is positive for $z\in(\eta_\msl,1)$, thus $t\mapsto H(\e^{-t})$ is positive for $t\in(0,t^\star)$. In fact, we have the following analytical properties of $t\mapsto H\bigl(\e^{-t}\bigr)$:

\begin{claim}\label{claim:concavity_replacement}
For any $\eps>0$ small enough, there exists $\delta = \delta(\eps)>0$ such that $t\mapsto H\bigl(\e^{-t}\bigr)>\delta$ on $t\in(\eps,t^\star-\eps)$ and $H\bigl(\e^{-(t^\star+\eps)}\bigr)<-2\delta$.
\end{claim}

\begin{proof}
Recall \eqref{eq:defH} and \eqref{eq:def_func_h} and note that $0 < \e^{-t_0} < \e^{-t} < 1$ is bounded for $t \in (0,t_0)$. It is sufficient to show that, for some $\delta'>0$,
\beq\label{eq:glob_proof2} \begin{cases}
h\bigl(\e^{-t}\bigr)>\delta' & \text{on $(\eps, t^\star-\eps)$,}\\
h\bigl(\e^{-t}\bigr)<-2\delta' & \text{at $t = t^\star+\eps$,}
\end{cases}
\eeq
then the required statement follows for $\delta := \delta' \E[D^\msl] \e^{-t_0}$. 
By the strict monotonity of the mapping $t \mapsto \e^{-t}$, \eqref{eq:glob_proof2} is equivalent to
\beq\label{eq:concavity_replacement}
\begin{cases}
h(z)>\delta' & \text{ for $z\in \bigl(\e^{-(t^\star-\eps)}, \e^{-\eps}\bigr) = (\eta_\msl+\eps_1, 1-\eps_2)$,}\\
h(z)<-2\delta' & \text{ for $z = \e^{-(t^\star+\eps)} < \eta_\msl$.}
\end{cases}
\eeq
Recall \eqref{eq:def_func_h}. Note that by $q_1>0$, $z\mapsto h(z)$ is strictly concave on its domain $[\wit q_0,1]$ and positive exactly on $(\eta_\msl,1)$, hence for any $\eps$ fixed, we can choose $\delta'>0$ appropriately such that \eqref{eq:concavity_replacement} holds. This concludes the proof of \cref{claim:concavity_replacement}.
\end{proof}

\smallskip
\paragraph*{\textbf{Finding the largest component}}
In the following, we aim to characterize those executions of \step{1} where we start exploring the giant component and the component after, i.e., when we finish exploring the giant. Recall the definition of $\steps_1$ from \cref{algo:exploration}. Denote the last element of $\steps_1$ that is less than $t^\star/2$ by $T_1$, and denote the next element after $T_1$ by $T_2$, i.e., $T_2$ is the first element of $\steps_1$ that is at least $t^\star/2$. Formally,
\beq\label{eq:gianttimes} T_1 = \max\{t\in\steps_1:\,t\leq t^\star/2\},
\qquad T_2 = \min\{t\in\steps_1:\,t>t^\star/2\}, \eeq
with the convention that the minimum over an empty set is $+\infty$. Later, we show that the exploration of the largest component lasts from $T_1$ to $T_2$. We first show the following:

\begin{lemma}[Exploration time of the ``giant'']\label{claim:giantcomptimes}
As $n\to\infty$,
\beq T_1 \toinp 0, \qquad T_2 \toinp t^\star. \eeq
\end{lemma}

\begin{proof}
Note that $\actives^\msl(t)-\wih\actives^\msl(t) = \wih\sleeping^\msl(t)-\sleeping^\msl(t) > 0$ by definition. By \eqref{eq:def_goodevent_G} and \cref{claim:concavity_replacement}, on the event $\goodevent_1(\delta)$ for $t\in(\eps,t^\star-\eps)$,
\beq\label{eq:Apositive} \actives^\msl(t) \geq \wih\actives^\msl(t) > \delta N_n > 0. \eeq
Recall that executing \step{1} requires $\actives^\msl(t) = 0$. Consequently, on the event $\goodevent_1(\delta)$, \step{1} could not have been executed within the time interval $(\eps,t^\star-\eps)$, hence on this event,
\beq\label{eq:nostep1} T_1\leq\eps, \qquad T_2\geq t^\star-\eps. \eeq
Noting that $0\in\steps_1$, thus $0\leq T_1$, it follows that $T_1 \toinp 0$ by \eqref{eq:goodevent_G_whp} and \eqref{eq:nostep1}. We have yet to give an upper bound on $T_2$ to prove that $T_2 \toinp t^\star$. We do so by proving that \step{1} must have been executed between $t^\star-\eps$ and $t^\star+\eps$. In fact, we show that the error $\actives^\msl - \wih \actives^\msl$ has increased on the smaller interval between $t^\star$ and $t^\star+\eps$, which can only happen due to \step{1}, as discussed in the proof of \cref{lem:step1effect}. 
Recall that $H\bigl(\e^{-t}\bigr)$ is positive on $(0,t^\star)$, hence on the event $\goodevent_1(\delta)$, $N_n^{-1}\wih\actives^\msl(t)>-\delta$ on $(0,t^\star)$. \Cref{lem:step1effect} is applicable for our choice of $t_0$ (see \eqref{eq:def_t0}) and $t^\star<t_0$, thus for any fixed $\delta$,
\beq\label{eq:A_Atilde_diff_small} \actives^\msl(t^\star) - \wih\actives^\msl(t^\star)
\leq -\inf_{t\leq t^\star} \wih\actives^\msl(t) + \dmax{\msl} < \delta N_n + \delta N_n / 2 = (3/2)\delta N_n, \eeq
as $\dmax{\msl}<\delta N_n/2$ for $n$ large enough by \cref{rem:asmp_consequences} \cref{cond:dmax}. However, on the event $\goodevent_1(\delta)$, 
\beq N_n^{-1} \wih\actives^\msl(t^\star+\eps) \leq -2\delta \eeq
by \cref{claim:concavity_replacement}, while $\actives^\msl(t)\geq0$ for any $t$. Thus
\beq\label{eq:A_Atilde_diff_big} \actives^\msl(t^\star+\eps) - \wih\actives^\msl(t^\star+\eps) \geq 2\delta N_n. \eeq
Comparing \eqref{eq:A_Atilde_diff_small} and \eqref{eq:A_Atilde_diff_big}, we see that $\actives^\msl(t)-\wih\actives^\msl(t)$ increased between times $t^\star$ and $t^\star+\eps$, which is only possible when \step{1} is executed. Consequently $T_2\leq t^\star+\eps$ on the event $\goodevent_1(\delta)$ that happens whp. Combining this with \eqref{eq:nostep1}, we obtain that $T_2 \toinp t^\star$, concluding the proof of \cref{claim:giantcomptimes}.
\end{proof}

\paragraph*{\textbf{Properties of the giant candidate}}
Recall that the exploration of each component starts with an execution of \step{1}, thus by \eqref{eq:gianttimes}, only one component is explored in the time interval $(T_1,T_2)$. Let us denote this component by $\comp_\star = \nth\comp_\star$. We study some properties of $\comp_\star$ that will help us in showing that $\comp_\star$ is whp the largest component. Recall from \cref{ss:exploration} that $\msl$-vertices can be sleeping or awake and $\msl$-half-edges can be sleeping, active or paired. Also recall that $\sleepingdeg_k^\msl(t)$ denotes the number of vertices of degree $k$ still sleeping at time $t$. Since $T_1, T_2\in\steps_1$, we have $\actives^\msl(T_1)=\actives^\msl(T_2)=0$, thus all $\msl$-half-edges that are removed from the sleeping set between $T_1$ and $T_2$ must be paired by time $T_2$. Thus all $\msl$-vertices and $\msl$-half-edges that are removed from the sleeping set between $T_1$ and $T_2$ are part of the component $\comp_\star$. Hence, with $\leftpar_k$ defined in \eqref{eq:def_Vk},
\begin{gather}\label{eq:giantcompthm_1_vertices} \bigl\lvert \vertices_k^\msl \cap \comp_\star \bigr\rvert 
= \sleepingdeg_k^\msl(T_1) - \sleepingdeg_k^\msl(T_2), \\
\label{eq:giantcompthm_1_edges} \bigl\lvert \edges(\comp_\star) \bigr\rvert = 
\sleeping^\msl(T_1) - \sleeping^\msl(T_2).
\end{gather}
Recall that $t^\star$ is defined in \eqref{eq:deftstar} so that $H(\e^{-t^\star})=0$, and further, $H(\e^{-t})>0$ for $t\in(0,t^\star)$. By \cref{claim:giantcomptimes} and the continuity of $H$, $\inf_{t\leq T_2}H\bigl(\e^{-t}\bigr) \toinp \inf_{t\leq t^\star}H\bigl(\e^{-t}\bigr) = 0$. Thus whp \eqref{eq:H_Atilde_relation} applies to $T_2$ and yields $N_n^{-1}\inf_{t\leq T_2} \wih\actives^\msl(t) \toinp 0$ as well. 
Note that $\wih\sleepingdeg_k^\msl(t)\geq\sleepingdeg_k^\msl(t)$ for all $t$ and $k$, with $\wih\sleepingdeg_k^\msl(t)$ defined in \cref{ss:exploration}. 
Recall \eqref{eq:def_sleepingtotal} and \eqref{eq:sleepingapprox}. 
By \cref{lem:step1effect},
\beq\label{eq:errorT2} \frac{1}{N_n} \sup_{t\leq T_2}\, \bigl(\wih\sleepingdeg_k^\msl(t)-\sleepingdeg_k^\msl(t)\bigr)
\leq \frac{1}{N_n} \sup_{t\leq T_2} \bigl(\wih\sleeping^\msl(t) - \sleeping^\msl(t)\bigr)
\leq \frac{1}{N_n} \inf_{t\leq T_2} \wih\actives^\msl(t) 
+ \frac{\dmax{\msl}}{N_n} \toinp 0. \eeq
Combining \eqref{eq:giantcompthm_1_vertices} and \eqref{eq:errorT2} with \eqref{eq:sleeping_deg_k} from \cref{lem:sleeping},
\beq\label{eq:giantcompthm_3} \bal &N_n^{-1}\bigl\lvert \vertices_k^\msl\cap\comp_\star \bigr\rvert
- \bigl( p_k \e^{-k T_1} - p_k \e^{-k T_2} \bigr) \\
&= N_n^{-1}\bigl(\sleepingdeg_k^\msl(T_1) - \sleepingdeg_k^\msl(T_2)\bigr)
- N_n^{-1}\bigl(\wih\sleepingdeg_k^\msl(T_1) - \wih\sleepingdeg_k^\msl(T_2)\bigr) \\
&\phantom{{}={}}+ \bigl( N_n^{-1}\wih\sleepingdeg_k^\msl(T_1) - p_k \e^{-k T_1} \bigr)
- \bigl( N_n^{-1}\wih\sleepingdeg_k^\msl(T_2) - p_k \e^{-k T_2} \bigr) \toinp 0. \eal \eeq
Since the function $t \mapsto p_k \e^{-k t}$ is continuous, by \cref{claim:giantcomptimes} and \eqref{eq:deftstar},
\beq\label{eq:giantcompthm_4} p_k \bigl( \e^{-k T_1} - \e^{-k T_2} \bigr)
\toinp p_k \bigl( \e^{-k0} - \e^{-k t^\star} \bigr)
= p_k \bigl( 1 - \eta_\msl^k \bigr). \eeq
Then combining \eqref{eq:giantcompthm_3} and \eqref{eq:giantcompthm_4} yields
\beq\label{eq:Cstarprop_degk} N_n^{-1}\bigl\lvert \vertices_k^\msl\cap \comp_\star \bigr\rvert \toinp p_k \bigl( 1 - \eta_\msl^k \bigr). \eeq
Similarly, by summation and \eqref{eq:sleeping_vert}, as well as \eqref{eq:giantcompthm_1_edges} and \eqref{eq:sleeping_halfedge}, respectively,
\begin{gather}
\label{eq:Cstarprop_vert} N_n^{-1}\bigl\lvert \leftpar\cap \comp_\star \bigr\rvert
\toinp G_{D^\msl}\bigl(\e^{-0}\bigr) - G_{D^\msl}\bigl(\e^{-t^\star}\bigr)
= 1 - G_{D^\msl}(\eta_\msl) = \xi_\msl, \\
\label{eq:Cstarprop_edge} N_n^{-1}\bigl\lvert \edges(\comp_\star) \bigr\rvert
\toinp \E[D^\msl]\bigl(1 - \eta_\msl G_{\wit D^\msl} (\eta_\msl) \bigr)
= \E[D^\msl]\bigl(1 - \eta_\msl \eta_\msr \bigr).
\end{gather}
In particular, $\comp_\star$ contains a linear proportion of edges and $\msl$-vertices.

\smallskip
\paragraph*{\textbf{Uniqueness}} Next, we prove that whp there is no other component containing a linear proportion of edges and vertices, hence $\comp_\star$ must be $\comp_{1,\msb}$ and further, the giant component is unique. Since $T_1 \toinp 0$, by \eqref{eq:sleeping_halfedge}, the total number of $\msl$-half-edges explored before $\comp_\star$ is $o_{\sss\P}(N_n)$. Consequently, whp no linear-sized component is explored before $\comp_\star$. Let us define $T_3$ as the element in $\steps_1$ right after $T_2$ (and $\infty$ if there is no such element)\footnote{It may occur that $T_3=T_2$, due to the multiplicities in the sequence $\steps_1$.}. The time of $T_3$ is given by
\beq T_3 = \min\bigl\{ t\in\steps_1\setminus\{T_2\}: t> t^\star/2 \bigr\}. \eeq
Recall \eqref{eq:A_Atilde_diff_small} and \eqref{eq:A_Atilde_diff_big} that we have used to prove that \step{1} must have been executed between $t^\star$ and $t^\star+\eps$, since the difference $\actives^\msl - \wih \actives^\msl$ can only increase due to \step{1}. In fact, we have shown that on the event $\goodevent_1(\delta)$, the difference increased by at least $\delta N_n /2$, that is, \emph{linearly} with $N_n$; however, each execution of \step{1} can only increase the difference by $\dmax{\msl}$, which is $o(N_n)$ by \cref{rem:asmp_consequences} \cref{cond:dmax}. Thus, \step{1} must have been executed not once, but many times between $t^\star$ and $t^\star+\eps$; in particular, $T_3 \leq t^\star + \eps$ on the event $\goodevent_1(\delta)$, which happens whp. 
Combining this with $T_3 \geq T_2$ and $T_2 \toinp t^\star$ yields that $T_3 \toinp t^\star$. Hence the component $\comp'$ explored between $T_2$ and $T_3$ has $o_{\sss \P}(N_n)$ edges by \eqref{eq:sleeping_halfedge}. 

Now assume that for some $\alpha>0$, there exists a component $\wih\comp$ with $\alpha N_n$ many edges, that was not explored before $\comp_\star$. Then, since we pick a new vertex by choosing a \emph{uniform} sleeping $\msl$-half-edge in \step{1}, we find $\wih\comp$ at $T_2$ with \emph{positive} probability, i.e., $\P(\comp' = \wih\comp) > 0$, which implies that $\liminf_{n\to\infty} \P\bigl( \lvert \edges(\comp') \rvert/N_n \geq \alpha \bigr)>0$. This contradicts that $\lvert\edges(\comp')\rvert / N_n \toinp 0$, and thus $\wih\comp$ cannot exist. We conclude that whp no component containing a linear proportion of edges was explored before or after $\comp_\star$. Note that if a connected component has linearly many vertices, it must also have linearly many edges. Hence whp $\comp_{1,\msb} = \comp_\star$ is the largest component, and is unique in the sense that there is no other linear-sized component. Then the properties proven for $\comp_\star$ in \eqref{eq:Cstarprop_degk}-\eqref{eq:Cstarprop_edge} verify the claimed properties of the giant in \eqref{eq:leftsize}-\eqref{eq:edges}. This concludes the proof of the \emph{supercritical case} of \cref{thm:BCM_giantcomp}.

\subsubsection{The non-supercritical case} We now study the case when \cref{cond:supercrit} does not hold. With $N^\msr$ from \eqref{eq:def_Nr}, 
we now have $\E[N^\msr] = \E[\wit D^\msl]\E[\wit D^\msr] \leq 1$, thus $G_{N^\msr} = G_{\wit D^\msr}\circ G_{\wit D^\msl}$ only has the trivial fixed point $1$. It is straightforward to check (rearranging \eqref{eq:defH}) that $t\mapsto H\bigl(\e^{-t}\bigr)$ is then negative on $\R^+$, its last and only zero is $t^\star_\text{subcrit}=0$.

Let us denote the first two elements of $\steps_1$ by $T_1'=t^\star_\text{subcrit}=0$ and $T_2' = \min\bigl\{t\in\steps_1\setminus\{T_1'\}\bigr\}$. By \cref{lem:step1effect} and its proof, we have that $\actives^\msl(0)-\wih\actives^\msl(0)=o_{\sss\P}(N_n)$ and since the error can only increase due to \step1, $\actives^\msl(T_2')-\wih\actives^\msl(T_2')\leq \actives^\msl(0)-\wih\actives^\msl(0) + \dmax{\msl} = o_{\sss\P}(N_n)$. On the other hand, for any $\eps>0$, $N_n^{-1} \wih\actives^\msl(\eps) \toinp H\bigl(\e^{-\eps}\bigr) < 0$ by \eqref{eq:H_Atilde_relation}. Noting that $\actives^\msl(\eps)\geq0$, $N_n^{-1}(\actives^\msl(\eps)-\wih\actives^\msl(\eps)) > 0$ whp, hence $T_2'<\eps$ whp, i.e., $T_2' \toinp 0$. Denote by $\comp^0$ the component explored between $T_1'$ and $T_2'$, then $\lvert\edges(\comp^0)\rvert = o_{\sss \P}(N_n)$ by \eqref{eq:sleeping_halfedge}. With an analogous argument to the proof of uniqueness in the supercritical case, no linear-sized component can exist, since we would find it at $T_1'$ with positive probability. Hence $\lvert\nth\comp_{1,\msb}\rvert = o_{\sss \P}(N_n)$. This concludes the proof of \cref{thm:BCM_giantcomp}. \proofends

We remark that the non-supercritical case could alternatively be proved by showing that $\limsup_{n\to\infty} \abs{\comp_1} / N_n \leq \xi_\msl$ through combining the local weak convergence result \cite[Theorem 2.14]{vdHKomVad18locArxiv} and arguments similar to those in \cref{s:local_global}. 

\section{Percolation phase transition and the giant of the {\rm RIGC}}
\label{s:perc_proof}
In this section, we prove \cref{thm:perc_bond} as a consequence of \cref{thm:giantcomp}. 

\subsection{Percolation on the {\rm RIGC} represented as an {\rm RIGC} with random parameters}
\label{ss:precolation_representation}
First, we focus on a qualitative understanding of the bond percolation model. 

Recall the construction of the $\RIGC$ from \cref{ss:RIGC_def}. Recall that $\edges(\comvect)$ denotes the disjoint union of edges in all community graphs. Further, denote the probability measure of the bipartite matching $\bmatch_n$ by $\P_{\bmatch_n}$. 
For a given $\bmatch_n$, denote the edge set of the corresponding realization of the $\RIGC$ by $\edges(\bmatch_n)$. Note that by construction and by our choice of treating the $\RIGC$ as a multigraph, for any given $\bmatch_n$, there is a one-to-one correspondence between $\edges(\comvect)$ and $\edges(\bmatch_n)$. 
For $e\in\edges(\comvect)$, we denote the corresponding edge $e'=e(\bmatch_n)\in\edges(\bmatch_n)$.

Recall that percolation is defined conditionally on the realization of the random graph $\RIGC$, as follows. Given $\bmatch_n$, each edge $e'\in\edges(\bmatch_n)$ is assigned an independent Bernoulli random variable $X_{e'}$ with success probability $\pi$, and we denote this conditional measure by $\P_\pi(\cdot \cond \bmatch_n)$. Together with the measure $\P_{\bmatch_n}$ of $\bmatch_n$, this determines the joint measure $\P_\pi$ of the percolated graph $\RIGC(\pi)$. In the following, we establish an alternative representation as a \emph{product measure}. Intuitively, we make use of the correspondence between $\edges(\comvect)$ and $\edges(\bmatch_n)$ to define percolation on the communities, rather than on the $\RIGC$, which can be done \emph{independently} of the bipartite matching.

\begin{figure}[hbt]
\centering
\begin{subfigure}[b]{0.28\textwidth}
	\centering
	\includegraphics[width=0.95\textwidth]{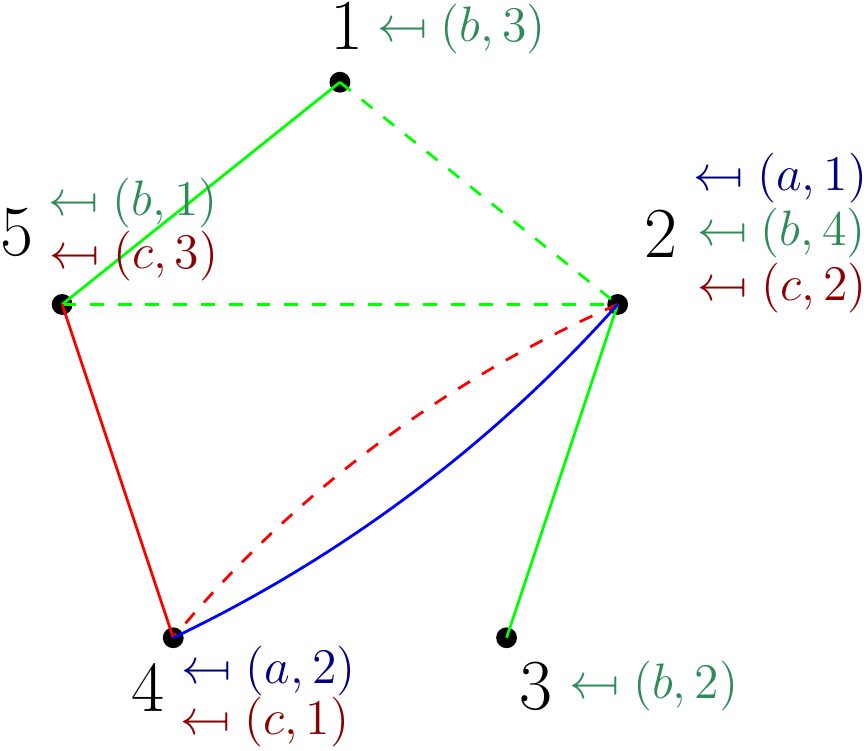}
	\caption{\emph{The percolated $\RIGC$}\\ The removed edges are represented as dashed.}
\end{subfigure}\hfill
\begin{subfigure}[b]{0.7\textwidth}
	\centering
	\includegraphics[width=0.95\textwidth]{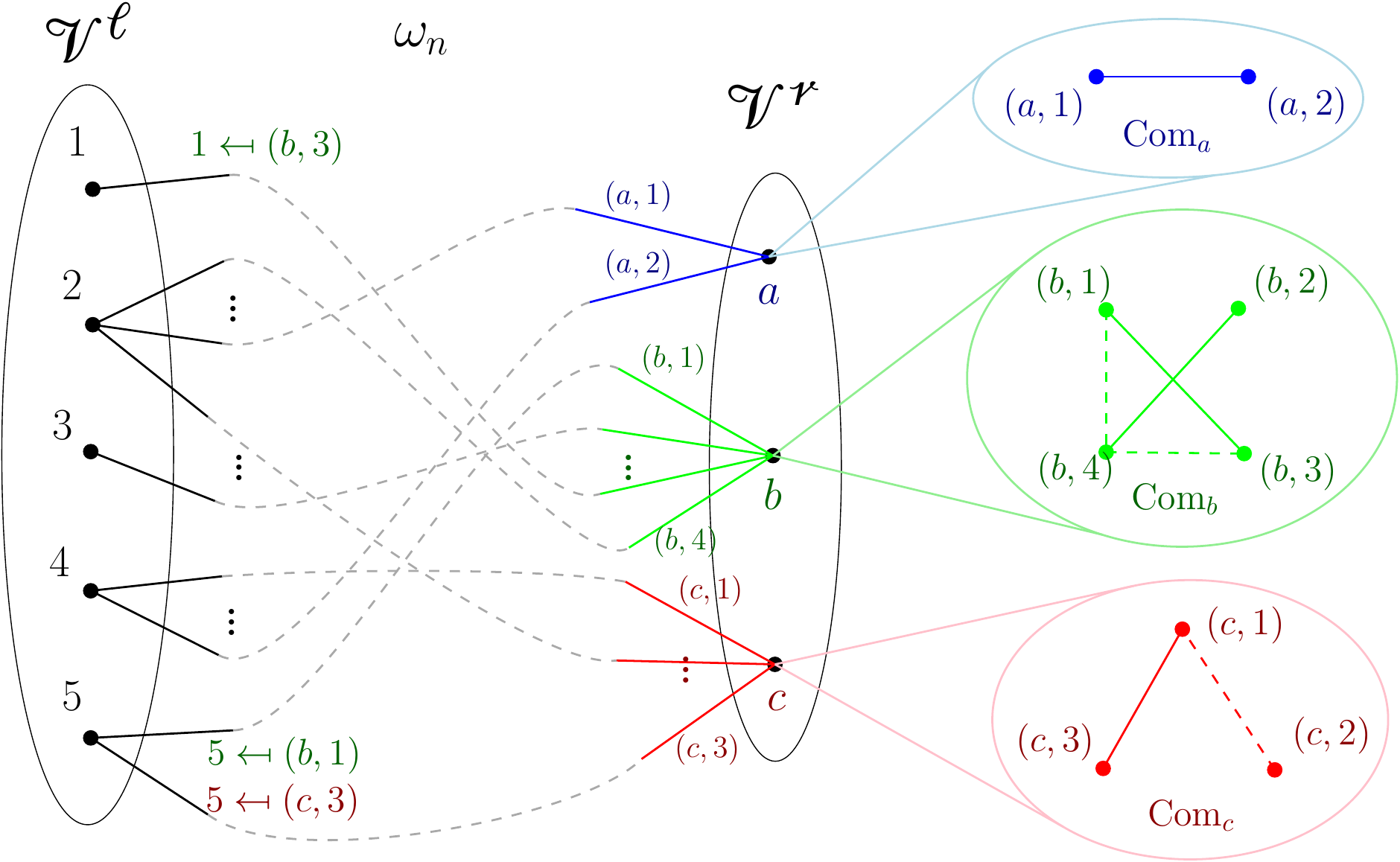}
	\caption{\emph{Percolation on the communities}\\ Using the group memberships, we can ``trace back'' each removed edge to a community edge.}
\end{subfigure}\\
\begin{subfigure}[b]{0.7\textwidth}
	\centering
	\includegraphics[width=0.95\textwidth]{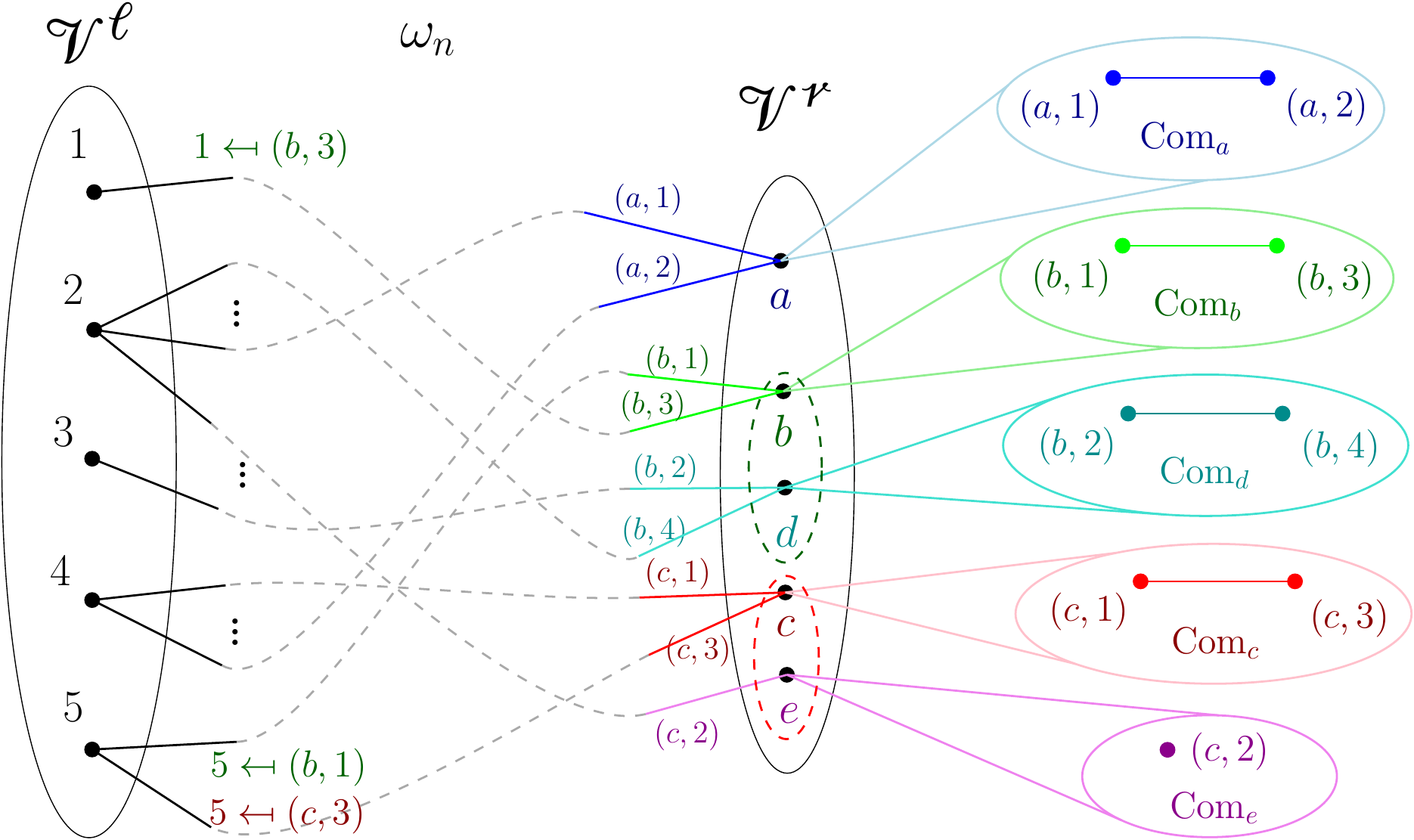}
	\caption{\emph{The new, percolated community list}\\ If communities become disconnected, we separate each connected component as its own community, e.g.\ $b$ is separated into $b$ and $d$.}
\end{subfigure}
\caption{Reducing percolation on the $\RIGC$ to percolation on the communities}
\end{figure}
We define percolation on the communities and the percolated community list $\comvect(\pi)$, as follows. With each $e\in\edges(\comvect)$, we associate an independent $\Bernoulli(\pi)$ random variable $X_e$; $e$ is retained exactly when $X_e=1$. Denote by $\com_a(\pi)$ the random graph produced by percolation on $\com_a$. Note that $\com_a(\pi)$ is not necessarily connected, which conflicts with our initial assumptions. Thus, we need to replace $\com_a(\pi)$ by the random list of its connected components $\bigl(\com_{a,i}(\pi)\bigr)_{i\in[c(\com_a(\pi))]}$, where $c(\com_a(\pi))$ denotes the number of connected components of $\com_a(\pi)$. Then $\bigl(\com_{a,i}(\pi)\bigr)_{a\in[M_n], i\in[c(\com_a(\pi))]}$ is the new list of communities. 
We introduce the new number of communities $M_n(\pi):=\sum_{a\in\rightpar} c(\com_a(\pi))$, so that the new rhs partition is $[M_n(\pi)]$. By re-indexing, we can now write and define $\comvect(\pi) := \bigl(\com_{a'}^\pi\bigr)_{a'\in[M_n(\pi)]}$. With these new parameters, the above intuition can be formalized as follows:

\begin{proposition}[Percolation on the $\RIGC$ is still an $\RIGC$]
\label{prop:rigc_perc}
Bond percolation with edge retention probability $\pi$ on an $\RIGC$ with parameters $\bitd^\msl$ and $\comvect$ is equivalent to an $\RIGC$ with parameters $\bitd^\msl$ and $\comvect(\pi)$. Formally,
\beq \RIGC(\bitd^\msl,\comvect)(\pi) \eqindis \RIGC(\bitd^\msl,\comvect(\pi)). \eeq
\end{proposition}

We refer to $\RIGC(\bitd^\msl,\comvect(\pi))$ as the \emph{$\RIGC$ representation of $\RIGC(\pi)$}. 

\begin{proof} 
Recall that given $\bmatch_n$, percolation on the $\RIGC$ is described by the iid $\Bernoulli(\pi)$ random variables $(X_{e'})_{e'\in\edges(\bmatch_n)}$. Also recall that each $e'\in\edges(\bmatch_n)$ can be written as $e'=e(\bmatch_n)$ for a unique $e\in\edges(\comvect)$ and define $X_e(\bmatch_n) := X_{e'}$ for each $e\in\edges(\comvect)$. 

A given realization of $\RIGC(\pi)$ can be characterized by its (unpercolated) edge set $\edges(\bmatch_n) = \edges$ and the outcomes of the Bernoulli variables, $x_{e'}\in\{0,1\}$ for $e'\in\edges$. Define $x_e := x_{e'}$, with $e'=e(\bmatch_n)$. We have that, for any given edge set $\edges$ and $(x_{e'})_{e'\in\edges}$,
\beq\label{eq:percolation_factorizes} \bal &\P_{\pi}\bigl( \edges(\bmatch_n) = \edges, X_{e'} = x_{e'} \;\forall {e'}\in\edges \bigr) \\
&= \P_{\bmatch_n}\bigl( \edges(\bmatch_n) = \edges \bigr) 
\P_\pi\bigl( X_{e'} = x_{e'} \;\forall {e'}\in\edges(\bmatch_n) \bcond \bmatch_n \bigr) \\
&= \P_{\bmatch_n}\bigl( \edges(\bmatch_n) = \edges \bigr) 
\P_\pi( X_e(\bmatch_n) = x_e \;\forall e\in\edges(\comvect) \bcond \bmatch_n \bigr) \\
&= \P_{\bmatch_n}\bigl( \edges(\bmatch_n) = \edges \bigr) 
\P\bigl( X_e = x_e \;\forall e\in\edges(\comvect) \bigr), \eal \eeq
where in the last step we have used that for any $\bmatch_n$, $(X_e(\bmatch_n))_{e\in\edges(\comvect)}$ are independent $\Bernoulli(\pi)$ random variables, thus the collection has the same law as $(X_e)_{e\in\edges(\comvect)}$. We conclude that the law of the percolated graph can indeed be written as a product measure.

Noting that $\bmatch_n$ did not change throughout \eqref{eq:percolation_factorizes}, we conclude that the new measure is still an $\RIGC$. Similarly, as $\bitd^\msl$ did not even appear in the formulas, it necessarily remains unchanged. As intuition has predicted, percolation can be executed on the communities before constructing the random graph, the formulas indeed contain the random variables $X_e$ corresponding to $e\in\edges(\comvect)$, meaning that the new $\RIGC$ must use $\comvect(\pi)$. This concludes the proof of \cref{prop:rigc_perc}.
\end{proof}

Next, we show that $\RIGC(\bitd^\msl,\comvect(\pi))$ still satisfies our assumptions, in the sense of \cref{rem:randomparam}. 
Denote the (random) empirical distribution of $\comvect(\pi)$ by $\nn{\bm\mu}(\pi)$. Then, we have the following convergence result:

\begin{lemma}[Convergence of percolated community list]
\label{lem:comvect_perc_conv}
Assume that the original $\comvect$ sequence satisfies \cref{asmp:convergence} \cref{cond:limit_com}. Then for the sequence of $\comvect(\pi)$, there exists a mass function $\bm\mu(\pi)$ on $\comgraphs$, such that for each $H\in\comgraphs$, as $n\to\infty$,
\beq\label{eq:comvect_perc_convinprob} \nn\mu_{H}(\pi) \toinp \mu_H(\pi). \eeq
Denote the empirical and limiting community-size distributions corresponding to $\nn{\bm\mu}(\pi)$ and $\bm\mu(\pi)$ respectively by $D_n^\msr(\pi)$ and $D^\msr(\pi)$. If the original $\comvect$ sequence also satisfies \cref{asmp:convergence} \cref{cond:mean_rdeg}, then
\beq\label{eq:comvect_perc_convinmean} \E\bigl[ D_n^\msr(\pi) \bcond \comvect(\pi) \bigr] 
\toinp \E[D^\msr(\pi)] < \infty. \eeq
\end{lemma}

We prove \cref{lem:comvect_perc_conv} in \cref{sss:perc_com_list_conv}. 
Recall that $\vertices(\comvect)$ denotes the disjoint union of vertices in all community graphs, and recall $J_n\sim\Unif[\vertices(\comvect)]$ and \eqref{eq:def_sizebiasing}. 
For $j\in\vertices(\comvect)$, let $\comp^\msc(j,\pi)$ denote the percolated component of $j$ within its community. 
The following statement provides insight into the percolated community sizes and it is also instrumental to the proof of \cref{prop:pi_c_properties}.

\begin{claim}[Representation of size-biased percolated community size]
\label{claim:sizebiased_comsize}
We have the following identity of distributions:
\beq\label{eq:perc_comsize_representation} \wit{D_n^\msr(\pi)} \eqindis \abs{\comp^\msc(J_n,\pi)} - 1. \eeq
\end{claim}

We give the proof of \cref{claim:sizebiased_comsize} in \cref{sss:sizebiased_comsize_representation}.

\subsection{Proof of Theorem \ref{thm:perc_bond}}
\label{ss:proof_bond_perc}
We now prove \cref{thm:perc_bond} subject to \cref{lem:comvect_perc_conv}.

\begin{proof}
By \cref{prop:rigc_perc}, studying percolated component sizes is equivalent to studying component sizes in $\RIGC(\bitd^\msl,\comvect(\pi))$, i.e., the $\RIGC$ representation of $\RIGC(\pi)$. By a slight abuse of notation, we use $\comp_1(\pi)$ and $\comp_2(\pi)$ to denote the largest and second largest component of $\RIGC(\bitd^\msl,\comvect(\pi))$, respectively. Recall $\bm\mu(\pi)$ from \cref{lem:comvect_perc_conv}, and the corresponding limiting community-size distribution $D^\msr(\pi)$. 

By \cref{lem:comvect_perc_conv,rem:randomparam}, our results apply to $\RIGC(\bitd^\msl,\comvect(\pi))$, the $\RIGC$ representation of $\RIGC(\pi)$. In particular, applying \cref{thm:giantcomp} yields that

{\it there exists $\eta_\msl(\pi) \in [0,1]$, the smallest solution of the fixed point equation
\beq\label{eq:fixedpoint_perc} \eta_\msl(\pi) = G_{\wit{ D^\msr(\pi)}}\bigl( G_{\wit D^\msl} \bigl( \eta_\msl(\pi) \bigr) \bigr), \eeq
and $\xi_\msl(\pi) := 1-G_{D^\msl}\bigl( \eta_\msl(\pi) \bigr) \in [0,1]$ such that
\beq\label{eq:perc_giant_conv} \abs{\comp_1(\pi)}/N_n \toinp \xi_\msl(\pi). \eeq
Furthermore, $\xi_\msl(\pi) > 0$ exactly when 
\beq\label{eq:supercrit_perc} 
\E[\wit D^\msl] \E[\wit{ D^\msr(\pi)}] > 1, \eeq
which we call \emph{supercritical percolation}. In this case, $\comp_1(\pi)$ is unique in the sense that $\abs{\comp_2(\pi)} = o_{\sss\P}(N_n)$, and we call $\comp_1(\pi)$ the \emph{percolated giant component}. }

In the following, we show that there exists $\pi_c\in[0,1]$ such that the set of supercritical parameters $\pi$ can be (almost exactly, as explained shortly) characterized by $\pi > \pi_c$. 
We do so by proving that $\xi_\msl(\pi)$ is a non-decreasing function of $\pi$. Subject to this statement, clearly there exists a threshold $\pi_c$ given by
\beq\label{eq:def_pi_c} \pi_c := \inf \{\pi:\, \xi_\msl(\pi) > 0 \} = \inf \bigl\{\pi:\, \E[\wit D^\msl] \E[\wit{ D^\msr(\pi)}] > 1 \bigr\}. \eeq 

This shows that the characterization is almost exact: $\pi > \pi_c$ implies that $\pi$ is supercritical, and $\pi < \pi_c$ implies that $\pi$ is not supercritical. Whether $\pi_c$ itself is supercritical, i.e., whether $\eta_\msl(\pi_c) < 1$, depends on continuity properties of $\pi \mapsto \xi_\msl(\pi)$, which are nontrivial in some cases. We conjecture that $\pi = \pi_c$ is \emph{not supercritical}; in \cite{KomVad19}, we show that this is true in special cases.

We now prove the required monotonicity of $\xi_\msl(\pi)$, by showing that for any $n$ fixed, $\abs{\comp_1(\pi)}/N_n$ is non-decreasing in $\pi$, in the sense of \emph{stochastic domination}: we say that $Y$ stochastically dominates $X$ and denote $X \preceq Y$ if $\P(Y \geq x) \geq \P(X \geq x)$ for all $x\in\R$. Subject to $\abs{\comp_1(\pi)}/N_n$ being non-decreasing in $\pi$, clearly the required monotonicity follows for the limit $\xi_\msl(\pi)$ as well, by \eqref{eq:perc_giant_conv}. We prove monotonicity for fixed $n$ through the so-called Harris-coupling, defined as follows. 
To each edge $e\in\edges(\RIGC)$, we assign independent standard uniform random variables $U_e$, and for any $\pi$, define $X_e^\pi := \ind_{\{U_e \leq \pi\}}$. The edges retained in $\RIGC(\pi)$, i.e., in $\pi$-percolation on $\RIGC$, are exactly the edges $e$ such that $X_e^\pi = 1$. Clearly, for $\pi_1 < \pi_2$, we have $X_e^{\pi_1} \leq X_e^{\pi_2}$ for any edge $e$, thus $\RIGC(\pi_1)$ is a subgraph (edge-subgraph) of $\RIGC(\pi_2)$. Denote by $\comp(v,\pi)$ the component of $v$ in $\RIGC(\pi)$. Suppose $v\in\comp_1(\pi_1)$, then 
\beq \abs{\comp_1(\pi_1)}
= \abs{\comp(v,\pi_1)} 
\leq \abs{\comp(v,\pi_2)} 
\leq \abs{\comp_1(\pi_2)}. \eeq
That is, $\abs{\comp_1(\pi_1)} \leq  \abs{\comp_1(\pi_2)}$ holds almost surely under this coupling, which implies the stochastic domination $\abs{\comp_1(\pi_1)}/N_n \preceq \abs{\comp_1(\pi_2)}/N_n$, as required. 
This concludes the proof of \cref{thm:perc_bond} subject to \cref{lem:comvect_perc_conv}. 
\end{proof}

\subsection{Percolated community list: convergence and percolated community sizes}
\label{ss:perc_com_list}
In this section, we provide the proof for \cref{lem:comvect_perc_conv,claim:sizebiased_comsize}.

\subsubsection{Convergence of the percolated community list}
\label{sss:perc_com_list_conv}

\begin{proof}[Proof of \cref{lem:comvect_perc_conv}]
We first prove \eqref{eq:comvect_perc_convinprob}. Recall that $\comvect(\pi) = \bigl(\com_a^\pi\bigr)_{a\in[M_n(\pi)]}$ (see \cref{ss:precolation_representation}). 
Recall \eqref{eq:def_VH} and for any possible community graph $H \in\comgraphs$, we introduce
\beq\label{eq:def_VHpi} \vertices_H^\msr(\pi) := \{ a\in[M_n(\pi)]:\; \com_a^\pi \isom H \}. \eeq
Our aim is to prove that the following quantity converges in probability to some constant: 
\beq\label{eq:perc_com_fraction} \nn\mu_H(\pi)
= \abs{\vertices_{H}^{\msr}(\pi)} \big/ M_n(\pi) 
= \frac{\abs{\vertices_{H}^{\msr}(\pi)}}{M_n} \cdot \frac{M_n}{M_n(\pi)} . \eeq
In the following, we prove convergence in probability to respective \emph{constants} for both factors separately. This implies convergence in probability for the product, despite the dependence. 

Recall that each new community in $\comvect(\pi)$ is a connected component under percolation on some original community in $\comvect$. Thus, to count the frequencies and total number of new communities, we break it down with respect to the original communities, and first study percolation on an arbitrary community and the frequency of each outcome.

We introduce some notation. Recall that $\comgraphs$ denotes the set of possible community graphs: simple, finite, connected graphs, with a fixed arbitrary labeling. For an arbitrary $F\in\comgraphs$, denote bond percolation on $F$ by $F(\pi)$. We introduce an object to compare realizations of this random graph to: let $\graphs$ denote the set of simple, finite, not necessarily connected, unlabeled graphs. Denote by $\supp(F(\pi))$ the set of all $G\in\graphs$ that are isomorphic to some possible realization of $F(\pi)$, which is exactly the set of edge-subgraphs of $F$. Note that for a fixed $F$, $\supp(F(\pi))$ is a finite set. Recall \eqref{eq:def_VH}. For $F\in\comgraphs$ and $G\in\supp(F(\pi))$, define the random subset of original $\msr$-vertices with original community graph $F$ that become isomorphic to $G$ under percolation:
\beq\label{eq:def_VFpiG} \vertices_{F(\pi)\isom G}^\msr := \{ a \in \vertices_F^\msr :\; \com_a(\pi) \isom G \}. \eeq
Note that
\beq\label{eq:perc_proof3} \sum_{G\in\supp(F(\pi))} \abs{\vertices_{F(\pi)\isom G}^\msr} = \abs{\vertices_F^\msr}, \eeq
with $\vertices_F^\msr$ defined as in \eqref{eq:def_VH}. 
Since percolation on different communities is independent, $(\abs{ \vertices_{F(\pi)\isom G}^\msr })_{G\in\supp(F(\pi))}$ has a multinomial distribution with number of trials $\abs{\vertices_F^\msr}$ and probability vector $(\P(F(\pi)\isom G))_{G\in\supp(F(\pi))}$. Thus by multinomial concentration and $\abs{\vertices_F^\msr}/M_n \to \mu_F$ (\cref{asmp:convergence} \cref{cond:limit_com}),
\beq\label{eq:multinom} M_n^{-1}\cdot \bigl( \abs{ \vertices_{F(\pi)\isom G}^\msr } \bigr)_{G\in\supp(F(\pi))} 
\toinp \mu_F \cdot \bigl(\P(F(\pi)\isom G)\bigr)_{G\in\supp(F(\pi))}. \eeq
In the following, we study how one or more copies of $H\in\comgraphs$ can be produced by percolating some $F\in\comgraphs$. 
We define the multiplicity of $H$ in any $G\in\graphs$, denoted by $\kappa(H \countin G) \geq 0$, as the number of distinct connected components in $G$ that are isomorphic to $H$. Note that there exists some (possibly more than one) $G\in\supp(F(\pi))$ such that $\kappa(H \countin G) \geq 1$ exactly when $H$ is isomorphic to an edge-subgraph of $F$. 
We compute, for arbitrary $H\in\comgraphs$,
\beq\label{eq:perc_com_pmf} M_n^{-1}\, \abs{\vertices_{H}^{\msr}(\pi)} 
= \sum_{F\in\comgraphs} 
\;\sum_{G\in\supp(F(\pi))} 
M_n^{-1} \,\abs{ \vertices_{F(\pi)\isom G}^\msr } \cdot \kappa(H \countin G). \eeq
We claim that 
\beq\label{eq:perc_com_lim} M_n^{-1}\, \abs{\vertices_{H}^\msr(\pi)} 
\toinp \sum_{F\in\comgraphs} 
\;\sum_{G\in\supp(F(\pi))} 
\mu_F \cdot \P(F(\pi)\isom G) \cdot \kappa(H \countin G) < \infty. \eeq
For convenience, we denote the inner sums
\beq \bal T_n^\pi(F,H) &:= \sum_{G\in\supp(F(\pi))} 
M_n^{-1} \,\abs{ \vertices_{F(\pi)\isom G}^\msr } \cdot \kappa(H \countin G), \\
T^\pi(F,H) &:= \sum_{G\in\supp(F(\pi))} 
\mu_F \cdot \P(F(\pi)\isom G) \cdot \kappa(H \countin G). \eal \eeq
Since $\supp(F(\pi))$ is a finite set, and $\kappa(H \countin G)$ are constants, by \eqref{eq:multinom} the linear combinations $T_n^\pi(F,H) \toinp T^\pi(F,H)$ as $n\to\infty$, for each $H,F\in\comgraphs$ and $\pi$. Next, we use a truncation argument to prove the convergence of the infinite sum over $F\in\comgraphs$. The same argument also reveals the rhs of \eqref{eq:perc_com_lim} to be finite.

Note that, as $\kappa(H \countin G)$ counts \emph{components} of $G$ that are isomorphic to $H$, $\kappa(H \countin G) \leq \abs{G}\big/\abs{H} = \abs{F}\big/\abs{H}$. Thus, for a fixed $F$, by \eqref{eq:perc_proof3} almost surely
\beq\label{eq:perc_com_tailbound1} T_n^\pi(F,H)
\leq \sum_{G\in\supp(F(\pi))} 
M_n^{-1}\, \abs{ \vertices_{F(\pi)\isom G}^\msr } \cdot \frac{\abs{F}}{\abs{H}} 
= M_n^{-1}\, \abs{\vertices_F^\msr} \cdot \frac{\abs{F}}{\abs{H}}
= \frac{\nn\mu_{F} \,\abs{F}}{\abs{H}}.
\eeq
By \cref{asmp:convergence} \cref{cond:mean_rdeg},
\beq\label{eq:perc_com_tailbound2} \sum_{F\in\comgraphs} \nn\mu_{F} \abs{F} = \E[D_n^\msr] \to \E[D^\msr] = \sum_{F\in\comgraphs} \mu_{F} \,\abs{F} < \infty. \eeq
Thus for arbitrary $\eps>0$, there exists $K=K(\eps)$ and $n_0=n_0(\eps)$ such that for all $n\geq n_0$,
\beq\label{eq:perc_com_tailbound3} \sum_{F\in \comgraphs, \abs{F}>K} \mu_F \,\abs{F} < \eps / 6, \qquad 
\sum_{F\in \comgraphs, \abs{F}>K} \nn\mu_{F} \,\abs{F} < \eps / 3. \eeq
Combining (\ref{eq:perc_com_tailbound1}-\ref{eq:perc_com_tailbound3}), we obtain that for $n\geq n_0$, almost surely
\beq\label{eq:perc_com_tailbound4} 0 \leq 
\sum_{\substack{F\in\comgraphs\\\abs{F}>K}} T_n^\pi(F,H)
\leq \sum_{\substack{F\in\comgraphs\\\abs{F}>K}} \frac{\nn\mu_{F}\,\abs{F}}{\abs{H}} 
< \eps/3. \eeq
Analogously, using \eqref{eq:perc_com_tailbound2} and the identity $\sum_{G\in\supp(F(\pi))} \P(F(\pi)\isom G) = 1$,
\beq\label{eq:perc_com_tailbound5} 0 \leq 
\sum_{\substack{F\in\comgraphs\\\abs{F}>K}} \!\! T^\pi(F,H) 
\leq \sum_{\substack{F\in\comgraphs\\\abs{F}>K}} 
\sum_{G\in\supp(F(\pi))} \!\!\!\!
\mu_F \,\P(F(\pi)\isom G) \frac{\abs{F}}{\abs{H}} 
= \sum_{\substack{F\in\comgraphs\\\abs{F}>K}} 
\mu_F \frac{\abs{F}}{\abs{H}} 
< \eps/6. \eeq
Note that the number of $F\in\comgraphs$ such that $\abs{F}\leq K$ is finite, thus by \eqref{eq:multinom}, it follows that the truncated sum converges in probability, i.e., as $n\to\infty$,
\beq\label{eq:perc_com_tailbound6} 
\P \biggl( \Bigl\lvert 
\sum_{\substack{F\in\comgraphs\\\abs{F}\leq K}} T_n^\pi(F,H) 
- \sum_{\substack{F\in\comgraphs\\\abs{F}\leq K}} T^\pi(F,H) 
\Bigr\rvert < \eps/2 \biggr) \to 1. \eeq
Combining (\ref{eq:perc_com_tailbound4}-\ref{eq:perc_com_tailbound6}) yields the convergence in probability claimed in \eqref{eq:perc_com_lim}. We see that the rhs of \eqref{eq:perc_com_lim} can be written as the sum of a finite sum $\sum_{F\in\comgraphs,\abs{F}\leq K} T^\pi(F,H)$ in \eqref{eq:perc_com_tailbound6} and a bounded quantity $\sum_{F\in\comgraphs,\abs{F}>K} T^\pi(F,H)$ in \eqref{eq:perc_com_tailbound5} and is thus finite.

Next, we study $M_n(\pi)/M_n$, that is the reciprocal of the second factor in \eqref{eq:perc_com_fraction}. Recall that for $G\in\graphs$, the number of connected components in $G$ is denoted by $c(G)$. Recall \eqref{eq:def_VFpiG} and compute
\beq M_n(\pi)/M_n 
= \sum_{F\in\comgraphs} \;\sum_{G\in\supp(F(\pi))} 
M_n^{-1} \, \abs{\vertices_{F(\pi)\isom G}^\msr} \cdot c(G). \eeq
We note the similarity between this formula and \eqref{eq:perc_com_pmf}, as well as $c(G) \leq \abs{G} = \abs{F}$. Thus, with analogous arguments and the same truncation as above, we conclude that
\beq\label{eq:perc_com_rescaling} M_n(\pi)/M_n 
\toinp \sum_{F\in\comgraphs} \;\sum_{G\in\supp(F(\pi))} 
\mu_F \cdot \P(F(\pi)\isom G) \cdot c(G) < \infty. \eeq
Combining \eqref{eq:perc_com_fraction}, \eqref{eq:perc_com_pmf} and \eqref{eq:perc_com_rescaling} yields 
\beq \bal \nn\mu_H(\pi)
&= \frac{\abs{\vertices_{H}^{\msr}(\pi)}}{M_n} \Big/ \frac{M_n(\pi)}{M_n} \\
&\toinp \frac{\sum_{F\in\comgraphs} \sum_{G\in\supp(F(\pi))} 
\mu_F \cdot \P(F(\pi)\isom G) \cdot \kappa(H \countin G)}
{\sum_{F\in\comgraphs} \sum_{G\in\supp(F(\pi))} 
\mu_F \cdot \P(F(\pi)\isom G) \cdot c(G)} =: \mu_H(\pi). \eal \eeq
This concludes the proof of \eqref{eq:comvect_perc_convinprob}. Note that community sizes cannot increase under percolation, thus $D_n^\msr(\pi) \preceq D_n^\msr$, and $(D_n^\msr)_{n\in\N}$ is tight (since it is also UI). Consequently, by $\P\bigl(D_n^\msr(\pi) \geq K\bigr) \leq \P\bigl( D_n^\msr \geq K \bigr)$, $(D_n^\msr(\pi))_{n\in\N}$ is also tight, thus 
$\sum_{H\in\comgraphs} \nn\mu_H(\pi) = 1$ for each $n\in\N$ implies that $\sum_{H\in\comgraphs} \mu_H(\pi) = 1$. 

Next, we prove \eqref{eq:comvect_perc_convinmean}. 
By definition, we compute
\beq \E\bigl[ D_n^\msr(\pi) \bcond \comvect(\pi) \bigr] 
= \frac{1}{M_n(\pi)} \sum_{a\in[M_n(\pi)]} \abs{\com_a^\pi} 
= \frac{\abs{\vertices(\comvect(\pi))}}{M_n(\pi)}. \eeq
Recall that, by the definition of $\comvect(\pi)$, $\abs{\vertices(\comvect(\pi))} = \abs{\vertices(\comvect)} = \he_n$. Thus,
\beq\bal &\E\bigl[ D_n^\msr(\pi) \bcond \comvect(\pi) \bigr] 
= \frac{\he_n}{M_n(\pi)} 
= \frac{\he_n}{M_n} \cdot \frac{M_n}{M_n(\pi)} 
= \E[D_n^\msr] \cdot \frac{M_n}{M_n(\pi)} \\
&\toinp \frac{\E[D^\msr]}{\sum_{F\in\comgraphs} \sum_{G\in\supp(F(\pi))} \mu_F \cdot \P(F(\pi)\isom G) \cdot c(G)} =: C < \infty, \eal\eeq
by \cref{asmp:convergence} \cref{cond:mean_rdeg} and \eqref{eq:perc_com_rescaling}. 
From \eqref{eq:comvect_perc_convinprob}, it follows that $D_n^\msr(\pi) \toindis D^\msr(\pi)$ as $n\to\infty$. Further, the stochastic domination $0 \leq D_n^\msr(\pi) \preceq D_n^\msr$ also holds conditionally on $\comvect(\pi)$. 
By \cref{asmp:convergence} \cref{cond:mean_rdeg}, 
$(D_n^\msr)_{n\in\N}$ is UI. For any fixed $K$, by stochastic domination, $\E\bigl[ D_n^\msr(\pi) \ind_{\{D_n^\msr(\pi) > K\}} \bigr] \leq \E\bigl[ D_n^\msr \ind_{\{D_n^\msr > K\}} \bigr]$, thus $(D_n^\msr(\pi))_{n\in\N}$ is also UI, and consequently $\E[D_n^\msr(\pi)] \to \E[D^\msr(\pi)] < \infty$. Noting that $\E\bigl[ D_n^\msr(\pi) \bcond \comvect(\pi) \bigr] \leq \E[D_n^\msr]$ by stochastic domination, and $(\E[D_n^\msr])_{n\in\N}$ is bounded, 
$\E\bigl[ D_n^\msr(\pi) \bcond \comvect(\pi) \bigr] \toinp C$ implies $\E\bigl[ \E[ D_n^\msr(\pi) \cond \comvect(\pi) ] \bigr] = \E[ D_n^\msr(\pi) ] \to C$. Since the limit is unique, we obtain that necessarily $C = \E[D^\msr(\pi)]$. 
This concludes the proof of \eqref{eq:comvect_perc_convinmean}, and consequently the proof of \cref{lem:comvect_perc_conv}.
\end{proof}

\subsubsection{Representation of size-biased community sizes}
\label{sss:sizebiased_comsize_representation}

\begin{proof}[Proof of \cref{claim:sizebiased_comsize}]
Recall that for $j\in\vertices(\com_a)$, $\comp^\msc(j,\pi)$ denotes its connected component under $\pi$-percolation on $\com_a$. 
Recall the percolated community list $\comvect(\pi) = (\com_a^\pi)_{a\in[M_n(\pi)]}$. Using the definition of the empirical distribution $D_n^\msr(\pi)$ and its transform by \eqref{eq:def_sizebiasing}, we compute the empirical mass function of $\wit{D_n^\msr(\pi)}$
\beq \bal \P\bigl( \wit{D_n^\msr(\pi)} = k \bcond \comvect(\pi)\bigr) 
&= \frac{(k+1)\, \P\bigl( D_n^\msr(\pi) = k+1 \bcond \comvect(\pi) \bigr)}{\E\bigl[ D_n^\msr(\pi) \bcond \comvect(\pi) \bigr]} \\
&= \frac{(k+1) \tfrac{1}{M_n(\pi)} \sum_{a\in[M_n(\pi)]} \ind_{\{\abs{\com_a^\pi} = k+1\}}}{\tfrac{1}{M_n(\pi)} \sum_{a\in[M_n(\pi)]} \abs{\com_a^\pi} }. \eal \eeq
Note that in the numerator, we can replace $(k+1) \cdot \ind_{\{\abs{\com_a^\pi} = k+1\}} = \abs{\com_a^\pi} \cdot \ind_{\{\abs{\com_a^\pi} = k+1\}}$. Also note that, by construction of $\comvect(\pi)$ (see \cref{ss:precolation_representation}), the sum in the denominator equals $\sum_{a\in[M_n(\pi)]} \abs{\com_a^\pi} = \abs{\vertices(\comvect)} = \sum_{a\in[M_n]} \abs{\com_a} = \he_n$, with $\he_n$ from \eqref{eq:def_halfedges}. Thus,
\beq \bal &\P\bigl( \wit{D_n^\msr(\pi)} = k \bcond \comvect(\pi)\bigr) 
= \frac{\sum_{a\in[M_n(\pi)]} \abs{\com_a^\pi} \cdot \ind_{\{\abs{\com_a^\pi} = k+1\}} }{\he_n} \\
&= \frac{1}{\he_n} \sum_{a\in[M_n(\pi)]} \sum_{j\in\vertices(\com_a^\pi)} \ind_{\{\abs{\com_a^\pi} = k+1\}} 
= \frac{1}{\he_n} \sum_{a\in[M_n(\pi)]} \sum_{j\in\vertices(\com_a^\pi)} \ind_{\{\abs{\comp^\msc(j,\pi)} = k+1\}}, \eal \eeq
where we have used that for $j\in\vertices(\com_a^\pi)$, its percolated component is exactly $\com_a^\pi$, thus $\abs{\comp^\msc(j,\pi)} = \abs{\com_a^\pi}$. Once again by $\cup_{a\in[M_n(\pi)]} \vertices(\com_a^\pi) = \vertices(\comvect)$,
\beq \begin{split} &\P\bigl( \wit{D_n^\msr(\pi)} = k \bcond \comvect(\pi)\bigr) 
= \frac{1}{\he_n} \sum_{j\in\vertices(\comvect)} \ind_{\{\abs{\comp^\msc(j,\pi)} = k+1\}} \\
&= \P\bigl( \abs{\comp^\msc(J_n,\pi)} = k+1 \bcond \comvect(\pi) \bigr) 
= \P\bigl( \abs{\comp^\msc(J_n,\pi)} - 1 = k \bcond \comvect(\pi) \bigr). \end{split} \eeq
It follows that $\P\bigl( \wit{D_n^\msr(\pi)} = k \bigr) = \P\bigl( \abs{\comp^\msc(J_n,\pi)} - 1 = k \bigr)$ for all $k$, which implies \eqref{eq:perc_comsize_representation}. 
This concludes the proof of \cref{claim:sizebiased_comsize}.
\end{proof}

\subsection{Proof of Proposition \ref{prop:pi_c_properties}}
\label{ss:pi_c_properties}
First, we prove \eqref{eq:pi_c_representation}. Recall that in \eqref{eq:def_pi_c}, we have identified $\pi_c = \inf \{\pi:\, \E[\wit D^\msl] \E[\wit{D^\msr(\pi)}] > 1 \}$. Thus, it is sufficient to show that
\beq\label{eq:pi_c_repres_reduction} \E\bigl[ \wit{D^\msr(\pi)} \bigr] 
= \E\bigl[ \abs{\rg} \,(\abs{\comp^\rg(U_\rg,\pi)}-1) \bigr] / \E[D^\msr], \eeq
where $\rg$ is a random graph with pmf $\bm\mu$, $U_\rg \cond \rg \distr \Unif[\rg]$, and $\E[\cdot]$ on the right hand side denotes total expectation (wrt the joint measure of $\rg$, $U_\rg$ and the percolation). 

To express its expectation, we analyze the distribution of $\wit{D^\msr(\pi)}$, using that $\wit{D_n^\msr(\pi)} \toindis \wit{D^\msr(\pi)}$ by \cref{lem:comvect_perc_conv} and $\abs{\comp^\msc(J_n,\pi)} - 1 \eqindis \wit{D_n^\msr(\pi)}$ by \cref{claim:sizebiased_comsize}. As before, $\comp^\msc(j,\pi)$ denotes the component of $j\in\com_a$ under $\pi$-percolation on $\com_a$, and $J_n\distr\Unif[\vertices(\comvect)]$, where $\vertices(\comvect)$ denotes the disjoint union of all vertices in community graphs. 

Note that $J_n$ is chosen uniformly at random among all community roles, which is equivalent to choosing a community in a \emph{size-biased} fashion, then choosing a uniform member of the chosen community. 
In the following, we use this observation to analyze the distribution of $\abs{\comp^\msc(J_n,\pi)}$. 
Recall from \cref{ss:RIGC_def} that all isomorphic community graphs are labeled in the same way and community roles in $H$ are described by a label $l\in[\abs{H}]$. 
Intuitively, for the distribution of $\comp^\msc(J_n,\pi)$, only the community graph and label of $J_n$ matters, thus we introduce the concept of \emph{type} to represent this pair. Recall that a community role $j\in\vertices(\comvect)$ that is in $\com_a$ and has label $l$ can be represented by the pair $(a,l)$, and define its type as $\type(j) := (\com_a,l)$. Using that $J_n \distr \Unif[\vertices(\comvect)]$, we compute the distribution of its random type $\type(J_n) =: (\rg_n^\star,I_n)$. Using \eqref{eq:def_mu} and that $\abs{\vertices(\comvect)} = \he_n = M_n \E[D_n^\msr]$ by \cref{rem:asmp_consequences} \cref{cond:partition_ratio}, 
\beq\label{eq:randomtype_graph} \bal \P\bigl( \rg_n^\star = H \bigr) 
= \P\bigl( J_n \in \cup_{a\in\rightpar_H} \vertices(\com_a) \bigr) 
= \frac{\abs{\rightpar_H} \cdot \abs{H}}{\abs{\vertices(\comvect)}} 
= \frac{\abs{\rightpar_H} \cdot \abs{H}}{M_n \E[D_n^\msr]} 
= \frac{\nn\mu_H \cdot \abs{H}}{\E[D_n^\msr]} 
=: \mu_H^{{\sss(n)},\star}. \eal\eeq
Indeed, as intuition suggests, $\rg_n^\star$ is chosen in a size-biased fashion. It is also intuitive that conditionally on $\rg_n^\star$, $I_n$ is uniform on $[\abs{\rg_n^\star}]$. Noting that in each community graph $\com_a \isom H$ there is one vertex with label $l$ and $\abs{H}$ vertices in total, we indeed obtain
\beq\label{eq:randomtype_label} \P\bigl( I_n = l \bcond \rg_n^\star = H \bigr) 
= \frac{\abs{\rightpar_H} \cdot 1}{\abs{\rightpar_H} \cdot \abs{H}} 
= \frac{1}{\abs{H}}. \eeq
By \cref{asmp:convergence} \cref{cond:limit_com}and \cref{cond:mean_rdeg}, the joint mass function converges:
\beq \P\bigl( (\rg_n^\star,I_n) = (H,l) \bigr) 
= \frac{\nn \mu_H \cdot \abs{H}}{\E[D_n^\msr]} \frac{1}{\abs{H}}
\to \frac{\mu_H \cdot \abs{H}}{\E[D^\msr]} \frac{1}{\abs{H}}
=: \P\bigl( (\rg^\star,I) = (H,l) \bigr), \eeq
so that $(\rg_n^\star,I_n) \toindis (\rg^\star,I)$. We now return to studying $\abs{\comp^\msc(J_n,\pi)}$. Let us we write $\comp^\msc(j,\pi)$ in terms of $\type(j) = (H,l)$ as $\comp^H(l,\pi) \eqindis \comp^\msc(j,\pi)$ (see \cref{prop:pi_c_properties}), and consider $\comp^\msc(J_n,\pi) \eqindis \comp^{\rg_n^\star}(I_n,\pi)$ as a \emph{mixture}. We have shown that the mixing variable $(\rg_n^\star,I_n) \toindis (\rg^\star,I)$, and since it has countably many values, this implies\footnote{The pointwise difference of mass functions of the mixtures can be bounded in terms of the total variation distance of the mixing variables.} that the mixture also converges in distribution: $\abs{\comp^\msc(J_n,\pi)} \eqindis \abs{\comp^{\rg_n^\star}(I_n,\pi)} \toindis \abs{\comp^{\rg^\star}(I,\pi)}$. Recall that $\wit{D_n^\msr(\pi)} + 1 \eqindis \abs{\comp^\msc(J_n,\pi)}$ by \cref{claim:sizebiased_comsize} and $\wit{D_n^\msr(\pi)} \toindis \wit{D^\msr(\pi)}$ by \cref{lem:comvect_perc_conv}. Necessarily, by the uniqueness of limit, $\wit{D^\msr(\pi)} + 1 \eqindis \abs{\comp^{\rg^\star}(I,\pi)}$. We compute 
\beq\label{eq:perc_new_proof2} \bal \E\bigl[ \wit{D^\msr(\pi)} \bigr] 
&= \E\bigl[ \abs{\comp^{\rg^\star}(I,\pi)} -1 \bigr] 
= \E\Bigl[ \E\bigl[ \abs{\comp^{\rg^\star}(I,\pi)} - 1 \bcond (\rg^\star,I) \bigr] \Bigr] \\
&= \sum_{H\in\comgraphs} \sum_{i\in[\abs{H}]} \E\bigl[ \abs{\comp^H(i,\pi)} - 1 \bigr] \frac{1}{\abs{H}} \frac{\mu_H \abs{H}}{\E[D^\msr]} \\
&= \sum_{H\in\comgraphs} \mu_H \abs{H} 
\biggl( \frac{1}{\abs{H}} \sum_{i\in[\abs{H}]} \E\bigl[ \abs{\comp^H(i,\pi)} - 1 \bigr] \biggr)  \Big/ \E[D^\msr]. \eal \eeq
Recall that $\rg$ denotes a random graph with pmf $\bm\mu$ and $U_\rg \cond \rg \distr \Unif[\rg]$; in particular, $U_H \distr \Unif[H]$ for fixed $H\in\comgraphs$. Then
\beq \bal \E[\wit{D^\msr(\pi)}] 
&= \sum_{H\in\comgraphs} \mu_H \abs{H} \cdot \E\bigl[ \abs{\comp^H(U_H,\pi)} - 1 \bigr] \big/ \E[D^\msr] \\
&= \E\Bigl[ \abs{\rg} \cdot \E\bigl[\abs{\comp^\rg(U_\rg,\pi)} - 1 \bcond \rg \bigr] \Bigr] \Big/ \E[D^\msr], \eal \eeq
which is equivalent to \eqref{eq:pi_c_repres_reduction} 
by the tower property of conditional expectation. This concludes the proof of \eqref{eq:pi_c_representation}. 

Next, we prove that $\pi_c<1$ by showing that there exists some $\pi<1$ such that $\E[\wit{D^\msr(\pi)}] > 1/\E[\wit D^\msl]$, which is equivalent to \eqref{eq:supercrit_perc}. 
Using \eqref{eq:perc_new_proof2}, for some $K\in\Z^+$ to be specified later, 
we bound
\beq\label{eq:perc_new_proof3} \E[\wit{D^\msr(\pi)}] 
\geq \sum_{\substack{H\in\comgraphs\\\abs{H}\leq K}} \frac{\mu_H}{\E[D^\msr]} \sum_{i\in[\abs{H}]} \E\bigl[ \abs{\comp^H(i,\pi)} -1 \bigr] =: S_{\leq K}(\pi). \eeq
We show that, for appropriately chosen $K$ and $\pi$, $S_{\leq K}(\pi) > 1/\E[\wit D^\msl]$ by comparing both to $S_{\leq K}(1)$ as an intermediate step. In fact, noting that $\comp^H(i,1)$ is the unpercolated component of $i$, i.e., the entire graph $H$, we have
\beq \begin{split} 
S_{\leq K}(1) 
&= \sum_{\substack{H\in\comgraphs\\\abs{H}\leq K}} \frac{\mu_H}{\E[D^\msr]} \sum_{i\in[\abs{H}]} \E\bigl[ \abs{\comp^H(i,1)} -1 \bigr] 
= \sum_{\substack{H\in\comgraphs\\\abs{H}\leq K}} \frac{\mu_H}{\E[D^\msr]} \sum_{i\in[\abs{H}]} (\abs{H} - 1) \\
&= \sum_{\substack{H\in\comgraphs\\\abs{H}\leq K}} \frac{\mu_H}{\E[D^\msr]} \abs{H} (\abs{H} - 1) 
= \frac{\E\bigl[ \abs{\rg}(\abs{\rg}-1) \ind_{\{\abs{\rg}\leq K\}}\bigr]}{\E[D^\msr]}, \end{split} \eeq
where $\rg$ denotes a random graph with pmf $\bm\mu$. Then, by \cref{asmp:convergence} \cref{cond:limit_rdeg} and \eqref{eq:def_sizebiasing},
\beq 
S_{\leq K}(1) 
= \frac{\E\bigl[ D^\msr(D^\msr-1) \ind_{\{D^\msr\leq K\}}\bigr]}{\E[D^\msr]}
= \E\bigl[ \wit D^\msr \cdot \ind_{\{\wit D^\msr < K\}} \bigr]. \eeq
Clearly, the partial sums $S_{\leq K}(1) \to \E[\wit D^\msr]$ as $K\to\infty$. 
By the supercriticality condition \cref{cond:supercrit}, we know that $\E[\wit D^\msr] > 1/\E[\wit D^\msl]$, thus 
there exists $K$ large enough so that $S_{\leq K}(1) > 1/\E[\wit D^\msl]$. We fix such a $K$. 
Next, we compare $S_{\leq K}(\pi)$ and $S_{\leq K}(1)$. Note that $S_{\leq K}(\pi)$ is a \emph{finite} sum, as $\{H\in\comgraphs:\, \abs{H}\leq K\}$ is a finite set. Further, for any fixed $(H,i)$, $\E\bigl[ \abs{\comp^H(i,\pi)} \bigr]$ is a polynomial in $\pi$, thus it is \emph{continuous}. Consequently, $\pi \mapsto S_{\leq K}(\pi)$ is continuous, and by $S_{\leq K}(1) > 1/\E[\wit D^\msl]$, we can choose $\pi<1$ sufficiently close to $1$ so that $S_{\leq K}(\pi) > 1/\E[\wit D^\msl]$. We have thus shown that for some $K\in\Z^+$ and $\pi<1$, as chosen above,
\beq \E[\wit{D^\msr(\pi)}] \geq S_{\leq K}(\pi) > 1/\E[\wit D^\msl], \eeq
and we conclude that indeed $\pi_c < 1$. This concludes the proof of \cref{prop:pi_c_properties}.
\proofends

\section{Further properties of the giant of the {\rm RIGC}}
\label{s:local_global}
In this section, we explore the relation between ``local'' and ``global'' properties of the $\RIGC$ and the (underlying) $\BCM$. 
In \cite{vdHKomVad18locArxiv}, we have shown that neighborhoods in the $\BCM$ and $\RIGC$ are respectively well-approximated by a branching process (that we recall shortly) and an appropriate transform of it (see \cite[Section 5.1]{vdHKomVad18locArxiv}). 
In \cref{s:proof_giant}, we have analyzed the largest component of each model using the continuous-time exploration introduced in \cref{ss:exploration}. As for the configuration model \cite{MolRee95,MolRee98} or the Erd\H{o}s-R\'{e}nyi random graph \cite{ErdosRenyi60evolution}, we show below that also in the $\BCM$ and $\RIGC$, the proportion of the giant component is given by the survival probability of the branching process approximating the neighborhoods. However, it is not entirely trivial to prove results on the largest component using \emph{solely} the branching process approximation, which is why we turn to the continuous-time exploration for an independent proof, and make the connection afterwards. This connection then allows us to derive ``local'' properties of the giant component; we present the degree distribution within the giant and as a (not entirely trivial\footnote{The difficulty lies in identifying when the obtained degree distribution has finite or infinite mean.}) consequence, the number of edges in the giant, however the technique is more generally applicable.

\subsection{The approximating BPs}
\label{ss:BPapprox}
In this section, we describe the branching processes (see e.g.\ \cite{athreya-ney} for an introduction on branching processes) related to the local weak limit of the $\BCM$ and the $\RIGC$ (intuitively speaking, the local weak limit describes the neighborhood of a typical vertex; for a precise definition, see \cite[Definition 2.7 and Section 3.1]{vdHKomVad18locArxiv}). Further, we show how these branching processes ($\BP$s) are related to the constants $\eta_\msl$, $\eta_\msr$, $\xi_\msl$ and $\xi_\msr$ (defined in \cref{thm:giantcomp,thm:degrees_giant,cor:BCM_giantcomp}) describing the size and properties of the giant component of the $\BCM$ and $\RIGC$.

\smallskip
\paragraph*{\textbf{Description of the approximating BPs}}
We first describe the discrete-time branching process $\BP_\msl$ defined in \cite[Section 4.1]{vdHKomVad18locArxiv} that approximates the neighborhood of $\msl$-vertices in the underlying $\BCM$. Recall $D^\msl$ and $D^\msr$ from \cref{asmp:convergence} \cref{cond:limit_ldeg} and \cref{cond:limit_rdeg} respectively, and further recall \eqref{eq:def_sizebiasing}. We start with a single root in generation $0$ that produces offspring distributed as $D^\msl$. Every other individual in an even generation has offspring distributed as $\wit D^\msl$, while every individual in an odd generation has offspring distributed as $\wit D^\msr$. The offspring of any two individuals are independent. In \cite[Section 4.1]{vdHKomVad18locArxiv}, the branching process $\BP_\msr$ that approximates the neighborhood of $\msr$-vertices in the underlying $\BCM$ is defined analogously by reversing the roles of $\msl$ and $\msr$.

In the following, we show the relation between the supercriticality of the $\BCM$ and the supercriticality of the approximating $\BP$s. Just like in \cref{thm:giantcomp,thm:BCM_giantcomp}, we exclude the special case $\P(D^\msl = 2) = \P(D^\msr = 2) = 1$, or equivalently, $\P(\wit D^\msl = 1) = \P(\wit D^\msr = 1) = 1$. The reason is that in this special case for both $\BP_\msl$ and $\BP_\msr$ exactly one offspring is guaranteed in each generation (except the root that has offspring $2$), and thus both $\BP$s are guaranteed infinite survival\footnote{For an individual in the process, the survival event means producing an infinite (sub)tree rooted at the vertex, for the BP, the survival event means the survival of the root as an individual.} in an \emph{atypical} way. Thus we exclude this special case from now on.

Next, we show that $\BP_\msl$ and $\BP_\msr$ are supercritical exactly when the supercriticality condition \cref{cond:supercrit} of the graph holds; and the survival probabilities are $\xi_\msl$ (from \cref{thm:giantcomp}) and $\xi_\msr$ (from \cref{cor:BCM_giantcomp}), respectively. We prove the statement for $\BP_\msl$ first. Consider the subprocess $\wih\BP_\msl$ formed by the descendants in odd generations of the first child of the root, so that offspring in $\wih\BP_\msl$ is defined as grandchildren in $\BP_\msl$. Consequently, $\wih\BP_\msl$ is a Galton-Watson process, with offspring distribution $N^\msr \eqindis \sum_{i=1}^{\wit D^\msr} \wit D_{(i)}^\msl$, where $\wit D_{(i)}^\msl$ are iid copies of $\wit D^\msl$ and are independent from $\wit D^\msr$. Note that $N^\msr$ has been introduced in \eqref{eq:def_Nr} to make sense of the \emph{composite generating function} $G_{N^\msr}(z) = G_{\wit D^\msr}\bigl(G_{\wit D^\msl}(z)\bigr)$. We know from branching process literature (see e.g.\ \cite{athreya-ney}) that 
the extinction probability of $\wih\BP_\msl$ is the smallest fixed point of $G_{N^\msr}$, defined as $\eta_\msl$ in \eqref{eq:fixedpoint}. Further, $\eta_\msl < 1$ exactly when $\E[N^\msr] = \E[\wit D^\msr] \E[\wit D^\msl] > 1$ (by Wald's identity), i.e., when \cref{cond:supercrit} holds. 
Each child of the root in $\BP_\msl$ survives exactly when their corresponding $\wih \BP_\msl$ survives. Using that the root of $\BP_\msl$ produces offspring distributed as $D^\msl$, the survival probability of $\BP_\msl$ is $1-G_{D^\msl}(\eta_\msl)$, defined as $\xi_\msl$ in \cref{thm:giantcomp}. By properties of $G_{D^\msl}$, $\xi_\msl > 0$ exactly when $\eta_\msl < 1$, that is, under the supercriticality condition \cref{cond:supercrit}.

\smallskip
\paragraph*{\textbf{Left-right correspondence}}
We define the analogous Galton-Watson process $\wih\BP_\msr$ (consisting of odd generations in $\BP_\msr$) with offspring distribution $N^\msl \eqindis \sum_{i=1}^{\wit D^\msl} \wit D_{(i)}^\msr$, where $\wit D_{(i)}^\msr$ are iid random variables with distribution $\wit D^\msr$ and independent from $\wit D^\msl$. This process is supercritical when $\E[N^\msl] = \E[\wit D^\msr] \E[\wit D^\msl] > 1$, where we again recognize the supercriticality condition \cref{cond:supercrit}. Thus the extinction probability of $\wih\BP_\msr$ is the smallest solution to the fixed point equation
\beq\label{eq:fixedpoint_BPr} G_{\wit D^\msl}\bigl( G_{\wit D^\msr}(z)\bigr) = z. \eeq
We now verify that the smallest solution is $\eta_\msr := G_{\wit D^\msl}(\eta_\msl)$ (defined in \cref{thm:degrees_giant}). Applying $G_{\wit D^\msl}$ to both sides of \eqref{eq:fixedpoint} implies that $\eta_\msr$ indeed satisfies \eqref{eq:fixedpoint_BPr}. Further, when \cref{cond:supercrit} holds, we have that $\eta_\msl < 1$, which implies $\eta_\msr < 1$. 
Analogously with $\xi_\msl$, $\xi_\msr = 1 - G_{D^\msr}(\eta_\msr)$ is indeed the survival probability of $\BP_\msr$, and $\xi_\msr > 0$ exactly when $\eta_\msr < 1$, that is, under the supercriticality condition \cref{cond:supercrit}. We also note that applying $G_{\wit D^\msr}$ to both sides of \eqref{eq:fixedpoint_BPr} yields $\eta_\msl = G_{\wit D^\msr}(\eta_\msr)$, showing that the roles of lhs and rhs are indeed symmetric, however the corresponding quantities are generally \emph{not equal}.

\subsection{The relation of local and global properties and degrees in the giant component}
\label{ss:local_global_rel}
Thinking of $\xi_\msl$ in \eqref{eq:leftsize} and $\xi_\msr$ in \eqref{eq:rightsize} as the probability for an $\msl$-vertex and $\msr$-vertex respectively to be in the giant, we find that the same quantities equal the probabilities of survival of the respective branching processes $\BP_\msl$ and $\BP_\msr$. Thus, we have established the ``asymptotic equivalence'', as formalized below, of the following events: a vertex being in the giant and the survival of the corresponding branching process.

For two sets $A$ and $B$, let $A \symmdiff B := (A\setminus B)\cup(B\setminus A)$ denote their symmetric difference. For $v\in\leftpar$, denote by $\comp(v)$ the connected component of $v$ in the $\RIGC$. For $K\in\Z^+$, introduce the set 
\beq \SZ_{\geq K} := \bigl\{ v\in\leftpar:\, \abs{\comp(v)} \geq K \bigr\}.
\eeq

\begin{lemma}[Relation of the $\BP$-approximation and the giant component]
\label{lem:loc_glob_relation}
Consider the \allowbreak$\RIGC(\bitd^\msl,\comvect)$ under \cref{asmp:convergence}. Then, for any $\eps>0$,
\beq\label{eq:equiv_bp_giant} \lim_{K\to\infty} \lim_{n\to\infty} \P\bigl( N_n^{-1}\, 
\abs{ \SZ_{\geq K} \symmdiff \comp_1} > \eps \bigr) = 0. \eeq
\end{lemma}

This notion is closely related to convergence in probability. 

\begin{proof}
We first prove the lemma assuming the supercriticality condition \eqref{cond:supercrit}. Note that, for any $K$ fixed, for $n$ large enough $\comp_1 \subseteq \SZ_{\geq K}$ whp, since $\abs{\comp_1} \asymp \xi_\msl N_n + o_{\sss\P}(N_n) \geq K$ whp for $n$ large enough, since under the supercriticality condition \eqref{cond:supercrit}, $\xi_\msl>0$. Thus $\abs{\SZ_{\geq K} \symmdiff \comp_1} = \abs{\SZ_{\geq K} \setminus \comp_1} = \abs{\SZ_{\geq K}} - \abs{\comp_1}$ whp, and \eqref{eq:equiv_bp_giant} is equivalent to
\beq\label{eq:equiv_bp_giant2} \lim_{K\to\infty} \lim_{n\to\infty} \P\bigl( N_n^{-1}\, 
\abs{\SZ_{\geq K}} - N_n^{-1}\,\abs{\comp_1} > \eps \bigr) = 0. \eeq
By \cref{thm:giantcomp}, as $n\to\infty$, independently of $K$
\beq\label{eq:equiv_bp_giant3} \P\bigl( \bigl\lvert N_n^{-1}\, \abs{\comp_1} - \xi_\msl \bigr\rvert > \eps/2 \bigr) \to 0. \eeq 
We denote by $(\rCP,\groot)$, introduced in detail in \cite[Section 5.1]{vdHKomVad18locArxiv}, the local weak limit of the $\RIGC$, i.e., the random graph $\rCP$ with root $\groot$ that approximates neighborhoods in the $\RIGC$. 
By the triangle inequality, 
\begin{subequations}
\label{eq:loc_glob0}
\begin{align} 
\nonumber &\lim_{K\to\infty} \lim_{n\to\infty} \P\bigl( \bigl\lvert N_n^{-1}\, \abs{\SZ_{\geq K}} - \xi_\msl \rvert > \eps/2 \bigr) \\
\label{eq:loc_glob1} &\leq \lim_{K\to\infty} \lim_{n\to\infty} \P\Bigl( \bigl\lvert N_n^{-1}\, \abs{\SZ_{\geq K}} - \P\bigl( \abs{\rCP} \geq K \bigr) \bigr\rvert > \eps/4 \Bigr) \\
\label{eq:loc_glob2} &\phantom{{}={}}+ \lim_{K\to\infty} \lim_{n\to\infty}
\P\Bigl( \bigl\lvert \P\bigl( \abs{\rCP} \geq K \bigr) - \xi_\msl \bigr\rvert > \eps/4 \Bigr) .
\end{align}
\end{subequations}
By the local weak convergence stated in \cite[Theorem 2.8]{vdHKomVad18locArxiv}, the inner limit in \eqref{eq:loc_glob1} equals $0$ for any fixed $K$, thus \eqref{eq:loc_glob1} equals $0$. Note that the limit in $n$ in \eqref{eq:loc_glob2} can be omitted, and also note that the survival probability is $\xi_\msl = \P(\abs{\BP_\msl} = \infty)$. From the construction of $\rCP$ in \cite[Section 5.1]{vdHKomVad18locArxiv} using $\BP_\msl$, we know that $\P(\abs{\rCP} = \infty) = \P(\abs{\BP_\msl} = \infty)$. As $\P(\abs{\rCP} \geq K) \to\P(\abs{\rCP} = \infty) = \xi_\msl$, \eqref{eq:loc_glob2} equals $0$ as well. Combining this with (\ref{eq:equiv_bp_giant2}-\ref{eq:equiv_bp_giant3}) through the triangle inequality yields \eqref{eq:equiv_bp_giant}, 
concluding the proof in the case when the supercriticality condition \eqref{cond:supercrit} holds.

When \eqref{cond:supercrit} does not hold, $\xi_\msl = 0$. In this case, the first argument (that $\comp_1 \subseteq \SZ_{\geq K}$ whp) does not hold, instead we can bound
\beq N_n^{-1} \abs{ \SZ_{\geq K} \symmdiff \comp_1 } 
\leq N_n^{-1} \abs{\SZ_{\geq K}} + N_n^{-1} \abs{\comp_1}. \eeq
The rest of the argument is analogous. By \cref{thm:giantcomp}, $\abs{\comp_1}/N_n \toinp \xi_\msl = 0$, and $\bigl\lvert N_n^{-1} \abs{\SZ_{\geq K}} - \xi_\msl \bigr\rvert = N_n^{-1} \abs{\SZ_{\geq K}}$ can still be bound by (\ref{eq:loc_glob1}-\ref{eq:loc_glob2}). We conclude that \eqref{eq:equiv_bp_giant} holds in this case as well. 
This concludes the proof of \cref{lem:loc_glob_relation}.
\end{proof}

The above equivalence can be extended to give a heuristic interpretation for further results, namely, the formulas in \cref{thm:BCM_giantcomp,cor:BCM_giantcomp}. Knowing the extinction probability of a child of the root and the degree of the root results in \eqref{eq:deg_k} and \eqref{eq:deg_k_right}. For \eqref{eq:edges}, note that choosing an (instance of an) edge uar \emph{in the random graph} is equivalent to picking the comprising $\msl$-and $\msr$-half-edges uar independently. Then the two endpoints can be viewed as a child of the root in $\BP_\msl$ and $\BP_\msr$, respectively, and at least one of them has to survive, which has probability $1 - \eta_\msl \eta_\msr$. 
In the following, we prove our results on degrees and edges 
in the giant as a consequence of local weak convergence and \cref{thm:giantcomp}.

\begin{proof}[Proof of \cref{thm:degrees_giant}]
As we recalled above from \cite[Section 5.1]{vdHKomVad18locArxiv}, we denote by $(\rCP,\groot)$ the local weak limit of the $\RIGC$, i.e., the random graph $\rCP$ with root $\groot$ that approximates neighborhoods in the $\RIGC$. Denote by $\deg(\groot)$ the degree of $\groot$ in $\rCP$, and by $\bdeg(\groot)$ the number of communities that $\groot$ is part of in $\rCP$. 
Let $V_n^\msl \distr \Unif[\leftpar]$ and note that
\beq N_n^{-1}\, \abs{\leftpar_k \cap \vertices_d^\msp \cap \comp_1} = \P\bigl( \pdeg(V_n^\msl)=d,\, \ldeg(V_n^\msl)=k,\, V_n^\msl \in \comp_1 \bigr). \eeq
Intuitively, as $(\rCP,\groot)$ approximates $(\RIGC,V_n^\msl)$, the limit of the above quantity must be $\P\bigl( \bdeg(\groot)=k,{\deg(\groot)}=d,\lvert{\vertices(\rCP)}\rvert = \infty \bigr)$. In the following, we prove this formally. 
Recall \eqref{eq:def_deggiant_distr} and that $\nu(c \countin H)$ denotes the number of vertices in $H$ with $\msc$-degree $c$. From the construction of $(\rCP,0)$ in \cite[Section 5.1]{vdHKomVad18locArxiv}, it is straightforward to see (by conditioning on the community graphs $H_1,\ldots,H_k$ assigned to the $k$ communities that the root $\groot$ is part of) that 
\beq\label{eq:Akd_expand} \bal
&\P\bigl( \bdeg(\groot)=k,{\deg(\groot)}=d,\lvert{\vertices(\rCP)}\rvert = \infty \bigr) \\
&= p_k 
\sum_{H_1,\ldots,H_k\in\comgraphs} \sum_{\substack{c_1,\ldots,c_k\in\Z^+\\ c_1+\ldots+c_k=d}} 
\Bigl( 1 - \eta_\msr^{ \sum_{i=1}^k (\abs{H_i}-1)} \Bigr)
\prod_{i=1}^k \frac{\nu(c_i \countin H_i) \mu_{H_i}}{\E[D^\msr]}
= A(k,d). \eal \eeq
For convenience, denote, with $K\in\Z^+$,
\beq\label{eq:def_AkdK} 
A_K(k,d) := \P\bigl( \bdeg(\groot)=k,{\deg(\groot)}=d,\lvert{\vertices(\rCP)}\rvert \geq K \bigr). \eeq
Keep in mind that our goal is to prove \eqref{eq:deggiant_distr}, i.e., that $N_n^{-1} \abs{\leftpar_k\cap\vertices_d^\msp\cap\comp_1} \toinp A(k,d)$. 
We compute
\begin{subequations}
\begin{align} 
\nonumber &\lim_{n\to\infty} \P\Bigl( \bigl\lvert N_n^{-1}\, \abs{\leftpar_k \cap \vertices_d^\msp \cap \comp_1} 
- A(k,d) \bigr\rvert > \eps \Bigr) \\
\label{eq:loc_glob3} &\leq \lim_{K\to\infty} \lim_{n\to\infty} \P\Bigl( \bigl\lvert N_n^{-1}\, \abs{\leftpar_k \cap \vertices_d^\msp \cap \comp_1} 
- N_n^{-1}\, \abs{\leftpar_k \cap \vertices_d^\msp \cap \SZ_{\geq K}}  \bigr\rvert > \eps/3 \Bigr) \\
\label{eq:loc_glob4} &\phantom{{}\leq{}}+ \lim_{K\to\infty} \lim_{n\to\infty} \P\Bigl( \bigl\lvert N_n^{-1}\, \abs{\leftpar_k \cap \vertices_d^\msp \cap \SZ_{\geq K}} 
- A_K(k,d) \bigr\rvert > \eps/3 \Bigr) \\
\label{eq:loc_glob5} &\phantom{{}\leq{}}+ \lim_{K\to\infty} \lim_{n\to\infty} \P\Bigl( \bigl\lvert A_K(k,d) - A(k,d) \bigr\rvert > \eps/3 \Bigr). 
\end{align}
\end{subequations}
Clearly, $A_K(k,d) \to A(k,d)$ as $K\to\infty$ (compare \eqref{eq:Akd_expand} and \eqref{eq:def_AkdK}), thus the double limit in \eqref{eq:loc_glob5} equals $0$. Next, we look at \eqref{eq:loc_glob4} and note that $N_n^{-1}\, \abs{\leftpar_k \cap \vertices_d^\msp \cap \SZ_{\geq K}} = \P\bigl( \pdeg(V_n^\msl)=d,\, \ldeg(V_n^\msl)=k,\, V_n^\msl \in \SZ_{\geq K} \bigr)$. Thus for $K$ fixed, by \eqref{eq:def_AkdK} and the local weak convergence stated in \cite[Theorem 2.8]{vdHKomVad18locArxiv}, the inner limit in \eqref{eq:loc_glob4} is $0$, thus \eqref{eq:loc_glob4} equals $0$. 
Removing some of the conditions, we can bound \eqref{eq:loc_glob3} as
\beq \bal
&\P\Bigl( \bigl\lvert N_n^{-1}\, \lvert \leftpar_k \cap \vertices_d^\msp \cap \comp_1 \rvert
- N_n^{-1}\, \lvert \leftpar_k \cap \vertices_d^\msp \cap {\SZ_{\geq K}} \rvert \bigr\rvert > \eps/3 \Bigr) \\
&\leq {\P\Bigl( N_n^{-1} \lvert \leftpar_k \cap \vertices_d^\msp \cap (\comp_1 \symmdiff \SZ_{\geq K}) \rvert \geq \eps/3 \Bigr)} \\
&\leq \P\Bigl( {N_n^{-1}}\, {\abs{\comp_1 \symmdiff \SZ_{\geq K}}} > \eps/3 \Bigr), \eal \eeq
which tends to $0$ as first $n\to\infty$ followed by $K\to\infty$, by \cref{lem:loc_glob_relation}. Thus \eqref{eq:loc_glob3} is also $0$, and combining everything\hyphenation{ev-ery-thing} above, indeed $N_n^{-1}\,\abs{\leftpar_k\cap\vertices_d^\msp\cap\comp_1} \toinp A(k,d)$, that is, \eqref{eq:deggiant_distr} holds. This concludes the proof of \cref{thm:degrees_giant}.
\end{proof}

\appendix
\section{Edges in the giant component of the {\rm RIGC}}
\label{apdx:edges}

In this section, we provide the details for the sketch proof of \cref{thm:edgesRIGC} from \cref{ss:results_giant_RIGC}, by proving both \cref{lem:UI_equiv} and the remaining statements \eqref{eq:edgesconv_Akd} and \eqref{eq:equiv_edgesformulas}. 

Recall $D_n^\msc$ with pmf $\nn{\bm\varrho}$ from \eqref{eq:def_rho} and its limit $D^\msc$ with pmf $\bm\varrho$ from \cref{rem:asmp_consequences} \cref{cond:limit_cdeg}. Denote a uniform $\msl$-vertex $V_n^\msl \distr \Unif[\leftpar]$ and its (projected) degree (see \eqref{eq:def_deg}) $D_n^\msp = \pdeg(V_n^\msl)$. In \cite[Corollary 2.9]{vdHKomVad18locArxiv}, the distributional limit of $D_n^\msp$ is established as $D^\msp \eqindis \sum_{i=1}^{D^\msl} D_{(i)}^\msc$, where $D_{(i)}^\msc$ are iid copies of $D^\msc$ independently of $D^\msl$. Let $\rg_n$ and $\rg$ denote random graphs with pmfs $\nn{\bm\mu}$ from \eqref{eq:def_mu} and $\bm\mu$ from \cref{asmp:convergence} \cref{cond:limit_com}, respectively.

\subsection{Proof of Claim \ref{lem:UI_equiv}}
\label{apxs:UI}

We prove \cref{lem:UI_equiv} by showing that statement \eqref{UI:comdeg} is equivalent to both statement \eqref{UI:deg} and statement \eqref{UI:edges}.

\begin{proof}[Proof of \cref{lem:UI_equiv} \eqref{UI:comdeg} implies \eqref{UI:deg}]
We prove that $(D_n^\msp)_{n\in\N}$ is uniformly integrable by showing that for any $\eps > 0$, $\E[D_n^\msp \ind_{\{D_n^\msp \geq K\}}] < \eps$ for $K$ large enough, uniformly in $n$. Recall that $\vertices(\comvect)$ denotes the disjoint union of all vertices in community graphs and that for $j\in\vertices(H)$, $d_j^\msc$ denotes the degree of $j$ within $H$. Recall that $v \comrole j$ denotes the event that $j\in\vertices(\comvect)$ is one of the community roles assigned to $v\in\leftpar$, so that we can write the \emph{random} degree of $v$ (see \eqref{eq:def_deg}), depending on the bipartite matching, as $d_v^\msp = \sum_{j\in\vertices(\comvect)} d_j^\msc \ind_{\{v \comrole j\}}$. We compute, for $K\in\N$, by taking the empirical average first,
\beq \bal \E\bigl[ D_n^\msp \ind_{\{D_n^\msp \geq K\}} \bigr] 
&= \E\Bigl[ N_n^{-1} \sum_{v\in [N_n]} 
\sum_{j\in\vertices(\comvect)} d_j^\msc \cdot 
\ind_{\{v \comrole j\}} \ind_{\{d_v^\msp \geq K\}} \Bigr] \\
&= N_n^{-1} \sum_{j\in\vertices(\comvect)} d_j^\msc \cdot
\E\Bigl[ \sum_{v\in [N_n]} \ind_{\{v \comrole j\}} \ind_{\{d_v^\msp \geq K\}} \Bigr] \\
&= N_n^{-1} \sum_{j\in\vertices(\comvect)} d_j^\msc \sum_{v\in [N_n]} \P\bigl( v \comrole j, d_v^\msp \geq K \bigr).
\eal \eeq
We further rewrite the probability as
\beq \P\bigl( v \comrole j, d_v^\msp \geq K \bigr) 
= \P\bigl( d_v^\msp \geq K \bcond v \comrole j \bigr) \cdot \P(v \comrole j). \eeq
Note that $\P( v \comrole j ) = d_v^\msl / \he_n$. We now split the sum over $j$ according to whether $d_j^\msc$ is smaller or larger than $\sqrt{K}$. For $d_j^\msc \geq \sqrt{K}$, we use the trivial bound $\P\bigl( d_v^\msp \geq K \bcond v \comrole j \bigr) \leq 1$.
\begin{subequations}
\label{eq:UIproof1}
\begin{align}
\label{eq:UIproof1a} \E\bigl[ D_n^\msp \ind_{\{D_n^\msp \geq K\}} \bigr] 
&\leq N_n^{-1} \sum_{j\in\vertices(\comvect)} d_j^\msc \ind_{\{d_j^\msc \geq \sqrt{K}\}} \sum_{v\in [N_n]} \frac{d_v^\msl}{\he_n} \\
\label{eq:UIproof1b} &+ N_n^{-1} \sum_{j\in\vertices(\comvect)} d_j^\msc \ind_{\{d_j^\msc < \sqrt{K}\}} 
\sum_{v\in [N_n]} \P\bigl( d_v^\msp \geq K \bcond v \comrole j \bigr) \cdot \frac{d_v^\msl}{\he_n}. 
\end{align}
\end{subequations}
The term corresponding to large values of $d_j^\msc$, i.e., the rhs of \eqref{eq:UIproof1a} equals
\beq N_n^{-1} \sum_{v\in{N_n}} d_v^\msl \frac{1}{\he_n} \sum_{j\in\vertices(\comvect)} d_j^\msc \ind_{\{d_j^\msc \geq \sqrt{K}\}} 
= \E[D_n^\msl] \cdot \E[D_n^\msc \ind_{\{D_n^\msc \geq \sqrt{K}\}}], \eeq
where $\E[D_n^\msl]$ is bounded due to \cref{asmp:convergence} \cref{cond:mean_ldeg} and since $(D_n^\msc)_{n\in\N}$ is UI by assumption, $\E[D_n^\msc \ind_{\{D_n^\msc \geq \sqrt{K}\}}]$ can be made arbitrarily small uniformly in $n$ by choosing $K$ large enough. 
In the term corresponding to small values of $d_j^\msc$, i.e., \eqref{eq:UIproof1b}, we further analyze $\P\bigl( d_v^\msp \geq K \bcond v \comrole j \bigr)$. Note that conditionally on $v \comrole j$, $d_v^\msp = d_j^\msc + d_{J_2}^\msc + \ldots + d_{J_{\ldeg(v)}}^\msc$, where $J_i$ is chosen uar from $\vertices(\comvect) \setminus \{j,J_2,\ldots,J_{i-1}\}$. Using that $d_j^\msc < \sqrt{K}$,
\beq\bal \P\bigl( d_v^\msp \geq K \bcond v \comrole j \bigr)
&= \P\bigl( d_v^\msp - d_j^\msc = d_{J_2}^\msc + \ldots d_{J_{\ldeg(v)}}^\msc \geq K - d_j^\msc \bigr) \\
&\leq \P\bigl( d_{J_2}^\msc + \ldots + d_{J_{\ldeg(v)}}^\msc \geq K - \sqrt{K} \bigr) \\
&\leq \frac{\E\bigl[d_{J_2}^\msc\bigr] + \ldots + \E\bigl[d_{J_{\ldeg(v)}}^\msc\bigr]}{K-\sqrt{K}}, \eal \eeq
by Markov's inequality. For any fixed $v$, $\ldeg(v)$ is constant, thus the depletion of community vertices becomes negligible as $n\to\infty$, and $\E\bigl[d_{J_i}^\msc\bigr] = \E[D_n^\msc] (1+o(1))$ for all $i$. Substituting, 
and using that $d_j^\msc \ind_{\{d_j^\msc < \sqrt{K}\}} \leq \sqrt{K} \ind_{\{d_j^\msc < \sqrt{K}\}}  \leq \sqrt{K}$, 
we can bound \eqref{eq:UIproof1b} by
\beq\bal &N_n^{-1} \sum_{j\in\vertices(\comvect)} d_j^\msc \ind_{\{d_j^\msc < \sqrt{K}\}} 
\sum_{v\in [N_n]} \frac{\E[D_n^\msl - 1] \cdot \E[D_n^\msc] (1+o(1))}{K-\sqrt{K}} \cdot \frac{d_v^\msl}{\he_n} \\
&= \E[D_n^\msl - 1] \E[D_n^\msc] (1+o(1)) \frac{1}{N_n} \sum_{v\in [N_n]} d_v^\msl \frac{1}{\he_n} \sum_{j\in\vertices(\comvect)} d_j^\msc \ind_{\{d_j^\msc < \sqrt{K}\}} \frac{1}{K-\sqrt{K}} \\
&\leq \E[D_n^\msl - 1] \E[D_n^\msc] (1+o(1)) \E[D_n^\msl] \frac{\he_n \sqrt{K}}{\he_n (K-\sqrt{K})},
\eal\eeq
where $\E[D_n^\msl - 1] \E[D_n^\msc] (1+o(1)) \E[D_n^\msl]$ is bounded due to \cref{asmp:convergence} \cref{cond:mean_ldeg} and our assumption that $(D_n^\msc)_{n\in\N}$ is UI, and $\sqrt{K}/(K-\sqrt{K})$ can be made arbitrarily small for $K$ large enough. We conclude that by choosing $K$ sufficiently large, $\E\bigl[D_n^\msp \ind_{\{D_n^\msp \geq K\}} \bigr]$ can be made arbitrarily small, uniformly in $n$, that is, $(D_n^\msp)_{n\in\N}$ is uniformly integrable.
\end{proof}

\begin{proof}[Proof of \cref{lem:UI_equiv} \eqref{UI:deg} implies \eqref{UI:comdeg}]
Recall that for $v\in\leftpar$ and $j\in\vertices(\comvect)$, $v \comrole j$ denotes the event that $j$ is one of the community roles assigned to $v$. Note that each $j$ is assigned to a \emph{unique} $v$, thus $\sum_{v\in[N_n]} \ind_{\{v \comrole j\}} = 1$ (almost surely). We calculate, for some $K\in\N$,
\beq\bal \E\bigl[D_n^\msc \ind_{\{ D_n^\msc \geq K \}} \bigr] 
&= \frac{1}{\he_n} \sum_{j\in\vertices(\comvect)} d_j^\msc \ind_{\{ d_j^\msc \geq K \}} \\
&= \frac{1}{\he_n} \sum_{j\in\vertices(\comvect)} d_j^\msc \ind_{\{ d_j^\msc \geq K \}} \E\Bigl[ \sum_{v\in[N_n]} \ind_{\{v \comrole j\}} \Bigr] \\
&= \E\Bigl[ \frac{1}{\he_n} \sum_{v\in[N_n]} \sum_{j\in\vertices(\comvect)} d_j^\msc \ind_{\{v \comrole j\}} \ind_{\{ d_j^\msc \geq K \}} \Bigr].
\eal\eeq
We recognize the $\msp$-degree (see \eqref{eq:def_deg}) of $v$ written as $d_v^\msp = \sum_{j\in\vertices(\comvect)} d_j^\msc \ind_{\{v \comrole j\}}$, and note that on the event $v \comrole j$, $d_j^\msc \geq K$ implies that $d_v^\msp \geq K$. Thus,
\beq\bal &\E\bigl[ D_n^\msc \ind_{\{ D_n^\msc \geq K \}} \bigr] 
\leq \E\Bigl[ \frac{1}{\he_n} \sum_{v\in[N_n]} \sum_{j\in\vertices(\comvect)} d_j^\msc \ind_{\{v \comrole j\}} \ind_{\{ d_v^\msp \geq K \}} \Bigr] \\
&= \E\Bigl[ \frac{1}{\E[D_n^\msl] N_n} \sum_{v\in[N_n]} d_v^\msp \ind_{\{ d_v^\msp \geq K \}} \Bigr] 
= \frac{1}{\E[D_n^\msl]} \E\bigl[ D_n^\msp \ind_{\{ D_n^\msp \geq K \}} \bigr]. \eal\eeq
This can be made arbitrarily small, uniformly in $n$, by choosing $K$ large enough, since $\E[D_n^\msl]$ is bounded due to \cref{asmp:convergence} \cref{cond:mean_ldeg} (and bounded away from $0$), and $(D_n^\msp)_{n\in\N}$ is UI by assumption. This implies that $(D_n^\msc)_{n\in\N}$ is also UI.
\end{proof}

\begin{proof}[Proof of \cref{lem:UI_equiv} \eqref{UI:comdeg} implies \eqref{UI:edges}]
Note that with an $\msr$-vertex $V_n^\msr \distr \Unif[M_n]$ chosen uar, $\com_{V_n^\msr} \eqindis \rg_n$. Thus, for $K\in\N$, we compute
\beq \E\bigl[ \abs{\edges(\rg_n)} \ind_{\{ \abs{\edges(\rg_n)} \geq K \}} \bigr] 
= \frac{1}{M_n} \sum_{a\in[M_n]} \abs{\edges(\com_a)} \cdot \ind_{\{ \abs{\edges(\com_a)} \geq K \}}.
\eeq
For any graph $H$, $\abs{\edges(H)} = \frac12 \sum_{v\in\vertices(H)} \deg(v)$. Noting that for any graph $H$, $\abs{\edges(H)} \leq \abs{H} (\abs{H}-1)/2$, we have that $\abs{\edges(H)}\geq K$ implies $\abs{H} \geq \sqrt{2K}$, thus
\beq\begin{split} \label{eq:UIproof2}
&\E\bigl[ \abs{\edges(\rg_n)} \ind_{\{ \abs{\edges(\rg_n)} \geq K \}} \bigr] 
= \frac{1}{M_n} \sum_{a\in[M_n]} \frac12 \sum_{j \in \vertices(\com_a)} d_j^\msc \cdot \ind_{\{ \abs{\edges(\com_a)} \geq K \}} \\
&\leq \frac{1}{2 M_n} \sum_{a\in[M_n]} \sum_{j \in \vertices(\com_a)} d_j^\msc \cdot \ind_{\{ \abs{\com_a} \geq \sqrt{2K} \}} 
= \frac{1}{2 M_n} \sum_{j\in\vertices(\comvect)} d_j^\msc \cdot \ind_{\{ \abs{\com_{a(j)}} \geq \sqrt{2K} \}}, 
\end{split}\eeq
where $a(j)$ is the community that $j$ is a vertex in. To estimate \eqref{eq:UIproof2}, we first compute the number of terms, that is, the number of vertices in large enough communities, using that $\vertices(\comvect) = \cup_{a\in[M_n]} \vertices(\com_a)$:
\beq \bal \kappa_n :&= \sum_{j\in\vertices(\comvect)} \ind_{\{ \abs{\com_{a(j)}} \geq \sqrt{2K} \}} 
= \sum_{a\in[M_n]} \abs{\com_a} \cdot \ind_{\{ \abs{\com_a} \geq \sqrt{2K} \}} \\
&= M_n \E\bigl[ D_n^\msr \ind_{\{ D_n^\msr \geq \sqrt{2K} \}} \bigr]. \eal\eeq
By \cref{asmp:convergence} \cref{cond:mean_rdeg}, $(D_n^\msr)_{n\in\N}$ is UI, thus $\E\bigl[ D_n^\msr \ind_{\{ D_n^\msr \geq \sqrt{2K} \}} \bigr]$ can be made arbitrarily small uniformly in $n$ by choosing $K$ large enough. That is, we can make $\kappa_n \leq \delta' M_n \leq \delta \he_n$, since $\he_n/M_n = \E[D_n^\msr]$ is bounded. 
Let $d_{(i)}^\msc$ denote the $i$th largest element in the sequence $(d_j^\msc)_{j\in\vertices(\comvect)}$. 
Since there are at most $\delta \he_n$ community roles in communities larger than $\sqrt{2K}$, we can bound the sum of their $\msc$-degrees on the rhs of \eqref{eq:UIproof2} by taking (at least\footnote{If there are several terms equal to the one with rank $\delta\he_n$, we may include more terms.}) $\delta \he_n$ of the largest $\msc$-degrees:
\beq \E\bigl[ \abs{\edges(\rg_n)} \ind_{\{ \abs{\edges(\rg_n)} \geq K \}} \bigr] 
\leq \frac{1}{2 M_n} \sum_{j \in \vertices(\comvect)} d_j^\msc \cdot \ind_{\{d_j^\msc \geq d_{(\delta \he_n)}^\msc\}}. \eeq
Note that $d_{(\delta \he_n)}^\msc$ the upper-$\delta$-quantile $Q_n^\delta$ of the empirical distribution $D_n^\msc$, thus it must converge to the upper-$\delta$-quantile $Q^\delta$ of the distribution $D^\msc$. Consequently, there exists $n_0$ such that for $n \geq n_0$, $Q_n^\delta \geq Q^{\delta} - \eps =: K'$, and
\beq\bal \E\bigl[ \abs{\edges(\rg_n)} \ind_{\{ \abs{\edges(\rg_n)} \geq K \}} \bigr] 
&\leq \frac{\E[D_n^\msr]}{2 \he_n} \sum_{j\in\vertices(\comvect)} d_j^\msc \ind_{\{ d_j^\msc \geq Q^\delta - \eps \}} \\
&= \frac12\E[D_n^\msr] \E\bigl[ D_n^\msc \ind_{\{ D_n^\msc \geq K' \}}\bigr]. \eal\eeq
Since $\E[D_n^\msr]$ is bounded due to \cref{asmp:convergence} \cref{cond:mean_rdeg} and $(D_n^\msc)_{n\in\N}$ is UI, the upper bound can be made arbitrarily small for $n \geq n_0$ by choosing $K'$ sufficiently large. This is possible by choosing the earlier $K$ sufficiently large so that $\delta$ is sufficiently small. Further, since $n < n_0$ is a finite set, we can increase $K$ so that $\max_{n < n_0} \E\bigl[ \abs{\edges(\rg_n)} \ind_{\{ \abs{\edges(\rg_n)} \geq K \}} \bigr]$ is sufficiently small as well. We conclude that indeed, $(\abs{\edges(\rg_n)})_{n\in\N}$ is UI.
\end{proof}

\begin{proof}[Proof of \cref{lem:UI_equiv} \eqref{UI:edges} implies \eqref{UI:comdeg}]
Recall that $\vertices(\comvect)$ denotes the union of vertices in all $\com_a \in \comvect$. We calculate, for some $K\in\N$,
\beq \E\bigl[D_n^\msc \ind_{\{ D_n^\msc \geq K \}} \bigr] 
= \frac{1}{\he_n} \sum_{j\in\vertices(\comvect)} d_j^\msc \ind_{\{ d_j^\msc \geq K \}} 
= \frac{1}{\he_n} \sum_{a\in[M_n]} \sum_{j\in\vertices(\com_a)} d_j^\msc \ind_{\{ d_j^\msc \geq K \}}. \eeq
Note that $\abs{\edges(\com_a)} = 2 \sum_{j\in\vertices(\com_a)} d_j^\msc$, and further, that for $j\in\vertices(\com_a)$, $d_j^\msc \geq K$ implies $\abs{\edges(\com_a)} \geq K$. Thus
\beq\bal &\E\bigl[D_n^\msc \ind_{\{ D_n^\msc \geq K \}} \bigr] 
\leq \frac{1}{\he_n} \sum_{a\in[M_n]} \sum_{j\in\vertices(\com_a)} d_j^\msc \cdot \ind_{\{ \abs{\edges(\com_a)} \geq K \}} \\
&= \frac{1}{\E[D_n^\msr] M_n} \sum_{a\in[M_n]} 2 {\abs{\edges(\com_a)}} \cdot \ind_{\{ {\abs{\edges(\com_a)}} \geq K \}} 
= \frac{2}{\E[D_n^\msr]} \E\bigl[ {\abs{\edges(\rg_n)}} \ind_{\{ {\abs{\edges(\rg_n)}} \geq K \}} \bigr]. \eal\eeq
This can be made arbitrarily small, uniformly in $n$, by choosing $K$ large enough, since $\E[D_n^\msr]$ is bounded due to \cref{asmp:convergence} \cref{cond:mean_rdeg} (and bounded away from $0$), and $\abs{\edges(\rg_n)}$ is UI. This implies that $(D_n^\msc)_{n\in\N}$ is also UI.
\end{proof}

\subsection{Proof of Theorem \ref{thm:edgesRIGC}} 
\label{apxs:edges_giant}
Recall from the sketch proof that in order to complete the proof of \cref{thm:edgesRIGC}, we need to prove \eqref{eq:edgesconv_Akd} and \eqref{eq:equiv_edgesformulas}, that we complete below one by one.

\begin{proof}[Proof of \eqref{eq:edgesconv_Akd}]
Recall \eqref{eq:def_deg} and $D_n^\msp = \pdeg(V_n^\msl)$, with $V_n^\msl\sim\Unif[N_n]$. Note that $D_n^\msp$ has double randomness: the choice of $V_n^\msl$ and the bipartite matching $\bmatch_n$ that determines $\pdeg(v)$ (see \eqref{eq:def_deg}) for each $\msl$-vertex $v$. 
We rewrite 
\beq \frac{\abs{\edges(\comp_1)}}{N_n }
= \frac12 \E\bigl[ \pdeg(V_n^\msl) \ind_{\{V_n^\msl\in\comp_1\}} \bcond \bmatch_n \bigr] 
= \frac{1}{2N_n} \sum_{v\in[N_n]} \pdeg(v) \ind_{\{v\in\comp_1\}}. \eeq
Under the condition \cref{cond:UI_Dc}, by \cref{lem:UI_equiv}, we have that $(D_n^\msp)_{n\in\N}$ is UI. Noting that \allowbreak $\pdeg(V_n^\msl) \ind_{\{V_n^\msl\in\comp_1\}} \cond \bmatch_n \leq \pdeg(V_n^\msl) \cond \bmatch_n$, 
\beq \E\bigl[ \pdeg(V_n^\msl) \ind_{\{V_n^\msl\in\comp_1\}} \ind_{\{\pdeg(V_n^\msl) > K\}} \bcond \bmatch_n \bigr] 
\leq \E\bigl[ \pdeg(V_n^\msl) \ind_{\{\pdeg(V_n^\msl) > K\}} \bcond \bmatch_n \bigr],  \eeq
which implies that also $(\pdeg(V_n^\msl) \ind_{\{V_n^\msl\in\comp_1\}})_{n\in\N}$ is UI. Recall that by \eqref{eq:deggiant_distr}, the empirical mass function of $\pdeg(V_n^\msl) \ind_{\{V_n^\msl\in\comp_1\}}$ converges in probability, which combined with the uniform integrability implies the convergence of the empirical average to the mean of the limit: 
\beq \bal \frac{\abs{\edges(\comp_1)}}{N_n}
&= \frac12 \E\bigl[ \pdeg(V_n^\msl) \ind_{\{V_n^\msl \in \comp_1\}} \bcond \bmatch_n \bigr] \\
&= \frac12 \sum_{d\in\N} d\cdot \P\bigl( \pdeg(V_n^\msl) = d,\, V_n^\msl \in \comp_1 \bcond \bmatch_n \bigr) 
= \frac12 \sum_{d\in\N} d\cdot \frac{\abs{\vertices_d^\msp \cap \comp_1}}{N_n} \\
&= \frac12 \sum_{d\in\N} d\cdot \sum_{k\in\Z^+} \frac{\abs{\vertices_k^\msl \cap \vertices_d^\msp \cap \comp_1}}{N_n} 
\toinp \frac12 \sum_{d\in\N} \sum_{k\in\Z^+} d\cdot A(k,d), \eal \eeq
by \eqref{eq:deggiant_distr}. This concludes the proof of \eqref{eq:edgesconv_Akd}.
\end{proof}

\begin{proof}[Proof of \eqref{eq:equiv_edgesformulas}]
Recall that $\rg$ denotes a random graph with pmf $\bm\mu$. Note that under condition \cref{cond:UI_Dc}, by \cref{lem:UI_equiv}, $\E\bigl[ \abs{\edges(\rg)} (1-\eta_\msr^{\abs{\rg}-1}) \bigr] \leq \E\bigl[ \abs{\edges(\rg)} \bigr] <\infty$. 
Substituting \eqref{eq:def_deggiant_distr}, and noting that under the condition $d=c_1+\ldots+c_k$, we can replace $d$ by this sum, we compute
\begin{subequations}
\begin{align}
\nonumber &\sum_{d\in\N} \sum_{k\in\Z^+} d \cdot A(k,d) \\
\nonumber &= \sum_{d\in\N} \sum_{k\in\Z^+} p_k 
\sum_{H_1,\ldots,H_k\in\comgraphs} \sum_{\substack{c_1,\ldots,c_k\in\N\\ c_1+\ldots+c_k=d}} 
\Bigl( \sum_{j=1}^k c_j \Bigr)
\Bigl( 1 - \eta_\msr^{ \sum_{i=1}^k (\abs{H_i}-1)} \Bigr)
\prod_{i=1}^k \frac{\nu(c_i \countin H_i) \mu_{H_i}}{\E[D^\msr]} \\
\label{eq:loc_glob_edges1} &= \sum_{k\in\Z^+} p_k 
\sum_{H_1,\ldots,H_k\in\comgraphs} \sum_{c_1,\ldots,c_k\in\N} 
\sum_{j=1}^k c_j 
\prod_{i=1}^k \frac{\nu(c_i \countin H_i) \mu_{H_i}}{\E[D^\msr]} \\
\label{eq:loc_glob_edges2} &\phantom{{}={}}- \sum_{k\in\Z^+} p_k 
\sum_{H_1,\ldots,H_k\in\comgraphs} \sum_{c_1,\ldots,c_k\in\N} 
\sum_{j=1}^k c_j 
\cdot \eta_\msr^{ \sum_{i=1}^k (\abs{H_i}-1)} 
\prod_{i=1}^k \frac{\nu(c_i \countin H_i) \mu_{H_i}}{\E[D^\msr]},
\end{align}
\end{subequations}
where we have used that summing over $d$ removes the restriction $c_1+\ldots+c_k = d$. In the following, we simplify \eqref{eq:loc_glob_edges1} and \eqref{eq:loc_glob_edges2} separately, starting with \eqref{eq:loc_glob_edges1}.

For fixed $k$ and $H_1,\ldots,H_k\in\comgraphs$, we compute
\beq\label{eq:loc_glob_edges3} \bal 
&\sum_{c_1,\ldots,c_k\in\N} 
\sum_{j=1}^k c_j 
\prod_{i=1}^k \frac{\nu(c_i \countin H_i) \mu_{H_i}}{\E[D^\msr]} 
= \! \sum_{c_1,\ldots,c_k\in\N} 
\sum_{j=1}^k \frac{c_j \cdot \nu(c_j \countin H_j) \mu_{H_j}}{\E[D^\msr]} 
\prod_{\substack{i\in[k]\\i\neq j}} \frac{\nu(c_i \countin H_i) \mu_{H_i}}{\E[D^\msr]} \\
&= \sum_{j=1}^k \frac{\sum_{c_j\in\N} c_j \cdot \nu(c_j \countin H_j) \mu_{H_j}}{\E[D^\msr]} 
\prod_{\substack{i\in[k]\\i\neq j}} \frac{\sum_{c_i\in\N} \nu(c_i \countin H_i) \mu_{H_i}}{\E[D^\msr]}. \eal \eeq
Recall that $\nu(c_i \countin H_i)$ counts the number of $v\in\vertices(H_i)$ with degree $c_i$. Thus
\beq\label{eq:loc_glob_edges4} \sum_{c_i\in\N} \nu(c_i \countin H_i) = \abs{H_i},
\qquad \sum_{c_j\in\N} c_j \cdot \nu(c_j \countin H_j) 
= \sum_{v\in\vertices(H_j)} \deg(v)
= 2 \abs{\edges(H_j)}. \eeq
Again, denote by $\rg$ a random graph with pmf $\bm\mu$. Combining (\ref{eq:loc_glob_edges3}-\ref{eq:loc_glob_edges4}) and substituting, \eqref{eq:loc_glob_edges1} equals
\beq \bal &\sum_{k\in\Z^+} p_k \sum_{H_1,\ldots,H_k\in\comgraphs} 
\sum_{j=1}^k \frac{2\abs{\edges(H_j)} \mu_{H_j}}{\E[D^\msr]} 
\prod_{\substack{i\in[k]\\i\neq j}} \frac{\abs{H_i} \mu_{H_i}}{\E[D^\msr]} \\
&= \sum_{k\in\Z^+} p_k \sum_{j=1}^k \frac{\sum_{H_j\in\comgraphs} 2\abs{\edges(H_j)} \mu_{H_j}}{\E[D^\msr]} 
\prod_{\substack{i\in[k]\\i\neq j}} \frac{\sum_{H_i\in\comgraphs} \abs{H_i} \mu_{H_i}}{\E[D^\msr]} \\
&= \sum_{k\in\Z^+} p_k \sum_{j=1}^k \frac{2\E\bigl[ \abs{\edges(\rg)} \bigr]}{\E[D^\msr]} 
\prod_{\substack{i\in[k]\\i\neq j}} \frac{\E\bigl[ \abs{\rg} \bigr]}{\E[D^\msr]} 
= \sum_{k\in\Z^+} p_k \cdot k \cdot \frac{2\E\bigl[ \abs{\edges(\rg)} \bigr]}{\E[D^\msr]},
\eal \eeq
since $\E\bigl[ \abs{\rg} \bigr] / \E[D^\msr] = 1$, so the product over $i$ disappears, which yields $k$ identical terms in the sum over 
$j$. Note that $\sum_{k\in\Z^+} p_k \cdot k = \E[D^\msl]$ and recall from \cref{rem:asmp_consequences} \cref{cond:partition_ratio} that $\E[D^\msl]/\E[D^\msr] = \gamma$. We conclude that \eqref{eq:loc_glob_edges1} equals
\beq\label{eq:loc_glob_edges5} \frac{2\E\bigl[ \abs{\edges(\rg)} \bigr]}{\E[D^\msr]} \E[D^\msl] 
= 2 \gamma\, \E\bigl[ \abs{\edges(\rg)} \bigr]. \eeq
Next, we study \eqref{eq:loc_glob_edges2}. Using (\ref{eq:loc_glob_edges3}-\ref{eq:loc_glob_edges4}) and substituting,
\beq\label{eq:loc_glob_edges6} \bal &
\sum_{k\in\Z^+} p_k \sum_{H_1,\ldots,H_k\in\comgraphs} \eta_\msr^{\sum_{i=1}^k (\abs{H_i}-1)} 
\sum_{j=1}^k \frac{2\abs{\edges(H_j)}\mu_{H_j}}{\E[D^\msr]}
\prod_{\substack{i\in[k]\\i\neq j}} \frac{\abs{H_i} \mu_{H_i}}{\E[D^\msr]} \\
&= \sum_{k\in\Z^+} p_k \sum_{H_1,\ldots,H_k\in\comgraphs} \sum_{j=1}^k \frac{2\abs{\edges(H_j)}\mu_{H_j} \cdot \eta_\msr^{\abs{H_j}-1}}{\E[D^\msr]} 
\prod_{\substack{i\in[k]\\i\neq j}} \frac{\abs{H_i} \mu_{H_i} \cdot \eta_\msr^{\abs{H_i}-1}}{\E[D^\msr]} \\
&= \sum_{k\in\Z^+} p_k \sum_{j=1}^k
\frac{\sum_{H_j\in\comgraphs} 2\abs{\edges(H_j)}\mu_{H_j} \cdot \eta_\msr^{\abs{H_j}-1}}{\E[D^\msr]} 
\prod_{\substack{i\in[k]\\i\neq j}} \frac{\sum_{H_i\in\comgraphs} \abs{H_i} \mu_{H_i} \cdot \eta_\msr^{\abs{H_i}-1}}{\E[D^\msr]}.
\eal \eeq
Recall \eqref{eq:def_sizebiasing}, \eqref{eq:def_genfunc}, and $\bitq$ from \cref{asmp:convergence} \cref{cond:limit_rdeg}, and note that
\beq\label{eq:loc_glob_edges7} \bal \frac{\sum_{H_i\in\comgraphs} \abs{H_i} \mu_{H_i} \cdot \eta_\msr^{\abs{H_i}-1}}{\E[D^\msr]}
&= \sum_{m\in\Z^+} \frac{m \cdot q_m}{\E[D^\msr]} \eta_\msr^{m-1} \\
&= \sum_{m\in\Z^+} \P\bigl( \wit D^\msr = m-1 \bigr) \eta_\msr^{m-1} 
= G_{\wit D^\msr} (\eta_\msr) = \eta_\msl,
\eal \eeq
as we have shown in \cref{ss:BPapprox}. 
Consequently in \eqref{eq:loc_glob_edges6}, the last product over $i$  equals $\eta_\msl^{k-1}$, and in the factor before, we recognize an expectation wrt the random graph $\rg$ with pmf $\bm\mu$. 
Recall $\bitp$ from \cref{asmp:convergence} \cref{cond:limit_ldeg}. Combining (\ref{eq:loc_glob_edges6}-\ref{eq:loc_glob_edges7}), we have that \eqref{eq:loc_glob_edges2} without the minus sign equals
\beq\label{eq:loc_glob_edges8} \bal &\sum_{k\in\Z^+} p_k \sum_{j=1}^k \frac{2\E\bigl[ \abs{\edges(\rg)} \eta_\msr^{\abs{\rg}-1} \bigr]}{\E[D^\msr]} \eta_\msl^{k-1} 
= \frac{2\E\bigl[ \abs{\edges(\rg)} \eta_\msr^{\abs{\rg}-1} \bigr]}{\E[D^\msr]} \sum_{k\in\Z^+} p_k \cdot k \cdot \eta_\msl^{k-1} \\
&= 2\E\bigl[ {\abs{\edges(\rg)}} \eta_\msr^{{\abs{\rg}-1}} \bigr] 
\frac{\E[D^\msl]}{\E[D^\msr]}
\sum_{k\in\Z^+} \frac{k \cdot p_k}{\E[D^\msl]} \eta_\msl^{k-1} 
= 2\gamma\, \E\bigl[ {\abs{\edges(\rg)}} \eta_\msr^{{\abs{\rg}-1}} \bigr] G_{\wit D^\msl}(\eta_\msl) \\
&= 2\gamma\, \E\bigl[ {\abs{\edges(\rg)}} \eta_\msr^{{\abs{\rg}-1}} \bigr] \eta_\msr 
= 2\gamma\, \E\bigl[ {\abs{\edges(\rg)}} \eta_\msr^{{\abs{\rg}}} \bigr], \eal \eeq
where we have used that we have defined $\eta_\msr = G_{\wit D^\msl}(\eta_\msl)$ in \cref{thm:degrees_giant}. Combining \eqref{eq:loc_glob_edges5} and \eqref{eq:loc_glob_edges8}, we obtain that
\beq \sum_{k\in\Z^+} \sum_{d\in\N} d\cdot A(k,d) 
= 2\gamma\, \E\bigl[ {\abs{\edges(\rg)}} \bigr] 
- 2\gamma\, \E\bigl[ {\abs{\edges(\rg)}} \eta_\msr^{{\abs{\rg}}} \bigr], 
\eeq
which is equivalent to \eqref{eq:equiv_edgesformulas}.
\end{proof}

\section{Death processes and their hitting times}
\label{apx:concentration_deathprocesses}
In this section, we prove \cref{claim:concentration}. We only show that statement \eqref{stm:t} implies statement \eqref{stm:c}, the proof of the reverse implication is analogous.

\begin{proof}[Proof: statement \eqref{stm:t} implies statement \eqref{stm:c}]
For $\delta>0, T<\infty$ and $\eps>0, c_0>0$, we define the events
\begin{gather} \nth\goodevent_1(\delta,T) := \bigl\{ \sup\nolimits_{t\leq T}\, \bigl\lvert a_n^{-1}\nth X(t) - f(t)\bigr\rvert < \delta \bigr\}, \\
\nth\goodevent_2(\eps,c_0) := \bigl\{ \sup\nolimits_{c\geq c_0}\, \bigl\lvert \nth\CT(c) - \inv f(c)\bigr\rvert < \eps \bigr\}.
\end{gather}
Statement \eqref{stm:t} implies that for any $\delta > 0$ and $T<\infty$ fixed,
\beq \P\bigl(\nth\goodevent_1(\delta,T)\bigr) \to 1. \eeq
Note that statement \eqref{stm:c} is equivalent to $\nth\goodevent_2(\eps,c_0)$ happening whp for any $\eps>0, c_0>0$. We prove this by finding a convenient correspondence of the parameters such that the (known to be whp) event $\nth\goodevent_1(\delta,T)$ implies $\nth\goodevent_2(\eps,c_0)$. We fix $\eps>0$ and $c_0>0$ and denote the random function $\nth x(t) := a_n^{-1}\nth X(t)$. For some $\delta$ and $T$ yet to be chosen, on the event $\goodevent_1(\delta,T)$, for each $t\leq T$,
\beq f(t)-\delta < \nth x(t) < f(t)+\delta. \eeq
Let $t_1 := \inv f (c+\delta)$ and $t_2 := \inv f (c-\delta)$. We choose $T := \inv f(c_0-\delta)$, so that $t_1, t_2 \leq T$ for any choice of $c\geq c_0$. Then, on the event $\goodevent_1(\delta,T)$,
\begin{subequations}
\begin{gather} \nth x(t_1) > f(t_1)-\delta = c, \\
\nth x(t_2) < f(t_2) + \delta = c.
\end{gather}
\end{subequations}
Recall that $\nth\CT(c) = \inf\{ t:\nth x(t)\leq c \}$. Since $\nth x(t)$ is non-increasing, necessarily
\beq\label{eq:taubound} \inv f (c+\delta) = t_1 \leq \nth\CT(c) \leq t_2 = \inv f (c-\delta). \eeq
Since $\inv f$ is continuous on $(0,1]$, it is uniformly continuous on $[c_0/2,1]$. Hence for our fixed $\eps$, we can choose $0<\delta=\delta(\eps)<c_0/2$ such that that for any $s_1, s_2\in[c_0/2,1]$, if $\abs{s_1-s_2}<\delta$, then $\abs{\inv f (s_1) - \inv f(s_2)} < \eps$. Hence for this choice of $\delta$,
\begin{subequations}\label{eq:unifcont}
\begin{gather} \inv f (c-\delta) < \inv f(c) + \eps, \\
\inv f (c+\delta) > \inv f(c) - \eps.
\end{gather}
\end{subequations}
Combining \eqref{eq:taubound} and \eqref{eq:unifcont}, we obtain that on the event $\goodevent_1(\delta,T)$, $\abs{\nth\CT(c) - \inv f(c)} < \eps$ holds uniformly in $c\in[c_0,1]$. 
That is, with the chosen $T$ and $\delta$, $\nth\goodevent_1(\delta,T)$ implies $\nth\goodevent_2(\eps,c_0)$ for the given $\eps, c_0$. Then
\beq 1 \geq \P\bigl( \nth\goodevent_2(\eps,c_0) \bigr) \geq
\P\bigl( \nth\goodevent_1(\delta,T) \bigr) \to 1 \eeq
as $n\to\infty$. We conclude that statement \eqref{stm:t} implies statement \eqref{stm:c}.
\end{proof}

\section*{Acknowledgements}
This work is supported by the Netherlands Organisation for Scientific Research (NWO) through VICI grant 639.033.806 (RvdH), VENI grant 639.031.447 (JK), the Gravitation {\sc Networks} grant 024.002.003 (RvdH), and TOP grant 613.001.451 (VV). 

\printbibliography
%\bibliography{../bibl}
%\bibliographystyle{abbrv}

\end{document}